\documentclass[12pt]{amsart}

\usepackage{amsmath}
\usepackage{amssymb}
\usepackage{ascmac}
\usepackage{amsthm}

\makeatletter

\@addtoreset{equation}{section}
\makeatother 

\theoremstyle{plain}
\newtheorem{thm}{Theorem}[section]
\newtheorem*{thm*}{Theorem}
\newtheorem{prop}[thm]{Proposition}
\newtheorem{lem}[thm]{Lemma}
\newtheorem{cor}[thm]{Corollary}

\theoremstyle{definition}

\theoremstyle{remark}
\newtheorem{rem}{Remark}[section]

\newcommand{\vol}{\operatorname{vol}}
\newcommand{\ric}{\operatorname{Ric}}
\newcommand{\Div}{\operatorname{div}}
\newcommand{\Hess}{\operatorname{Hess}}
\newcommand{\bm}{\partial M}
\newcommand{\ball}{B^{n}_{\kappa,\lambda}}
\newcommand{\const}{C_{\kappa,\lambda}}
\newcommand{\bconst}{\bar{C}_{\kappa,\lambda}}
\newcommand{\bball}{\partial B^{n}_{\kappa,\lambda}}
\newcommand{\sk}{s_{\kappa}}
\newcommand{\ck}{c_{\kappa}}
\newcommand{\dm}{D(M,\partial M)}
\newcommand{\tr}{\operatorname{trace}}
\newcommand{\cut}{\mathrm{Cut}\,}

\newcommand{\expp}{\exp^{\perp}}
\newcommand{\inte}{\mathrm{Int}\,}
\newcommand{\tbp}{T^{\perp}\bm}

\newcommand{\thetaex}{\theta}

\usepackage{latexsym}

\title[rigidity of manifolds with boundary]{rigidity of manifolds with boundary under a lower Bakry-\'Emery\\ Ricci curvature bound}
\author{Yohei Sakurai}
\date{September 21, 2016}
\address{Graduate School of Pure and Applied Sciences, University of Tsukuba, Tennodai 1-1-1, Tsukuba, Ibaraki, 305-8577, Japan}
\email{sakurai@math.tsukuba.ac.jp}
\thanks{Research Fellow of Japan Society for the Promotion of Science for 2014-2016}
\subjclass[2010]{Primary 53C20}
\keywords{Manifold with boundary; Bakry-\'Emery Ricci curvature}
\begin{document}
\maketitle

\begin{abstract}
We study Riemannian manifolds with boundary under a lower Bakry-\'Emery Ricci curvature bound.
In our weighted setting,
we prove several rigidity theorems for such manifolds with boundary. 
We conclude a rigidity theorem for the inscribed radii,
a volume growth rigidity theorem for the metric neighborhoods of the boundaries,
and various splitting theorems.
We also obtain rigidity theorems for the smallest Dirichlet eigenvalues for the weighted $p$-Laplacians.
\end{abstract}

\section{Introduction}
For Riemannian manifolds without boundary,
under a lower Bakry-\'Emery Ricci curvature bound,
we know several comparison results and rigidity theorems (see e.g., \cite{FLZ}, \cite{Lo}, \cite{Q} and \cite{WW}).
For metric measure spaces,
Lott and Villani \cite{LV1}, \cite{LV2},
and Sturm \cite{St2}, \cite{St3} have introduced the so-called curvature dimension condition
that is equivalent to a lower Bakry-\'Emery Ricci curvature bound for manifolds without boundary.
Under a curvature dimension condition,
they have obtained comparison results in \cite{LV2} and \cite{St2}.
Under a more restricted condition,
Gigli \cite{G},
and Ketterer \cite{Ke1}, \cite{Ke2} have recently studied rigidity theorems.

In this paper,
we study Riemannian manifolds with boundary
under a lower Bakry-\'Emery Ricci curvature bound,
and under a lower mean curvature bound for the boundary.
For such manifolds with boundary,
we obtain several comparison results,
and we prove rigidity theorems.
In an unweighted standard setting,
for instance,
Heintze and Karcher \cite{HK},
and Kasue \cite{K2} have obtained comparison results,
and Kasue \cite{K3}, \cite{K4},
and the author \cite{Sa} have done rigidity theorems.
We generalize them in our weighted setting.
\subsection{Setting}
For $n\geq 2$,
let $M$ be an $n$-dimensional, 
connected complete Riemannian manifold with boundary with Riemannian metric $g$.
The boundary $\bm$ is assumed to be smooth.
We denote by $d_{M}$ the Riemannian distance on $M$ induced from the length structure determined by $g$.
Let $f:M\to \mathbb{R}$ be a smooth function.
For the Riemannian volume measure $\vol_{g}$ on $M$ induced from $g$,
we put $m_{f}:=e^{-f}\, \vol_{g}$.

We denote by $\ric_{g}$ the Ricci curvature on $M$ defined by $g$.
We denote by $\nabla f$ the \textit{gradient} of $f$,
and by $\Hess f$ the \textit{Hessian} of $f$.
For $N\in (-\infty,\infty]$,
the \textit{Bakry-\'Emery Ricci curvature} $\ric^{N}_{f}$ is defined as follows (\cite{BE}, \cite{Q}):
If $N \in (-\infty,\infty)\setminus\{n\}$,
then
\begin{equation*}
\ric^{N}_{f}:=\ric_{g}+\Hess f-\frac{\nabla f \otimes \nabla f}{N-n};
\end{equation*}
if $N=\infty$,
then $\ric^{N}_{f}:=\ric_{g}+\Hess f$;
if $N=n$,
and if $f$ is a constant function,
then $\ric^{N}_{f}:=\ric_{g}$;
if $N=n$,
and if $f$ is not constant,
then put $\ric^{N}_{f}:=-\infty$.
For $K \in \mathbb{R}$, 
by $\ric^{N}_{f,M}\geq K$
we mean that the infimum of $\ric^{N}_{f}$ on the unit tangent bundle on the interior $\inte M$ of $M$ is at least $K$.
For $x\in \bm$,
we denote by $H_{x}$ the mean curvature on $\bm$ at $x$ in $M$ defined as the trace of the shape operator for the unit inner normal vector $u_{x}$ at $x$.
The \textit{$f$-mean curvature} $H_{f,x}$ at $x$ is defined by
\begin{equation*}
H_{f,x}:=H_{x}+g((\nabla f)_{x},u_{x}).
\end{equation*}
For $\Lambda \in \mathbb{R}$, 
by $H_{f,\bm}\geq \Lambda$
we mean $\inf_{x\in \bm} H_{f,x}\geq \Lambda$.
The subject of our study is a metric measure space $(M,d_{M},m_{f})$ such that
for $N\in [n,\infty)$,
and for $\kappa,\lambda \in \mathbb{R}$,
we have $\ric^{N}_{f,M}\geq (N-1)\kappa$ and $H_{f,\bm}\geq (N-1)\lambda$,
or such that $\ric^{\infty}_{f,M}\geq 0$ and $H_{f,\bm}\geq 0$.

\subsection{Inscribed radius rigidity}
For $\kappa \in \mathbb{R}$,
we denote by $M^{n}_{\kappa}$ the $n$-dimensional space form with constant curvature $\kappa$.
We say that $\kappa\in \mathbb{R}$ and $\lambda \in \mathbb{R}$ satisfy the \textit{ball-condition}
if there exists a closed geodesic ball $\ball$ in $M^{n}_{\kappa}$ with non-empty boundary $\bball$
such that $\bball$ has a constant mean curvature $(n-1)\lambda$.
We denote by $\const$ the radius of $\ball$.
We see that $\kappa$ and $\lambda$ satisfy the ball-condition if and only if either
(1) $\kappa>0$; 
(2) $\kappa=0$ and $\lambda>0$;
or (3) $\kappa<0$ and $\lambda>\sqrt{\vert \kappa \vert}$.
Let $s_{\kappa,\lambda}(t)$ be a unique solution of the so-called Jacobi-equation
\begin{equation*}
\phi''(t)+\kappa \phi(t)=0
\end{equation*}
with initial conditions $\phi(0)=1$ and $\phi'(0)=-\lambda$.
We see that $\kappa$ and $\lambda$ satisfy the ball-condition if and only if
the equation $s_{\kappa,\lambda}(t)=0$ has a positive solution;
in particular,
$\const=\inf \{t>0 \mid s_{\kappa,\lambda}(t)=0\}$.

Let $\rho_{\bm}:M\to \mathbb{R}$ be the distance function from $\bm$ defined as $\rho_{\bm}(p):=d_{M}(p,\bm)$.
The \textit{inscribed radius of} $M$ is defined as
\begin{equation*}
\dm:=\sup _{p\in M}\rho_{\bm}(p).
\end{equation*}

We have the following rigidity theorem for the inscribed radius:
\begin{thm}\label{thm:Ball rigidity}
Let $M$ be an $n$-dimensional, 
connected complete Riemannian manifold with boundary,
and let $f$ be a smooth function on $M$.
Let $\kappa \in \mathbb{R}$ and $\lambda \in \mathbb{R}$ satisfy the ball-condition.
For $N\in [n,\infty)$,
we suppose $\ric^{N}_{f,M}\geq (N-1)\kappa$ and $H_{f,\bm}\geq (N-1)\lambda$.
Then we have $D(M,\bm)\leq \const$.
Moreover,
if there exists $p\in M$ such that $\rho_{\bm}(p)=\const$,
then $(M,d_{M})$ is isometric to $(\ball,d_{\ball})$ and $N=n$;
in particular,
$f$ is constant on $M$.
\end{thm}
Kasue \cite{K3} has proved Theorem \ref{thm:Ball rigidity}
in the standard case where $f=0$ and $N=n$.
We prove Theorem \ref{thm:Ball rigidity} in a similar way to that in \cite{K3}.
\begin{rem}
M. Li \cite{L} later than \cite{K3} has proved Theorem \ref{thm:Ball rigidity} when $f=0,\, N=n$ and $\kappa=0$.
H. Li and Wei have proved Theorem \ref{thm:Ball rigidity} in \cite{LW2} when $\kappa=0$,
and in \cite{LW1} when $\kappa<0$.
In \cite{LW1} and \cite{LW2},
Theorem \ref{thm:Ball rigidity} in the specific cases have been proved in a similar way to that in \cite{L}.
\end{rem}

\subsection{Volume growth rigidity}
For $\kappa\in \mathbb{R}$ and $\lambda\in \mathbb{R}$,
if $\kappa$ and $\lambda$ satisfy the ball-condition,
then we put $\bar{C}_{\kappa,\lambda}:=\const$;
otherwise,
$\bar{C}_{\kappa,\lambda}:=\infty$.
We define a function $\bar{s}_{\kappa,\lambda}:[0,\infty)\to \mathbb{R}$ by
\[
\bar{s}_{\kappa,\lambda}(t) := \begin{cases}
                                            s_{\kappa,\lambda}(t) & \text{if $t< \bar{C}_{\kappa,\lambda}$},\\
                                            0                     & \text{if $t\geq    \bar{C}_{\kappa,\lambda}$}.
                                 \end{cases}
\]
For $N\in [2,\infty)$,
we define a function $s_{N,\kappa,\lambda}:(0,\infty)\to \mathbb{R}$ by
\begin{equation*}
s_{N,\kappa,\lambda}(r):=\int^{r}_{0}\, \bar{s}^{N-1}_{\kappa,\lambda}(t)\,dt.
\end{equation*}

For $r>0$,
we put $B_{r}(\bm):=\{\,p\in M \mid \rho_{\bm}(p) \leq r\,\}$.
For $x\in \bm$,
let $\gamma_{x}:[0,T)\to M$ be the geodesic with initial conditions $\gamma_{x}(0)=x$ and $\gamma_{x}'(0)=u_{x}$.
We denote by $h$ the induced Riemnnian metric on $\bm$.
For the Riemannian volume measure $\vol_{h}$ on $\bm$ induced from $h$,
we put $m_{f,\bm}:=e^{-f|_{\bm}}\, \vol_{h}$.
For an interval $I$,
and for a connected component $\bm_{1}$ of $\bm$,
let $I \times_{\kappa,\lambda} \bm_{1}$ denote the warped product $(I \times \bm_{1}, dt^{2}+s^{2}_{\kappa,\lambda}(t)h)$.
We put $I_{\kappa,\lambda}:=[0,\bar{C}_{\kappa,\lambda}]\setminus \{\infty\}$,
and denote by $d_{\kappa,\lambda}$ the Riemannian distance on $I_{\kappa,\lambda} \times_{\kappa,\lambda} \bm$.

We obtain relative volume comparison theorems of Bishop-Gromov type for the metric neighborhoods of the boundaries (see Theorems \ref{thm:volume comparison} and \ref{thm:infinite volume comparison}).
We conclude rigidity theorems concerning the equality cases in those comparison theorems (see Subsection \ref{sec:Volume growth rigidity}).

One of the volume growth rigidity results is the following:
\begin{thm}\label{thm:volume growth distance rigidity}
Let $M$ be an $n$-dimensional,
connected complete Riemannian manifold with boundary,
and let $f:M\to \mathbb{R}$ be a smooth function.
Suppose that
$\bm$ is compact.
For $N\in [n,\infty)$,
we suppose $\ric^{N}_{f,M}\geq (N-1)\kappa$ and $H_{f,\bm} \geq (N-1)\lambda$.
If we have
\begin{equation}\label{eq:assumption of volume growth}
\liminf_{r\to \infty}\frac{m_{f}(B_{r}(\bm))}{s_{N,\kappa,\lambda}(r)}\geq m_{f,\bm}(\bm),
\end{equation}
then $(M,d_{M})$ is isometric to $(I_{\kappa,\lambda}\times_{\kappa,\lambda}\bm,d_{\kappa,\lambda})$,
and for every $x\in \bm$
we have $f\circ \gamma_{x}=f(x)-(N-n)\log s_{\kappa,\lambda}$ on $I_{\kappa,\lambda}$.
Moreover,
if $\kappa$ and $\lambda$ satisfy the ball-condition,
then $(M,d_{M})$ is isometric to $(\ball,d_{\ball})$ and $N=n$;
in particular,
$f$ is constant on $M$.
\end{thm}
In \cite{Sa},
Theorem \ref{thm:volume growth distance rigidity} has been proved when $f=0$ and $N=n$.

In the case of $N=\infty$,
we have the following:
\begin{thm}\label{thm:infinite volume growth distance rigidity}
Let $M$ be a connected complete Riemannian manifold with boundary,
and let $f:M\to \mathbb{R}$ be a smooth function.
Suppose that
$\bm$ is compact.
Suppose $\ric^{\infty}_{f,M}\geq 0$ and $H_{f,\bm} \geq 0$.
If we have
\begin{equation}\label{eq:assumption of infinite volume growth}
\liminf_{r\to \infty}\frac{m_{f}(B_{r}(\bm))}{r}\geq m_{f,\bm}(\bm),
\end{equation}
then $(M,d_{M})$ is isometric to $([0,\infty)\times \bm,d_{[0,\infty)\times \bm})$.
\end{thm}
\begin{rem}
On one hand,
under the same setting as in Theorem \ref{thm:volume growth distance rigidity},
we always have the following (see Lemma \ref{lem:absolute volume comparison}):
\begin{equation}\label{eq:volume growth sup}
\limsup_{r\to \infty}\frac{m_{f}(B_{r}(\bm))}{s_{N,\kappa,\lambda}(r)}\leq m_{f,\bm}(\bm).
\end{equation}
On the other hand,
under the same setting as in Theorem \ref{thm:infinite volume growth distance rigidity},
we always have the following (see Lemma \ref{lem:infinite absolute volume comparison}):
\begin{equation}\label{eq:infinite volume growth sup}
\limsup_{r\to \infty}\frac{m_{f}(B_{r}(\bm))}{r}\leq m_{f,\bm}(\bm).
\end{equation}
Theorems \ref{thm:volume growth distance rigidity} and \ref{thm:infinite volume growth distance rigidity} are concerned with rigidity phenomena.
\end{rem}
\begin{rem}
In the forthcoming paper \cite{Sa2},
we prove the same result as Theorem \ref{thm:infinite volume growth distance rigidity}
under a weaker assumption that $\ric^{N}_{f,M}\geq 0$ and $H_{f,\bm} \geq 0$ for $N<1$.
In the rigidity case,
we prove further that
for every $x\in \bm$
the function $f\circ \gamma_{x}$ is constant on $[0,\infty)$ (see Theorem 1.1 in \cite{Sa2}).
\end{rem}

\subsection{Splitting theorems}
Define a function $\tau:\bm\to \mathbb{R}\cup \{\infty\}$ by
\begin{equation}\label{eq:tau}
\tau(x):=\sup \{\,t \in (0,\infty) \mid \rho_{\bm}(\gamma_{x}(t))=t\,\}.
\end{equation}

We obtain the following splitting theorem:
\begin{thm}\label{thm:splitting}
Let $M$ be an $n$-dimensional,
connected complete Riemannian manifold with boundary,
and let $f:M\to \mathbb{R}$ be a smooth function.
Let $\kappa \leq 0$ and $\lambda:=\sqrt{\vert \kappa \vert}$.
For $N\in [n,\infty)$,
we suppose $\ric^{N}_{f,M}\geq (N-1)\kappa$ and $H_{f,\bm} \geq (N-1)\lambda$.
If for some $x_{0}\in \bm$
we have $\tau(x_{0})=\infty$,
then $(M,d_{M})$ is isometric to $([0,\infty)\times_{\kappa,\lambda}\bm,d_{\kappa,\lambda})$,
and for all $x\in \bm$ and $t\in [0,\infty)$
we have $(f\circ \gamma_{x})(t)=f(x)+(N-n)\lambda t$.
\end{thm}
In the standard case where $f=0$ and $N=n$,
Kasue \cite{K3} has proved Theorem \ref{thm:splitting}
under the assumption that the boundary is compact (see also the work of Croke and Kleiner \cite{CK}). 
In the standard case,
Theorem \ref{thm:splitting} itself has been proved in \cite{Sa}.

In the case of $N=\infty$,
we have the following splitting theorem:
\begin{thm}\label{thm:infinite splitting}
Let $M$ be a connected complete Riemannian manifold with boundary,
and let $f:M\to \mathbb{R}$ be a smooth function such that $\sup f(M)<\infty$.
Suppose $\ric^{\infty}_{f,M}\geq 0$ and $H_{f,\bm} \geq 0$.
If for some $x_{0}\in \bm$
we have $\tau(x_{0})=\infty$,
then the metric space $(M,d_{M})$ is isometric to $([0,\infty)\times\bm,d_{[0,\infty)\times\bm})$.
\end{thm}
\begin{rem}
In Theorem \ref{thm:infinite splitting},
we need the assumption $\sup f(M)<\infty$.
We denote by $\mathbb{S}^{n-1}$ the $(n-1)$-dimensional standard unit sphere,
and by $ds^{2}_{n-1}$ the canonical metric on $\mathbb{S}^{n-1}$.
We put 
\begin{equation*}
M:=\left([0,\infty)\times \mathbb{S}^{n-1},dt^{2}+\cosh^{2}t\,ds^{2}_{n-1}\right).
\end{equation*}
Let $f$ be a function on $M$ defined by $f(p):=(n-1)\rho_{\bm}(p)^{2}$.
Then for all $x\in \bm$
we have $H_{f,x}=H_{x}=0$.
Take $p\in \inte M$,
and put $l:=\rho_{\bm}(p)$.
We choose an orthonormal basis of $\{e_{i}\}^{n}_{i=1}$ of $T_{p}M$ such that $e_{n}=\nabla \rho_{\bm}$.
For all $i=1,\dots,n-1$,
we have 
\begin{equation*}
\ric_{g}(e_{i},e_{i})=(n-2)\frac{1-\sinh^{2}l}{\cosh^{2}l}-1,\, \Hess f(e_{i},e_{i})=2(n-1)l\frac{\sinh l}{\cosh l},
\end{equation*}
and $\ric_{g}(e_{n},e_{n})=-(n-1),\,\Hess f(e_{n},e_{n})=2(n-1)$.
For all $i,j=1,\dots,n$ with $i\neq j$,
we have $\ric_{g}(e_{i},e_{j})=0$ and $\Hess f(e_{i},e_{j})=0$.
From direct computations,
it follows that 
if $n\geq 3$,
then $\ric^{\infty}_{f,M}\geq 0$ and $H_{f,\bm}\geq 0$.
On the other hand,
$M$ is not isometric to the direct product $[0,\infty)\times \mathbb{S}^{n-1}$.
\end{rem}
\begin{rem}
In \cite{Sa2},
we prove the same splitting theorem as Theorem \ref{thm:infinite splitting}
under a weaker assumption that $\ric^{N}_{f,M}\geq 0$ and $H_{f,\bm} \geq 0$ for $N<1$.
In the splitting case,
we show that
for every $x\in \bm$
the function $f\circ \gamma_{x}$ is constant (see Theorem 1.3 in \cite{Sa2}).
\end{rem}
In Theorems \ref{thm:splitting} and \ref{thm:infinite splitting},
by applying the splitting theorems of Cheeger-Gromoll type (cf. \cite{CG}) to the boundary,
we obtain the multi-splitting theorems (see Subsection \ref{sec:Multi-splitting}).
We also generalize the splitting theorems studied in \cite{K3} (and \cite{CK}, \cite{I})
for manifolds with boundary whose boundaries are disconnected (see Subsection \ref{sec:Variants of splitting theorems}).

\subsection{Eigenvalue rigidity}
For $p\in [1,\infty)$,
the \textit{$(1,p)$-Sobolev space} $W^{1,p}_{0}(M,m_{f})$ \textit{on} $(M,m_{f})$ \textit{with compact support}
is defined as the completion of the set of all smooth functions on $M$ whose support is compact and contained in $\inte M$
with respect to the standard $(1,p)$-Sobolev norm.
We denote by $\Vert \cdot \Vert$ the standard norm induced from $g$,
and by $\Div$ the divergence with respect to $g$.
For $p\in [1,\infty)$,
the \textit{$(f,p)$-Laplacian} $\Delta_{f,p}\, \phi$ for $\phi \in W^{1,p}_{0}(M,m_{f})$ is defined by
\begin{equation*}
\Delta_{f,p}\,\phi:=-e^{f}\,\Div \,\left(e^{-f} \Vert \nabla \phi \Vert^{p-2}\, \nabla \phi \right)
\end{equation*}
as a distribution on $W^{1,p}_{0}(M,m_{f})$.
A real number $\mu$ is said to be an \textit{$(f,p)$-Dirichlet eigenvalue} for $\Delta_{f,p}$ on $M$
if there exists a non-zero function $\phi \in W^{1,p}_{0}(M,m_{f})$
such that $\Delta_{f,p} \phi=\mu \vert \phi \vert^{p-2}\,\phi$ holds  on $\inte M$ in a distribution sense on $W^{1,p}_{0}(M,m_{f})$. 
For $p\in [1,\infty)$,
the \textit{Rayleigh quotient} $R_{f,p}(\phi)$ for $\phi \in W^{1,p}_{0}(M,m_{f})\setminus \{0\}$ is defined as
\begin{equation*}
R_{f,p}(\phi):=\frac{\int_{M}\, \Vert \nabla \phi \Vert^{p}\,d\,m_{f}}{\int_{M}\,  \vert \phi \vert^{p}\,d\,m_{f}}.
\end{equation*}
We put $\mu_{f,1,p}(M):=\inf_{\phi} R_{f,p}(\phi)$,
where the infimum is taken over all non-zero functions in $W^{1,p}_{0}(M,m_{f})$.
The value $\mu_{f,1,2}(M)$ is equal to the infimum of the spectrum of $\Delta_{f,2}$ on $(M,m_{f})$.
If $M$ is compact,
and if $p\in (1,\infty)$,
then $\mu_{f,1,p}(M)$ is equal to the infimum of the set of all $(f,p)$-Dirichlet eigenvalues on $M$.

Let $p\in (1,\infty)$.
For $N\in [2,\infty)$,
$\kappa,\lambda\in \mathbb{R}$,
and $D\in (0,\bar{C}_{\kappa,\lambda}]\setminus \{\infty\}$,
let $\mu_{p,N,\kappa,\lambda,D}$ be the positive minimum real number $\mu$ such that
there exists a non-zero function $\phi:[0,D]\to \mathbb{R}$ satisfying
\begin{align}\label{eq:model space eigenvalue problem}
\left(\vert \phi'(t)\vert^{p-2} \phi'(t)\right)'&+(N-1)\frac{s'_{\kappa,\lambda}(t)}{s_{\kappa,\lambda}(t)}(\vert \phi'(t) \vert^{p-2} \phi'(t))\\
                                                              &+\mu\, \vert \phi(t)\vert^{p-2}\phi(t)=0, \quad \phi(0)=0, \quad \phi'(D)=0.                           \notag
\end{align}
For $D\in (0,\infty)$,
let $\mu_{p,\infty,D}$ be the positive minimum real number $\mu$ such that
there exists a non-zero function $\phi:[0,D]\to \mathbb{R}$ satisfying
\begin{equation}\label{eq:infinite model space eigenvalue problem}
\left(\vert \phi'(t)\vert^{p-2} \phi'(t)\right)'+\mu\, \vert \phi(t)\vert^{p-2}\phi(t)=0, \quad \phi(0)=0, \quad \phi'(D)=0.
\end{equation}

We recall the notion of model spaces
that has been introduced by Kasue in \cite{K4} in our setting.
We say that
$\kappa \in \mathbb{R}$ and $\lambda \in \mathbb{R}$ satisfy the \textit{model-condition}
if the equation $s'_{\kappa,\lambda}(t)=0$ has a positive solution.
We see that
$\kappa$ and $\lambda$ satisfy the model-condition if and only if either
(1) $\kappa>0$ and $\lambda<0$;
(2) $\kappa=0$ and $\lambda=0$;
or (3) $\kappa<0$ and $\lambda \in (0,\sqrt{\vert \kappa \vert})$.

Let $\kappa\in \mathbb{R}$ and $\lambda \in \mathbb{R}$ satisfy the ball-condition or
the model-condition.
Suppose that
$M$ is compact.
For $\kappa$ and $\lambda$ satisfying the model-condition,
we define a positive number $D_{\kappa,\lambda}(M)$ as follows:
If $\kappa=0$ and $\lambda=0$,
then $D_{\kappa,\lambda}(M):=D(M,\bm)$;
otherwise,
$D_{\kappa,\lambda}(M):=\{t>0\mid s'_{\kappa,\lambda}(t)=0\}$.
We say that $(M,d_{M})$ is a \textit{$(\kappa,\lambda)$-equational model space}
if $M$ is isometric to either
(1) for $\kappa$ and $\lambda$ satisfying the ball-condition,
the closed geodesic ball $\ball$;
(2) for $\kappa$ and $\lambda$ satisfying the model-condition,
and for a connected component $\bm_{1}$ of $\bm$,
the warped product $[0,2D_{\kappa,\lambda}(M)]\times_{\kappa,\lambda} \bm_{1}$;
or (3) for $\kappa$ and $\lambda$ satisfying the model-condition,
and for an involutive isometry $\sigma$ of $\bm$ without fixed points,
the quotient space $([0,2D_{\kappa,\lambda}(M)]\times_{\kappa,\lambda} \bm)/G_{\sigma}$,
where $G_{\sigma}$ is the isometry group on $[0,2D_{\kappa,\lambda}(M)]\times_{\kappa,\lambda} \bm$ whose elements consist of the identity and the involute isometry $\hat{\sigma}$ defined by $\hat{\sigma}(t,x):=(2D_{\kappa,\lambda}(M)-t,\sigma(x))$.

Let $p\in (1,\infty)$.
Let $M$ be a $(\kappa,\lambda)$-equational model space.
From a standard argument,
we see that if $M$ is isometric to $\ball$,
then $\mu_{0,1,p}(M)=\mu_{p,n,\kappa,\lambda,\const}$.
Furthermore,
if $M$ is not isometric to $\ball$,
then $\mu_{0,1,p}(M)=\mu_{p,n,\kappa,\lambda,D_{\kappa,\lambda}(M)}$ for the corresponding $\kappa,\lambda$ and $D_{\kappa,\lambda}(M)$.

We establish the following rigidity theorem for $\mu_{f,1,p}$:
\begin{thm}\label{thm:eigenvalue rigidity}
Let $M$ be an $n$-dimensional,
connected complete Riemannian manifold with boundary,
and let $f:M\to \mathbb{R}$ be a smooth function.
Suppose that
$M$ is compact.
Let $p\in (1,\infty)$.
For $N\in [n,\infty)$,
we suppose $\ric^{N}_{f,M}\geq (N-1)\kappa$ and $H_{f,\bm} \geq (N-1)\lambda$.
For $D\in (0,\bar{C}_{\kappa,\lambda}]\setminus \{\infty\}$,
we assume $\dm \leq D$.
Then we have
\begin{equation}\label{eq:eigenvalue rigidity}
\mu_{f,1,p}(M)\geq \mu_{p,N,\kappa,\lambda,D}.
\end{equation}
If the equality in $(\ref{eq:eigenvalue rigidity})$ holds,
then $(M,d_{M})$ is a $(\kappa,\lambda)$-equational model space;
more precisely,
the following hold:
\begin{enumerate}
\item if $D=\bar{C}_{\kappa,\lambda}$,
then $\kappa$ and $\lambda$ satisfy the ball-condition,
$(M,d_{M})$ is isometric to $(\ball,d_{\ball})$,
and $N=n$;
in particular,
$f$ is constant on $M$;
\item if $D \in (0,\bar{C}_{\kappa,\lambda})$,
then $\kappa$ and $\lambda$ satisfy the model-condition,
$(M,d_{M})$ is a $(\kappa,\lambda)$-equational model space,
and $f\circ \gamma_{x}=f(x)-(N-n)\log s_{\kappa,\lambda}$ on $[0,D_{\kappa,\lambda}(M)]$ for all $x\in \bm$.
\end{enumerate}
\end{thm}
Kasue \cite{K4} has proved Theorem \ref{thm:eigenvalue rigidity} when $p=2,f=0$ and $N=n$.
It seems that
the method of the proof in \cite{K4} does not work in our non-linear case of $p\neq 2$ (see Remark \ref{rem:method of the proof}).
We prove Theorem \ref{thm:eigenvalue rigidity}
by a global Laplacian comparison result for $\rho_{\bm}$ (see Proposition \ref{prop:global p-Laplacian comparison})
and an inequality of Picone type for the $p$-Laplacian (see Lemma \ref{lem:Picone identity}).

In the case of $N=\infty$,
we have the following:
\begin{thm}\label{thm:infinite eigenvalue rigidity}
Let $M$ be a connected complete Riemannian manifold with boundary,
and let $f:M\to \mathbb{R}$ be a smooth function.
Suppose that
$M$ is compact.
Let $p\in (1,\infty)$.
Suppose $\ric^{\infty}_{f,M}\geq 0$ and $H_{f,\bm} \geq 0$.
For $D\in (0,\infty)$,
we assume $\dm \leq D$.
Then we have
\begin{equation}\label{eq:infinite eigenvalue rigidity}
\mu_{f,1,p}(M)\geq \mu_{p,\infty,D}.
\end{equation}
If the equality in $(\ref{eq:infinite eigenvalue rigidity})$ holds,
then the metric space $(M,d_{M})$ is a $(0,0)$-equational model space,
and $\dm=D$.
\end{thm}
\begin{rem}
In \cite{Sa2},
we prove the same rigidity result as Theorem \ref{thm:infinite eigenvalue rigidity}
under a weaker assumption that $\ric^{N}_{f,M}\geq 0$ and $H_{f,\bm} \geq 0$ for $N<1$.
In the rigidity case,
we also prove that
for every $x\in \bm$
the function $f\circ \gamma_{x}$ is constant on $[0,D]$ (see Theorem 1.5 in \cite{Sa2}).
\end{rem}
In Theorems \ref{thm:eigenvalue rigidity} and \ref{thm:infinite eigenvalue rigidity},
we have explicit lower bounds for $\mu_{f,1,p}$ (see Subsection \ref{sec:Concrete large lower bounds}).

We show some volume estimates for a relatively compact domain in $M$ (see Propositions \ref{prop:Kasue volume estimate} and \ref{prop:infinite Kasue volume estimate}).
From the volume estimates,
we derive lower bounds for $\mu_{f,1,p}$ for manifolds with boundary that are not necessarily compact (see Theorems \ref{thm:p-Laplacian1} and \ref{thm:infinite p-Laplacian1}).
By using the estimate for $\mu_{f,1,p}$,
and by using Theorem \ref{thm:splitting},
we obtain the following:
\begin{thm}\label{thm:spectrum rigidity}
Let $M$ be an $n$-dimensional,
connected complete Riemannian manifold with boundary.
Let $f:M\to \mathbb{R}$ be a smooth function.
Suppose that
$\bm$ is compact.
Let $p\in (1,\infty)$.
Let $\kappa<0$ and $\lambda:=\sqrt{\vert \kappa \vert}$.
For $N\in [n,\infty)$,
we suppose $\ric^{N}_{f,M}\geq (N-1)\kappa$ and $H_{f,\bm} \geq (N-1)\lambda$.
Then we have
\begin{equation}\label{eq:noncompact estimate}
\mu_{f,1,p}(M)\geq \left(\frac{(N-1)\lambda}{p}\right)^{p}.
\end{equation}
If the equality in $(\ref{eq:noncompact estimate})$ holds,
then the metric space $(M,d_{M})$ is isometric to $([0,\infty)\times_{\kappa,\lambda}\bm, d_{\kappa,\lambda})$,
and for all $x\in \bm$ and $t\in [0,\infty)$
we have $(f \circ \gamma_{x})(t)=f(x)+(N-n)\lambda t$.
\end{thm}
Theorem \ref{thm:spectrum rigidity} has been proved in \cite{Sa} in the standard case where $f=0$ and $N=n$.

\subsection{Organization}
In Section \ref{sec:Preliminaries},
we prepare some notations
and recall the basic facts for Riemannian manifolds with boundary.
In Section \ref{sec:Laplacian comparisons},
we show Laplacian comparison theorems for the distance function from the boundary.
In Section \ref{sec:Inscribed radius rigidity},
we prove Theorem \ref{thm:Ball rigidity}.
In Section \ref{sec:Volume comparisons},
we show several volume comparison theorems,
and conclude Theorems \ref{thm:volume growth distance rigidity} and \ref{thm:infinite volume growth distance rigidity}.
In Section \ref{sec:Splitting theorems},
we prove Theorems \ref{thm:splitting} and \ref{thm:infinite splitting},
and study the variants of the splitting theorems.
In Section \ref{sec:Eigenvalue rigidity},
we prove Theorems \ref{thm:eigenvalue rigidity} and \ref{thm:infinite eigenvalue rigidity},
and study explicit lower bounds for $\mu_{f,1,p}$.
In Section \ref{sec:First eigenvalue estimates},
we prove Theorem \ref{thm:spectrum rigidity}.

\subsection*{{\rm Acknowledgements}}
The author would like to express his gratitude to Professor Koichi Nagano for his constant advice and suggestions.
The author would also like to thank Professors Takashi Shioya and Christina Sormani for their valuable comments.
The author is grateful to an anonymous referee of some journal for useful comments.
One of the comments leads the author to study rigidity phenomena in weighting functions.

\section{Preliminaries}\label{sec:Preliminaries}
We refer to \cite{BBI} for the basics of metric geometry,
and to \cite{S} for the basics of Riemannian manifolds with boundary.
\subsection{Metric spaces}
Let $(X,d_{X})$ be a metric space with metric $d_{X}$.
For $r>0$ and $A\subset X$,
we denote by $U_{r}(A)$ the open $r$-neighborhood of $A$ in $X$,
and by $B_{r}(A)$ the closed one.
For $A_{1}, A_{2}\subset X$,
we put $d_{X}(A_{1},A_{2}):=\inf_{x_{1}\in A_{1}, x_{2}\in A_{2}}\,d_{X}(x_{1},x_{2})$.

For a metric space $(X,d_{X})$,
the length metric $\bar{d}_{X}$ is defined as follows:
For two points $x_{1},x_{2}\in X$,
we put $\bar{d}_{X}(x_{1},x_{2})$
to the infimum of the length of curves connecting $x_{1}$ and $x_{2}$ with respect to $d_{X}$.
A metric space $(X,d_{X})$ is said to be a \textit{length space} if $d_{X}=\bar{d}_{X}$.

Let $(X,d_{X})$ be a metric space.
For an interval $I$,
we say that
a curve $\gamma:I\to X$ is a \textit{normal minimal geodesic}
if for all $s,t\in I$
we have $d_{X}(\gamma(s),\gamma(t))=\vert s-t\vert$,
and $\gamma$ is a \textit{normal geodesic}
if for each $t\in I$
there exists an interval $J\subset I$ with $t\in J$ such that $\gamma|_{J}$ is a normal minimal geodesic.
A metric space $(X,d_{X})$ is said to be a \textit{geodesic space} if
for every pair of points in $X$,
there exists a normal minimal geodesic connecting them.
A metric space is \textit{proper} if all closed bounded subsets of the space are compact.
The Hopf-Rinow theorem for length spaces states that 
if a length space $(X,d_{X})$ is complete and locally compact,
and if $d_{X}<\infty$,
then $(X,d_{X})$ is a proper geodesic space (see e.g., Theorem 2.5.23 in \cite{BBI}).

\subsection{Riemannian manifolds with boundary}
For $n\geq 2$, 
let $M$ be an $n$-dimensional,
connected Riemannian manifold with (smooth) boundary
with Riemannian metric $g$.
For $p\in \inte M$, 
let $T_{p}M$ be the tangent space at $p$ on $M$,
and let $U_{p}M$ be the unit tangent sphere at $p$ on $M$.
We denote by $\Vert \cdot \Vert$ the standard norm induced from $g$.
If $v_{1},\dots,v_{k}\in T_{p}M$ are linearly independent,
then we see $\Vert v_{1}\wedge \cdots \wedge v_{k} \Vert=\sqrt{\det (g(v_{i},v_{j}))}$.
Let $d_{M}$ be the length metric induced from $g$.
If $M$ is complete with respect to $d_{M}$,
then the Hopf-Rinow theorem for length spaces tells us that the metric space $(M,d_{M})$ is a proper geodesic space.

For $i=1,2$,
let $M_{i}$ be connected Riemannian manifolds with boundary with Riemannian metric $g_{i}$.
For each $i$,
the boundary $\bm_{i}$ carries the induced Riemannian metric $h_{i}$.
We say that a homeomorphism $\Phi:M_{1}\to M_{2}$ is a \textit{Riemannian isometry with boundary} from $M_{1}$ to $M_{2}$ if $\Phi$ satisfies the following conditions:
\begin{enumerate}
 \item $\Phi|_{\inte M_{1}}:\inte M_{1} \to \inte M_{2}$ is smooth, and $(\Phi|_{\inte M_{1}})^{\ast} (g_{2})=g_{1}$;\label{enum:inner isom}
 \item $\Phi|_{\bm_{1}}:\bm_{1} \to \bm_{2}$ is smooth, and $(\Phi|_{\bm_{1}})^{\ast} (h_{2})=h_{1}$.\label{enum:bdry isom}
\end{enumerate}
If $\Phi:M_{1}\to M_{2}$ is a Riemannian isometry with boundary,
then the inverse $\Phi^{-1}$ is also a Riemannian isometry with boundary.
Notice that
there exists a Riemannian isometry with boundary from $M_{1}$ to $M_{2}$ if and only if
the metric space $(M_{1},d_{M_{1}})$ is isometric to $(M_{2},d_{M_{2}})$ (see e.g., Section 2 in \cite{Sa}).

\subsection{Jacobi fields orthogonal to the boundary}
Let $M$ be a connected Riemannian manifold with boundary with Riemannian metric $g$.
For a point $x\in \bm$,
and for the tangent space $T_{x}\bm$ at $x$ on $\bm$,
let $T_{x}^{\perp} \bm$ be the orthogonal complement of $T_{x}\bm$ in the tangent space at $x$ on $M$.
Take $u\in T_{x}^{\perp}\bm$.
For the second fundamental form $S$ of $\bm$,
let $A_{u}:T_{x}\bm \to T_{x}\bm$ be the \textit{shape operator} for $u$ defined as
\begin{equation*}
g(A_{u}v,w):=g(S(v,w),u).
\end{equation*}
We denote by $u_{x}$ the unit inner normal vector at $x$.
The \textit{mean curvature} $H_{x}$ at $x$ is defined as $H_{x}:=\tr A_{u_{x}}$.
We denote by $\gamma_{x}:[0,T)\to M$ the normal geodesic with initial conditions $\gamma_{x}(0)=x$ and $\gamma_{x}'(0)=u_{x}$.
We say that a Jacobi field $Y$ along $\gamma_{x}$ is a $\bm$-\textit{Jacobi field} if $Y$ satisfies the following initial conditions:
\begin{equation*}
Y(0)\in T_{x}\bm, \quad Y'(0)+A_{u_{x}}Y(0)\in T_{x}^{\perp}\bm.
\end{equation*}
We say that $\gamma_{x}(t_{0})$ is a \textit{conjugate point} of $\bm$ along $\gamma_{x}$
if there exists a non-zero $\bm$-Jacobi field $Y$ along $\gamma_{x}$ with $Y(t_{0})=0$.
We denote by $\tau_{1}(x)$ the first conjugate value for $\bm$ along $\gamma_{x}$.
It is well-known that for all $x\in \bm$ and $t>\tau_{1}(x)$,
we have $t>\rho_{\bm}(\gamma_{x}(t))$.

For a point $x\in \bm$,
and for a piecewise smooth vector field $X$ along $\gamma_{x}$ with $X(0)\in T_{x}\bm$,
the \textit{index form} of $\gamma_{x}$ is defined as
\begin{align*}
I_{\bm}(X,X):&= \int_{0}^{t}g(X'(t),X'(t))-g(R(X(t),\gamma'_{x}(t))\gamma'_{x}(t),X(t))\,dt\\
             &- g(A_{u_{x}}X(0),X(0)).
\end{align*}
\begin{lem}\label{lem:index form}
For $x\in \bm$,
we suppose that
there exists no conjugate point of $\bm$ on $\gamma_{x}|_{[0,t_{0}]}$.
Then for every piecewise smooth vector field $X$ along $\gamma_{x}$ with $X(0)\in T_{x}\bm$,
there exists a unique $\bm$-Jacobi field $Y$ along $\gamma_{x}$ with $X(t_{0})=Y(t_{0})$ such that
\begin{equation*}
I_{\bm}(Y,Y)\leq I_{\bm}(X,X);
\end{equation*}
the equality holds if and only if $X=Y$ on $[0,t_{0}]$.
\end{lem}
For the normal tangent bundle $\tbp:=\bigcup_{x\in \bm} T_{x}^{\perp}\bm$ of $\bm$,
let $0(\tbp)$ be the zero-section $\bigcup_{x\in \bm} \{\,0_{x}\in T_{x}^{\perp}\bm\, \}$ of $T^{\perp}\bm$.
On an open neighborhood of $0(\tbp)$ in $\tbp$, 
the normal exponential map $\expp$ of $\bm$ is defined as $\expp(x,u):=\gamma_{x}(\Vert u\Vert)$
for $x\in \bm$ and $u\in T_{x}^{\perp}\bm$.

For $x\in \bm$ and $t\in [0,\tau_{1}(x))$,
we denote by $\theta(t,x)$ the absolute value of the Jacobian of $\expp$ at $(x,tu_{x})\in \tbp$.
For each $x\in \bm$,
we choose an orthonormal basis $\{e_{x,i}\}_{i=1}^{n-1}$ of $T_{x}\bm$.
For each $i$,
let $Y_{x,i}$ be the $\bm$-Jacobi field along $\gamma_{x}$ with initial conditions $Y_{x,i}(0)=e_{x,i}$ and $Y'_{x,i}(0)=-A_{u_{x}}e_{x,i}$.
Note that for all $x\in \bm$ and $t\in [0,\tau_{1}(x))$,
we have $\theta(t,x)=\Vert Y_{x,1}(t)\wedge \cdots \wedge Y_{x,n-1}(t)\Vert$.
This does not depend on the choice of the orthonormal bases.
\subsection{Cut locus for the boundary}
We recall the basic properties of the cut locus for the boundary.
The basic properties seem to be well-known.
We refer to \cite{Sa} for the proofs.

Let $M$ be a connected complete Riemannian manifold with boundary with Riemannian metric $g$.
For $p\in M$, 
we call $x\in \bm$ a \textit{foot point} on $\bm$ of $p$ if $d_{M}(p,x)=\rho_{\bm}(p)$.
Since $(M,d_{M})$ is proper, 
every point in $M$ has at least one foot point on $\bm$.
For $p\in \inte M$, 
let $x\in \bm$ be a foot point on $\bm$ of $p$.
Then there exists a unique normal minimal geodesic $\gamma:[0,l]\to M$ from $x$ to $p$
such that $\gamma=\gamma_{x}|_{[0,l]}$,
where $l=\rho_{\bm}(p)$.
In particular,
$\gamma'(0)=u_{x}$ and $\gamma|_{(0,l]}$ lies in $\inte M$.

Let $\tau:\bm\to \mathbb{R}\cup \{\infty\}$ be the function defined as (\ref{eq:tau}).
By the property of $\tau_{1}$,
for all $x\in \bm$
we have $0<\tau(x)\leq \tau_{1}(x)$.
The function $\tau$ is continuous on $\bm$.

We have already known the following (see e.g., Section 3 in \cite{Sa}):
\begin{prop}\label{prop:metric neighborhood}
For every $r\in (0,\infty)$
we have
\begin{equation*}\label{eq:neighborhood}
B_{r}(\bm)=\expp \left(\bigcup_{x\in \bm} \{tu_{x}\mid t\in [0,\min\{r,\tau(x)\}] \}\right).
\end{equation*}
\end{prop}
For the inscribed radius $D(M,\bm)$ of $M$,
from the definition of $\tau$,
we deduce $D(M,\bm)=\sup_{x\in \bm} \tau(x)$.

The continuity of $\tau$ implies the following (see e.g., Section 3 in \cite{Sa}):
\begin{lem}\label{lem:bmcompact}
Suppose that
$\bm$ is compact.
Then $D(M,\bm)$ is finite if and only if $M$ is compact.
\end{lem}
We put
\begin{align*}
TD_{\bm}  &:= \bigcup_{x\in \bm} \{\,t\,u_{x} \in T^{\perp}_{x}\bm \mid t\in[0,\tau(x)) \,\},\\
T\cut \bm &:= \bigcup_{x\in \bm} \{\,\tau(x)\,u_{x}\in T^{\perp}_{x}\bm \mid \tau(x)<\infty \,\},
\end{align*}
and define $D_{\bm}:=\expp (TD_{\bm})$ and $\cut \bm:=\expp (T\cut \bm)$.
We call $\cut \bm$ the \textit{cut locus for the boundary} $\bm$.
By the continuity of $\tau$,
the set $\cut \bm$ is a null set of $M$.
Furthermore,
we have
\begin{equation*}
\inte M=(D_{\bm}\setminus \bm) \sqcup \cut \bm,\quad M=D_{\bm}\sqcup \cut \bm.
\end{equation*}
This implies that
if $\cut\bm=\emptyset$,
then $\bm$ is connected.
The set $TD_{\bm}\setminus 0(T^{\perp}\bm)$ is a maximal domain in $T^{\perp}\bm$ on which $\expp$ is regular and injective.

The following has been shown in the proof of Theorem 1.3 in \cite{Sa}:
\begin{lem}\label{lem:splitting1}
If there exists a connected component $\bm_{0}$ of $\bm$ such that 
for all $x\in \bm_{0}$ we have $\tau(x)=\infty$,
then $\bm$ is connected and $\cut \bm=\emptyset$.
\end{lem}
The function $\rho_{\bm}$ is smooth on $\inte M\setminus \cut \bm$.
For each $p\in \inte M\setminus \cut \bm$,
the gradient vector $\nabla \rho_{\bm}(p)$ of $\rho_{\bm}$ at $p$
is given by $\nabla \rho_{\bm}(p)=\gamma'(l)$,
where $\gamma:[0,l]\to M$ is the normal minimal geodesic from the foot point on $\bm$ of $p$ to $p$.

For $\Omega \subset M$,
we denote by $\bar{\Omega}$ the closure of $\Omega$ in $M$,
and by $\partial \Omega$ the boundary of $\Omega$ in $M$.
For a domain $\Omega$ in $M$ such that
$\partial \Omega$ is a smooth hypersurface in $M$,
we denote by $\vol_{\partial \Omega}$ the canonical Riemannian volume measure on $\partial \Omega$.

We have the following fact to avoid the cut locus for the boundary:
\begin{lem}\label{lem:avoiding the cut locus2}
Let $\Omega$ be a domain in $M$ such that
$\partial \Omega$ is a smooth hypersurface in $M$.
Then there exists a sequence $\{\Omega_{k}\}_{k\in \mathbb{N}}$ of closed subsets of $\bar{\Omega}$ satisfying that
for every $k\in \mathbb{N}$,
the set $\partial \Omega_{k}$ is a smooth hypersurface in $M$ except for a null set in $(\partial \Omega,\vol_{\partial \Omega})$,
and satisfying the following properties:
\begin{enumerate}
\item for all $k_{1},k_{2}\in \mathbb{N}$ with $k_{1}<k_{2}$,
         we have $\Omega_{k_{1}}\subset \Omega_{k_{2}}$;
\item $\bar{\Omega} \setminus \cut \bm=\bigcup_{k\in \mathbb{N}}\,\Omega_{k}$;
\item for every $k\in \mathbb{N}$,
         and for almost every point $p \in \partial \Omega_{k}\cap \partial \Omega$ in $(\partial \Omega,\vol_{\partial \Omega})$,
         there exists the unit outer normal vector for $\Omega_{k}$ at $p$
         that coincides with the unit outer normal vector on $\partial \Omega$ for $\Omega$ at $p$;
\item for every $k\in \mathbb{N}$,
         on $\partial \Omega_{k}\setminus \partial \Omega$,
         there exists the unit outer normal vector field $\nu_{k}$ for $\Omega_{k}$ such that $g(\nu_{k},\nabla \rho_{\bm})\geq 0$.
\end{enumerate}
Moreover,
if $\bar{\Omega}=M$,
then for every $k\in \mathbb{N}$,
the set $\partial \Omega_{k}$ is a smooth hypersurface in $M$,
and satisfies $\partial \Omega_{k}\cap \bm=\bm$.
\end{lem}
For the cut locus for a single point,
a similar result to Lemma \ref{lem:avoiding the cut locus2} is well-known (see e.g., Theorem 4.1 in \cite{Che2}).
One can prove Lemma \ref{lem:avoiding the cut locus2}
by a similar method to that of the proof of the result for the cut locus for a single point.
We omit the proof.
\subsection{Busemann functions and asymptotes}
Let $M$ be a connected complete Riemannian manifold with boundary.
A normal geodesic $\gamma:[0,\infty)\to M$ is said to be a \textit{ray}
if for all $s,t\in [0,\infty)$
it holds that $d_{M}(\gamma(s),\gamma(t))=|s-t|$.
For a ray $\gamma:[0,\infty)\to M$, 
the \textit{Busemann function} $b_{\gamma}:M\to \mathbb{R}$ \textit{of} $\gamma$ is defined as
\begin{equation*}
b_{\gamma}(p):=\lim_{t\to \infty}(t-d_{M}(p,\gamma(t))).
\end{equation*}

Take a ray $\gamma:[0,\infty)\to M$
and a point $p\in \inte M$,
and choose a sequence $\{t_{i}\}$ with $t_{i}\to \infty$.
For each $i$,
we take a normal minimal geodesic $\gamma_{i}:[0,l_{i}]\to M$ from $p$ to $\gamma(t_{i})$.
Since $\gamma$ is a ray,
it follows that $l_{i}\to \infty$.
Take a sequence $\{T_{j}\}$ with $T_{j}\to \infty$.
Using the fact that $M$ is proper,
we take a subsequence $\{\gamma_{1,i}\}$ of $\{\gamma_{i}\}$,
and a normal minimal geodesic $\gamma_{p,1}:[0,T_{1}]\to M$ from $p$ to $\gamma_{p,1}(T_{1})$
such that $\gamma_{1,i}|_{[0,T_{1}]}$ uniformly converges to $\gamma_{p,1}$.
In this manner,
take a subsequence $\{\gamma_{2,i}\}$ of $\{\gamma_{1,i}\}$
and a normal minimal geodesic $\gamma_{p,2}:[0,T_{2}]\to M$ from $p$ to $\gamma_{p,2}(T_{2})$
such that $\gamma_{2,i}|_{[0,T_{2}]}$ uniformly converges to $\gamma_{p,2}$,
where $\gamma_{p,2}|_{[0,T_{1}]}=\gamma_{p,1}$.
By means of a diagonal argument,
we obtain a subsequence $\{\gamma_{k}\}$ of $\{\gamma_{i}\}$
and a ray $\gamma_{p}$ in $M$ such that
for every $t\in (0,\infty)$
we have $\gamma_{k}(t)\to \gamma_{p}(t)$ as $k\to \infty$.
We call such a ray $\gamma_{p}$ an \textit{asymptote for} $\gamma$ \textit{from} $p$.

The following lemmas have been shown in \cite{Sa}.
\begin{lem}\label{lem:busemann function}
Suppose that for some $x \in \bm$
we have $\tau(x)=\infty$.
Take $p\in \inte M$.
If $b_{\gamma_{x}}(p)=\rho_{\bm}(p)$,
then $p\notin \cut \bm$.
Moreover,
for the unique foot point $y$ on $\bm$ of $p$,
we have $\tau(y)=\infty$.
\end{lem}
\begin{lem}\label{lem:asymptote}
Suppose that for some $x\in \bm$
we have $\tau(x)=\infty$.
For $l \in (0,\infty)$,
put $p:=\gamma_{x}(l)$.
Then there exists $\epsilon \in (0,\infty)$ such that
for all $q\in B_{\epsilon}(p)$,
all asymptotes for the ray $\gamma_{x}$ from $q$ lie in $\inte M$.
\end{lem}

\subsection{Weighted Laplacians}
Let $M$ be a connected complete Riemannian manifold with boundary with Riemannian metric $g$,
and let $f:M\to \mathbb{R}$ be a smooth function.
For a smooth function $\phi$ on $M$,
the \textit{weighted Laplacian $\Delta_{f} \phi$ for $\phi$} is defined by
\begin{equation*}
\Delta_{f} \phi:=\Delta \phi+g(\nabla f,\nabla \phi),
\end{equation*}
where $\Delta \phi$ is the Laplacian for $\phi$ defined as the minus of the trace of its Hessian.
Notice that
$\Delta_{f}$ coincides with the $(f,2)$-Laplacian $\Delta_{f,2}$.

For $x \in \bm$ and $t \in [0,\tau_{1}(x))$,
we put $\theta_{f}(t,x):=e^{-f(\gamma_{x}(t))}\, \theta(t,x)$.
For all $x\in \bm$ and $t\in (0,\tau(x))$,
we see
\begin{equation}\label{eq:Laplacian representation}
\Delta_{f}\, \rho_{\bm}(\gamma_{x}(t)) =-(\log \theta(t,x))'+f(\gamma_{x}(t))'=-\frac{\theta_{f}'(t,x)}{\theta_{f}(t,x)}.
\end{equation}

For $\kappa\in \mathbb{R}$, 
let $s_{\kappa}(t)$ be a unique solution of the so-called Jacobi-equation $\phi''(t)+\kappa \phi(t)=0$ with initial conditions $\phi(0)=0$ and $\phi'(0)=1$.
We put $c_{\kappa}(t):=s'_{\kappa}(t)$.

For $p\in M$,
let $\rho_{p}:M\to \mathbb{R}$ denote the distance function from $p$ defined as $\rho_{p}(q):=d_{M}(p,q)$.

Qian \cite{Q} has proved a Laplacian comparison inequality for the distance function from a single point (see equation 7 in \cite{Q}).
In our setting,
the comparison inequality holds in the following form:
\begin{lem}[\cite{Q}]\label{lem:finite pointed Laplacian comparison}
For $N\in [n,\infty)$,
we suppose $\ric^{N}_{f,M}\geq (N-1)\kappa$.
Take $p\in \inte M$.
Assume that
there exists $q\in \inte M\setminus \{p\}$ such that
a normal minimal geodesic in $M$ from $p$ to $q$ lies in $\inte M$,
and $\rho_{p}$ is smooth at $q$.
Then
\begin{equation}\label{eq:finite pointed Laplacian comparison}
\Delta_{f}\, \rho_{p}(q)\geq -(N-1)\frac{c_{\kappa}(\rho_{p}(q))}{s_{\kappa}(\rho_{p}(q))}.
\end{equation}
\end{lem}
\begin{rem}\label{rem:equality pointed Laplacian comparison}
In Lemma \ref{lem:finite pointed Laplacian comparison},
we choose a normal minimal geodesic $\gamma:[0,l]\to M$ from $p$ to $q$ that lies in $\inte M$,
and an orthonormal basis $\{e_{i}\}^{n}_{i=1}$ of $T_{p}M$ with $e_{n}=\gamma'(0)$.
Let $\{Y_{i}\}^{n-1}_{i=1}$ be the Jacobi fields along $\gamma$ with initial conditions $Y_{i}(0)=0$ and $Y'_{i}(0)=e_{i}$.
If the equality in (\ref{eq:finite pointed Laplacian comparison}) holds,
then for all $i$
we see $Y_{i}=s_{\kappa}\,E_{i}$ on $[0,l]$,
where $\{E_{i}\}^{n-1}_{i=1}$ are the parallel vector fields along $\gamma$ with initial condition $E_{i}(0)=e_{i}$.
\end{rem}
\begin{rem}
Kasue and Kumura \cite{KK} have been proved Lemma \ref{lem:finite pointed Laplacian comparison}
in the case where $N$ is an integer,
and $\kappa \leq 0$.
\end{rem}

Let $\phi:M\to \mathbb{R}$ be a continuous function,
and let $U$ be a domain contained in $\inte M$.
For $p \in U$,
and for a function $\psi$ defined on an open neighborhood of $p$,
we say that
$\psi$ is a \textit{support function of $\phi$ at $p$}
if we have $\psi(p)=\phi(p)$ and $\psi \leq \phi$.
We say that
$\phi$ is \textit{$f$-subharmonic on $U$}
if for every $p\in U$,
and for every $\epsilon \in (0,\infty)$,
there exists a smooth,
support function $\psi_{p,\epsilon}$ of $\phi$ at $p$ such that $\Delta_{f}\, \psi_{p,\epsilon}(p)\leq \epsilon$.

We recall the following maximal principle of Calabi type (see e.g., \cite{C}, and Lemma 2.4 in \cite{FLZ}).
\begin{lem}\label{lem:maximal principle}
If an $f$-subharmonic function on a domain $U$ contained in $\inte M$
takes the maximal value at a point in $U$,
then it must be constant on $U$.
\end{lem}
Fang, Li and Zhang \cite{FLZ} have proved a subharmonicity of Busemann functions on manifolds without boundary (see Lemma 2.1 in \cite{FLZ}).
In our setting,
the subharmonicity holds in the following form:
\begin{lem}[\cite{FLZ}]\label{lem:infinite pointed Laplacian comparison}
Assume $\sup f(M)<\infty$.
Suppose $\ric^{\infty}_{f,M}\geq 0$.
Let $\gamma:[0,\infty)\to M$ be a ray that lies in $\inte M$,
and let $U$ be a domain contained in $\inte M$ such that
for each $p\in U$,
there exists an asymptote for $\gamma$ from $p$
that lies in $\inte M$.
Then $b_{\gamma}$ is $f$-subharmonic on $U$. 
\end{lem}

\section{Laplacian comparisons}\label{sec:Laplacian comparisons}
In this section,
let $M$ be an $n$-dimensional, 
connected complete Riemannian manifold with boundary with Riemannian metric $g$,
and let $f:M\to \mathbb{R}$ be a smooth function.
\subsection{Basic comparisons}
We prove the following basic lemma:
\begin{lem}\label{lem:Basic comparison}
Take $x\in \bm$.
For $N\in [n,\infty)$,
suppose that for all $t\in (0,\min \{\tau_{1}(x),\bar{C}_{\kappa,\lambda}\} )$
we have $\ric^{N}_{f}(\gamma'_{x}(t)) \geq (N-1)\kappa$,
and suppose $H_{f,x}\geq (N-1)\lambda$.
Then for all $t\in (0,\min \{\tau_{1}(x),\bar{C}_{\kappa,\lambda}\})$
we have
\begin{equation}\label{eq:Basic1}
\frac{\theta_{f}'(t,x)}{\theta_{f}(t,x)}\leq (N-1)\frac{s_{\kappa,\lambda}'(t)}{s_{\kappa,\lambda}(t)},
\end{equation}
and for all $s,t\in [0,\min \{\tau_{1}(x),\bar{C}_{\kappa,\lambda}\})$ with $s\leq t$
we have
\begin{equation}\label{eq:Basic2}
\frac{\theta_{f}(t,x)}{\theta_{f}(s,x)}\leq \frac{s_{\kappa,\lambda}^{N-1}(t)}{s_{\kappa,\lambda}^{N-1}(s)};
\end{equation}
in particular,
$\theta_{f}(t,x)\leq  e^{-f(x)}\,s_{\kappa,\lambda}^{N-1}(t)$.
\end{lem}
\begin{proof}
Put $F:=f\circ \gamma_{x}$.
From direct computations,
it follows that
\begin{equation}\label{eq:Basic3}
\frac{\theta_{f}'(t,x)}{\theta_{f}(t,x)}=\frac{\theta'(t,x)}{\theta(t,x)}-F'(t)
\end{equation}
for all $t\in (0,\min \{\tau_{1}(x),\bar{C}_{\kappa,\lambda}\})$.
Choose an orthonormal basis $\{ e_{i} \}_{i=1}^{n-1}$ of $T_{x}\partial M$.
For each $i$,
we denote by $E_{i}$ the parallel vector field along $\gamma_{x}$
with initial condition $E_{i}(0)=e_{i}$.
We fix $t_{0}\in (0,\min \{\tau_{1}(x),\bar{C}_{\kappa,\lambda}\})$,
and put $W_{i}(t):= (s_{\kappa,\lambda}(t)/s_{\kappa,\lambda}(t_{0})) E_{i}(t)$ for $t\in (0,\min \{\tau_{1}(x),\bar{C}_{\kappa,\lambda}\})$.
For a unique $\bm$-Jacobi field $Y_{t_{0},i}$ along $\gamma_{x}|_{[0,t_{0}]}$
with initial conditions $Y_{t_{0},i}(t_{0})=W_{i}(t_{0})\, (=E_{i}(t_{0}))$ and $Y'_{t_{0},i}(0)=-A_{u_{x}}Y_{t_{0},i}(0)$,
let $\theta_{t_{0}}(t):=\Vert Y_{t_{0},1}(t)\wedge \cdots \wedge Y_{t_{0},n-1}(t) \Vert$ for $t\in (0,\min \{\tau_{1}(x),\bar{C}_{\kappa,\lambda}\})$.
The linearity of the Jacobi equations implies that
for the $\bm$-Jacobi field $Y_{i}$ along $\gamma_{x}$ with initial conditions $Y_{i}(0)=e_{i}$ and $Y'_{i}(0)=-A_{u_{x}}Y_{i}(0)$,
there exist some constants $\{a_{ij}\}^{n-1}_{j=1}$ satisfying $Y_{i}=\sum_{j=1}^{n-1} a_{ij}\, Y_{t_{0},j}$.
Since $\theta_{t_{0}}(t_{0})=1$, 
we have $\thetaex'(t_{0},x)/\thetaex(t_{0},x)= \theta'_{t_{0}}(t_{0})$.
Furthermore,
\begin{equation}\label{eq:Basic4}
\theta'_{t_{0}}(t_{0})=\sum _{i=1}^{n-1} g(Y_{t_{0},i}(t_{0}),Y'_{t_{0},i}(t_{0}))= \sum _{i=1}^{n-1} I_{\bm}(Y_{t_{0},i},Y_{t_{0},i}).
\end{equation}
We have $Y_{t_{0},i}(t_{0})=W_{i}(t_{0})$.
Therefore,
Lemma \ref{lem:index form} implies
\begin{equation}\label{eq:Basic5}
\sum _{i=1}^{n-1} I_{\bm}(Y_{t_{0},i},Y_{t_{0},i})\leq \sum _{i=1}^{n-1} I_{\bm}(W_{i},W_{i}).
\end{equation}

We assume $N>n$.
Put $\phi(t):=\Vert W_{i}(t) \Vert(=s_{\kappa,\lambda}(t)/s_{\kappa,\lambda}(t_{0}))$ for $t\in (0,\min \{\tau_{1}(x),\bar{C}_{\kappa,\lambda}\})$.
Note that
we have $\phi'(t)=\Vert W_{i}'(t) \Vert$ for all $t\in (0,\min \{\tau_{1}(x),\bar{C}_{\kappa,\lambda}\})$.
From (\ref{eq:Basic3}), (\ref{eq:Basic4}) and (\ref{eq:Basic5}),
it follows that
\begin{multline*}
\frac{\theta_{f}'(t_{0},x)}{\theta_{f}(t_{0},x)}
\leq    (n-1) \int ^{t_{0}}_{0} \phi'(t)^{2}\,dt-\int ^{t_{0}}_{0} \ric_{g}(\gamma'_{x}(t))\, \phi(t)^{2}\,dt\\
\shoveright{-H_{x}\, \phi(0)^{2}-F'(t_{0})}\\
\setlength{\multlinegap}{0pt}\shoveleft{= (N-1) \int ^{t_{0}}_{0} \phi'(t)^{2}\,dt-\int ^{t_{0}}_{0}\,\ric^{N}_{f}(\gamma'_{x}(t)) \, \phi(t)^{2}\, dt-H_{f,x} \phi(0)^{2}}\\
-(N-n)\int ^{t_{0}}_{0} \phi'(t)^{2}\,dt+\int ^{t_{0}}_{0} \left(F''(t)-\frac{1}{N-n}F'(t)^{2} \right)\,\phi(t)^{2}\,dt\\
+F'(0)\, \phi(0)^{2}-F'(t_{0}).
\end{multline*}
From the curvature assumptions,
we derive
\begin{multline}\label{eq:Basic6}
\frac{\theta_{f}'(t_{0},x)}{\theta_{f}(t_{0},x)} \leq  (N-1)\frac{s_{\kappa,\lambda}'(t_{0})}{s_{\kappa,\lambda}(t_{0})}-(N-n)\int ^{t_{0}}_{0} \phi'(t)^{2}\,dt\\
+\int ^{t_{0}}_{0} \left(F''(t)-\frac{1}{N-n}F'(t)^{2} \right)\,\phi(t)^{2} \,dt+F'(0)\, \phi(0)^{2}-F'(t_{0}).
\end{multline}
By integration by parts,
we have
\begin{equation}\label{eq:Basic7}
\int ^{t_{0}}_{0} F''(t)\,\phi(t)^{2}\,dt=F'(t_{0})-F'(0)\, \phi(0)^{2}-2\int ^{t_{0}}_{0}\, F'(t)\phi'(t)\phi(t)\,dt.
\end{equation}
Furthermore,
for all $t\in (0,t_{0})$,
we have
\begin{align}\label{eq:Basic8}
&(N-n)\phi'(t)^{2}+2F'(t)\phi'(t)\phi(t)+\frac{F'(t)^{2}\phi(t)^{2}}{N-n}\\ \notag
=&\frac{\phi(t)^{2}}{N-n}\left((N-n)\frac{s'_{\kappa,\lambda}(t)}{s_{\kappa,\lambda}(t)}+F'(t)  \right)^{2}\geq 0.
\end{align}
By using (\ref{eq:Basic6}), (\ref{eq:Basic7}) and (\ref{eq:Basic8}),
we obtain (\ref{eq:Basic1}).

We assume $N=n$.
In this case,
$f$ is a constant function;
in particular,
$H_{f,x}=H_{x}$ and $F'(t_{0})=0$.
By Lemma \ref{lem:index form},
we see
\begin{equation*}
\frac{\theta_{f}'(t_{0},x)}{\theta_{f}(t_{0},x)}\leq    (n-1) \int ^{t_{0}}_{0} \phi'(t)^{2}\,dt-\int ^{t_{0}}_{0} \ric^{n}_{f}(\gamma'_{x}(t))\, \phi(t)^{2}\,dt-H_{f,x}\, \phi(0)^{2}.
\end{equation*}
The curvature assumptions imply (\ref{eq:Basic1}).

By (\ref{eq:Basic1}), 
for all $t\in (0,\min \{\tau_{1}(x),\bar{C}_{\kappa,\lambda}\})$,
we have
\begin{equation*}
\frac{d}{dt} \log \frac{s^{N-1}_{\kappa,\lambda}(t)}{\theta_{f}(t,x)}=(N-1)\frac{s_{\kappa,\lambda}'(t)}{s_{\kappa,\lambda}(t)}-\frac{\theta_{f}'(t,x)}{\theta_{f}(t,x)}\geq 0.
\end{equation*}
This implies the inequality (\ref{eq:Basic2}).
\end{proof}
In \cite{HK},
Lemma \ref{lem:Basic comparison} has been proved when $f= 0$ and $N=n$.
\begin{rem}\label{rem:Equality in Basic comparison}
In Lemma \ref{lem:Basic comparison},
choose an orthonormal basis $\{e_{x,i}\}_{i=1}^{n-1}$ of $T_{x}\bm$,
and let $\{Y_{x,i}\}^{n-1}_{i=1}$ be the $\bm$-Jacobi fields along $\gamma_{x}$
with initial conditions $Y_{x,i}(0)=e_{x,i}$ and $Y_{x,i}'(0)=-A_{u_{x}}e_{x,i}$.
Suppose that for some $t_{0}\in (0,\min \{\tau_{1}(x),\bar{C}_{\kappa,\lambda}\})$
the equality in (\ref{eq:Basic1}) holds.
Then the equality in (\ref{eq:Basic5}) also holds.
By Lemma \ref{lem:index form},
for all $i$
we have $Y_{x,i}=s_{\kappa,\lambda}\, E_{x,i}$ on $[0,t_{0}]$,
where $\{E_{x,i}\}^{n-1}_{i=1}$ are the parallel vector fields along $\gamma_{x}$ with initial condition $E_{x,i}(0)=e_{x,i}$.
Moreover,
if $N>n$,
then the equality in (\ref{eq:Basic8}) holds on $[0,t_{0}]$.
This implies $f\circ \gamma_{x}=f(x)-(N-n)\log s_{\kappa,\lambda}$ on $[0,t_{0}]$.
\end{rem}
In the case of $N=\infty$,
we have the following:
\begin{lem}\label{lem:infinite Basic comparison}
Take $x\in \bm$.
Suppose that for all $t\in (0,\tau_{1}(x))$
we have $\ric^{\infty}_{f}(\gamma'_{x}(t)) \geq 0$,
and suppose $H_{f,x}\geq 0$.
Then for all $t\in (0,\tau_{1}(x))$,
we have $\theta_{f}'(t,x)\leq 0$.
In particular,
for all $s,t\in [0,\tau_{1}(x))$ with $s\leq t$,
we have $\theta_{f}(t,x)\leq \theta_{f}(s,x)$.
\end{lem}
\begin{proof}
Let $F:=f\circ \gamma_{x}$.
Choose an orthonormal basis $\{ e_{i} \}_{i=1}^{n-1}$ of $T_{x}\partial M$.
For each $i$,
let $E_{i}$ denote the parallel vector field along $\gamma_{x}$
with initial condition $E_{i}(0)=e_{i}$. 
Put $\phi(t):=\Vert E_{i}(t) \Vert(=1)$ for $t\in (0,\tau_{1}(x))$.
Fix $t_{0}\in (0,\tau_{1}(x))$.
By Lemma \ref{lem:index form},
we see
\begin{multline*}
\frac{\theta_{f}'(t_{0},x)}{\theta_{f}(t_{0},x)}
\leq    -\int ^{t_{0}}_{0} \left(\ric^{\infty}_{f}(\gamma'_{x}(t))-F''(t) \right)\, \phi(t)^{2}\,dt\\
-  \left(H_{f,x}-F'(0)  \right) \phi(0)^{2}-F'(t_{0}).
\end{multline*}
By the curvature assumptions,
and by integration by parts,
we have
\begin{equation*}
\theta_{f}'(t_{0},x)\leq \theta_{f}(t_{0},x) \left(\int ^{t_{0}}_{0} F''(t)\,\phi(t)^{2}\,dt+F'(0)\, \phi(0)^{2}-F'(t_{0})\right)=0.
\end{equation*}
This proves the lemma.
\end{proof}
\begin{rem}\label{rem:Equality in infinite Basic comparison}
In Lemma \ref{lem:infinite Basic comparison},
choose an orthonormal basis $\{e_{x,i}\}_{i=1}^{n-1}$ of $T_{x}\bm$,
and let $\{Y_{x,i}\}^{n-1}_{i=1}$ be the $\bm$-Jacobi fields along $\gamma_{x}$
with initial conditions $Y_{x,i}(0)=e_{x,i}$ and $Y_{x,i}'(0)=-A_{u_{x}}e_{x,i}$.
Suppose that for some $t_{0}\in (0,\tau_{1}(x))$
we have $\theta_{f}'(t_{0},x)=0$.
By Lemma \ref{lem:index form},
for all $i$
we have $Y_{x,i}=E_{x,i}$ on $[0,t_{0}]$,
where $\{E_{x,i}\}^{n-1}_{i=1}$ are the parallel vector fields along $\gamma_{x}$ with initial condition $E_{x,i}(0)=e_{x,i}$.
\end{rem}
\subsection{Laplacian comparisons}
Combining Lemma \ref{lem:Basic comparison} and (\ref{eq:Laplacian representation}),
we have the following Laplacian comparison result:
\begin{lem}\label{lem:Laplacian comparison}
Take $x\in \bm$.
For $N\in [n,\infty)$,
we suppose that for all $t\in (0,\tau(x))$
we have $\ric^{N}_{f}(\gamma'_{x}(t)) \geq (N-1)\kappa$,
and suppose $H_{f,x}\geq (N-1)\lambda$.
Then for all $t\in (0,\tau(x))$
we have
\begin{equation*}\label{eq:Laplacian comparison}
\Delta_{f}\, \rho_{\bm} (\gamma_{x}(t)) \geq -(N-1)\frac{s_{\kappa,\lambda}'(t)}{s_{\kappa,\lambda}(t)}.
\end{equation*}
\end{lem}
In \cite{K2},
Lemma \ref{lem:Laplacian comparison} has been proved when $f= 0$ and $N=n$.

In the case of $N=\infty$,
by using Lemma \ref{lem:infinite Basic comparison} and (\ref{eq:Laplacian representation}),
we have:
\begin{lem}\label{lem:infinite Laplacian comparison}
Take $x\in \bm$.
Suppose that for all $t\in (0,\tau(x))$
we have $\ric^{\infty}_{f}(\gamma'_{x}(t)) \geq 0$,
and suppose $H_{f,x}\geq 0$.
Then for all $t\in (0,\tau(x))$
we have $\Delta_{f} \rho_{\bm} (\gamma_{x}(t)) \geq 0$.
\end{lem}
\begin{rem}\label{rem:Equality in Laplacian comparison}
The equality case in Lemma \ref{lem:Laplacian comparison} (resp. \ref{lem:infinite Laplacian comparison})
results into that in Lemma \ref{lem:Basic comparison} (resp. \ref{lem:infinite Basic comparison}) (see Remarks \ref{rem:Equality in Basic comparison} and \ref{rem:Equality in infinite Basic comparison}).
\end{rem}
\subsection{Distributions}\label{sec:Distributions}
From Lemma \ref{lem:Laplacian comparison},
we derive the following:
\begin{lem}\label{lem:p-Laplacian comparison}
Take $x\in \bm$.
Let $p\in (1,\infty)$.
For $N\in [n,\infty)$,
we suppose that for all $t\in (0,\tau(x))$
we have $\ric^{N}_{f}(\gamma'_{x}(t)) \geq (N-1)\kappa$,
and suppose $H_{f,x}\geq (N-1)\lambda$.
Let $\phi:[0,\infty)\to \mathbb{R}$ be a monotone increasing smooth function.
Then for all $t\in (0,\tau(x))$ we have
\begin{equation}\label{eq:p-Laplacian comparison}
\Delta_{f,p}\, (\phi \circ \rho_{\bm}) (\gamma_{x}(t)) \geq -    \left(     \left(  \phi'  \right)^{p-1}     \right)'    (t)-(N-1)\frac{s'_{\kappa,\lambda}(t)}{s_{\kappa,\lambda}(t)}\,\phi'(t)^{p-1}.
\end{equation}
\end{lem}
\begin{proof}
By straightforward computations,
for all $t\in (0,\tau(x))$
\begin{equation*}
\Delta_{f,p}\, (\phi \circ \rho_{\bm}) (\gamma_{x}(t))=- \left(     \left(  \phi'  \right)^{p-1}     \right)'(t)  +\Delta_{f,2}\, \rho_{\bm}(\gamma_{x}(t))\, \phi'(t)^{p-1}.
\end{equation*}
This together with Lemma \ref{lem:Laplacian comparison},
we obtain (\ref{eq:p-Laplacian comparison}).
\end{proof}
In the case of $N=\infty$,
we have:
\begin{lem}\label{lem:infinite p-Laplacian comparison}
Take $x\in \bm$.
Let $p\in (1,\infty)$.
Suppose that for all $t\in (0,\tau(x))$
we have $\ric^{\infty}_{f}(\gamma'_{x}(t)) \geq 0$,
and suppose $H_{f,x}\geq 0$.
Let $\phi:[0,\infty)\to \mathbb{R}$ be a monotone increasing smooth function.
Then for all $t\in (0,\tau(x))$
\begin{equation}\label{eq:infinite p-Laplacian comparison}
\Delta_{f,p} (\phi \circ \rho_{\bm}) (\gamma_{x}(t)) \geq -\left(     \left(  \phi'  \right)^{p-1}     \right)'(t).
\end{equation}
\end{lem}
\begin{proof}
For all $t\in (0,\tau(x))$,
we have
\begin{equation*}
\Delta_{f,p}\, (\phi \circ \rho_{\bm}) (\gamma_{x}(t))=- \left(     \left(  \phi'  \right)^{p-1}     \right)'(t)  +\Delta_{f,2}\, \rho_{\bm}(\gamma_{x}(t))\, \phi'(t)^{p-1}.
\end{equation*}
Lemma \ref{lem:infinite Laplacian comparison} implies (\ref{eq:infinite p-Laplacian comparison}).
\end{proof}
\begin{rem}\label{rem:Equality in p-Laplacian comparison}
The equality case in Lemma \ref{lem:p-Laplacian comparison} (resp. \ref{lem:infinite p-Laplacian comparison})
results into that in Lemma \ref{lem:Laplacian comparison} (resp. \ref{lem:infinite Laplacian comparison})
(see Remarks \ref{rem:Equality in Basic comparison}, \ref{rem:Equality in infinite Basic comparison} and \ref{rem:Equality in Laplacian comparison}).
\end{rem}
By Lemma \ref{lem:p-Laplacian comparison},
we have the following:
\begin{prop}\label{prop:global p-Laplacian comparison}
Let $p\in (1,\infty)$.
For $N\in [n,\infty)$,
we suppose $\ric^{N}_{f,M} \geq (N-1)\kappa$ and $H_{f,\bm}\geq (N-1)\lambda$.
For a monotone increasing smooth function $\phi:[0,\infty)\to \mathbb{R}$,
we put $\Phi:=\phi \circ \rho_{\bm}$.
Then we have
\begin{equation*}
\Delta_{f,p}\,\Phi \geq \left( -\left( \left(\phi' \right)^{p-1} \right)' -(N-1)\frac{s'_{\kappa,\lambda}}{s_{\kappa,\lambda}}\left(\phi' \right)^{p-1}\right)\circ \rho_{\bm}
\end{equation*}
in a distribution sense on $M$.
More precisely,
for every non-negative smooth function $\psi:M\to \mathbb{R}$ whose support is compact and contained in $\inte M$,
we have
\begin{multline}\label{eq:global p-Laplacian comparison}
\int_{M}\,   \Vert \nabla \Phi \Vert^{p-2} g\left(\nabla \psi, \nabla \Phi  \right)\, d\,m_{f}\\
\geq \int_{M}\,\psi\,  \left(\left( -\left( \left(\phi' \right)^{p-1} \right)' -(N-1)\frac{s'_{\kappa,\lambda}}{s_{\kappa,\lambda}}\left(\phi' \right)^{p-1}\right)\circ \rho_{\bm}\right)   \,d\,m_{f}.
\end{multline}
\end{prop}
\begin{proof}
By Lemma \ref{lem:avoiding the cut locus2},
there exists a sequence $\{\Omega_{k}\}_{k\in \mathbb{N}}$ of closed subsets of $M$
satisfying that
for every $k$,
the set $\partial \Omega_{k}$ is a smooth hypersurface in $M$,
and satisfying the following:
(1) for all $k_{1},k_{2}$ with $k_{1}<k_{2}$,
we have $\Omega_{k_{1}}\subset \Omega_{k_{2}}$;
(2) $M\setminus \cut \bm=\bigcup_{k}\,\Omega_{k}$;
(3) $\partial \Omega_{k}\cap \bm=\bm$ for all $k$;
(4) for each $k$,
     on $\partial \Omega_{k}\setminus \bm$,
     there exists the unit outer normal vector field $\nu_{k}$ for $\Omega_{k}$ with $g(\nu_{k},\nabla \rho_{\bm})\geq 0$.

For the canonical Riemannian volume measure $\vol_{k}$ on $\partial \Omega_{k}\setminus \bm$,
put $m_{f,k}:=e^{  -f|_{\partial \Omega_{k}\setminus \bm}}\,\vol_{k}$.
Let $\psi:M\to \mathbb{R}$ be a non-negative smooth function whose support is compact and contained in $\inte M$.
By the Green formula,
and by $\partial \Omega_{k}\cap \bm=\bm$,
we have
\begin{align*}
&\quad \,  \int_{\Omega_{k}}\,  \Vert \nabla \Phi \Vert^{p-2} g\left(\nabla \psi, \nabla \Phi  \right)  \, d\,m_{f}\\
&=\int_{\Omega_{k}}\, \left(-\psi \, g\left(\nabla \left(\Vert \nabla \Phi \Vert^{p-2}\right), \nabla \Phi \right)+\Vert \nabla \Phi \Vert^{p-2}\,\psi\,\Delta_{f,2} \Phi \right) \,d\,m_{f}\\
&\qquad \qquad \qquad \qquad \qquad \qquad+\int_{\partial \Omega_{k}\setminus \bm}\, \Vert \nabla \Phi \Vert^{p-2}\, \psi \,g\left(\nu_{k},\nabla \Phi \right) \, d\,m_{f,k}\\
&=\int_{\Omega_{k}}\, \psi\,\Delta_{f,p} \Phi \,d\,m_{f}+\int_{\partial \Omega_{k}\setminus \bm}\, \Vert \nabla \Phi \Vert^{p-2}\, \psi \,g\left(\nu_{k},\nabla \Phi \right) \, d\,m_{f,k}.
\end{align*}
Lemma \ref{lem:p-Laplacian comparison} and $g(\nu_{k},\nabla \rho_{\bm})\geq 0$ imply
\begin{multline*}
\int_{\Omega_{k}}\,   \Vert \nabla \Phi \Vert^{p-2} g\left(\nabla \psi, \nabla \Phi  \right)\, d\,m_{f}\\
\geq \int_{\Omega_{k}}\,\psi\,  \left(\left( -\left( \left(\phi' \right)^{p-1} \right)' -(N-1)\frac{s'_{\kappa,\lambda}}{s_{\kappa,\lambda}}\left(\phi' \right)^{p-1}\right)\circ \rho_{\bm}\right)   \,d\,m_{f}.
\end{multline*}
Letting $k\to \infty$,
we obtain the desired inequality.
\end{proof}
In the case of $N=\infty$,
we have:
\begin{prop}\label{prop:global infinite p-Laplacian comparison}
Let $p\in (1,\infty)$.
Suppose $\ric^{\infty}_{f,M} \geq 0$ and $H_{f,\bm}\geq 0$.
For a monotone increasing smooth function $\phi:[0,\infty)\to \mathbb{R}$,
put $\Phi:=\phi \circ \rho_{\bm}$.
Then we have
\begin{equation*}
\Delta_{f,p}\,\Phi \geq  -\left( \left(\phi' \right)^{p-1} \right)' \circ \rho_{\bm}
\end{equation*}
in a distribution sense on $M$.
More precisely,
for every non-negative smooth function $\psi:M\to \mathbb{R}$ whose support is compact and contained in $\inte M$,
we have
\begin{equation}\label{eq:global infinite p-Laplacian comparison}
\int_{M}\,   \Vert \nabla \Phi \Vert^{p-2} g\left(\nabla \psi, \nabla \Phi  \right)\, d\,m_{f}\\
\geq \int_{M}\,\psi \left( -\left( \left(\phi' \right)^{p-1} \right)' \circ \rho_{\bm}\right)   \,d\,m_{f}.
\end{equation}
\end{prop}
\begin{proof}
Lemma \ref{lem:avoiding the cut locus2} implies that
there exists a sequence $\{\Omega_{k}\}_{k\in \mathbb{N}}$ of closed subsets of $M$
satisfying that
for every $k$,
the set $\partial \Omega_{k}$ is a smooth hypersurface in $M$,
and satisfying the following:
(1) for all $k_{1},k_{2}\in \mathbb{N}$ with $k_{1}<k_{2}$,
we have $\Omega_{k_{1}}\subset \Omega_{k_{2}}$;
(2) $M\setminus \cut \bm=\bigcup_{k}\,\Omega_{k}$;
(3) $\partial \Omega_{k}\cap \bm=\bm$ for all $k$;
(4) for each $k$,
on $\partial \Omega_{k}\setminus \bm$,
there exists the unit outer normal vector field $\nu_{k}$ for $\Omega_{k}$ with $g(\nu_{k},\nabla \rho_{\bm})\geq 0$.

For the canonical Riemannian volume measure $\vol_{k}$ on $\partial \Omega_{k}\setminus \bm$,
put $m_{f,k}:=e^{-f|_{\partial \Omega_{k}\setminus \bm}}\,\vol_{k}$.
Let $\psi:M\to \mathbb{R}$ be a non-negative smooth function whose support is compact and contained in $\inte M$.
By the Green formula,
and by $\partial \Omega_{k}\cap \bm=\bm$,
we see
\begin{multline*}
\int_{\Omega_{k}}\,  \Vert \nabla \Phi \Vert^{p-2} g\left(\nabla \psi, \nabla \Phi  \right)  \, d\,m_{f}\\
=\int_{\Omega_{k}}\, \psi\,\Delta_{f,p} \Phi \,d\,m_{f}+\int_{\partial \Omega_{k}\setminus \bm}\, \Vert \nabla \Phi \Vert^{p-2}\, \psi \,g\left(\nu_{k},\nabla \Phi \right) \, d\,m_{f,k}.
\end{multline*}
By Lemma \ref{lem:p-Laplacian comparison} and $g(\nu_{k},\nabla \rho_{\bm})\geq 0$,
\begin{equation*}\label{eq:weak global p-Laplacian comparison}
\int_{\Omega_{k}}\,   \Vert \nabla \Phi \Vert^{p-2} g\left(\nabla \psi, \nabla \Phi  \right)\, d\,m_{f}
\geq \int_{\Omega_{k}}\,\psi\,  \left( -\left( \left(\phi' \right)^{p-1} \right)' \circ \rho_{\bm}\right)\,d\,m_{f}.
\end{equation*}
By letting $k \to \infty$,
we complete the proof.
\end{proof}
\begin{rem}\label{rem:Equality case in global p-Laplacian comparison}
In Proposition \ref{prop:global p-Laplacian comparison} (resp. \ref{prop:global infinite p-Laplacian comparison}),
assume that
the equality in (\ref{eq:global p-Laplacian comparison}) (resp. (\ref{eq:global infinite p-Laplacian comparison})) holds.
In this case,
for a fixed $x\in \bm$ we see that
for every $t\in (0,\tau(x))$
the equality in (\ref{eq:p-Laplacian comparison}) (resp. (\ref{eq:infinite p-Laplacian comparison})) also holds.
The equality case in Proposition \ref{prop:global p-Laplacian comparison} (resp. \ref{prop:global infinite p-Laplacian comparison})
results into that in Lemma \ref{lem:p-Laplacian comparison} (resp. \ref{lem:infinite p-Laplacian comparison}) (see Remark \ref{rem:Equality in p-Laplacian comparison}).
\end{rem}
\begin{rem}
Perales \cite{P} has proved a Laplacian comparison inequality for the distance function from the boundary in a barrier sense
for manifolds with boundary of non-negative Ricci curvature.
We can prove that
the Laplacian comparison inequalities for $\rho_{\bm}$ in Lemmas \ref{lem:Laplacian comparison} and \ref{lem:infinite Laplacian comparison} globally hold on $M$ in a barrier sense.
\end{rem}

\section{Inscribed radius rigidity}\label{sec:Inscribed radius rigidity}
Let $M$ be an $n$-dimensional, 
connected complete Riemannian manifold with boundary,
and let $f:M\to \mathbb{R}$ be a smooth function.
\subsection{Inscribed radius comparison}
From Lemma \ref{lem:Basic comparison},
we derive the following comparison result for the inscribed radius.
\begin{lem}\label{lem:Inscribed radius comparison}
Let $\kappa\in \mathbb{R}$ and $\lambda \in \mathbb{R}$ satisfy the ball-condition.
For $N\in [n,\infty)$,
we suppose $\ric^{N}_{f,M}\geq (N-1)\kappa$ and $H_{f,\bm}\geq (N-1)\lambda$.
Then $\dm \leq \const$.
\end{lem}
\begin{proof}
Take $x\in \bm$.
We suppose $\const<\tau_{1}(x)$.
By Lemma \ref{lem:Basic comparison},
for all $t\in [0,\const)$
we have $\theta_{f}(t,x)\leq e^{-f(x)}s^{N-1}_{\kappa,\lambda}(t)$.
Letting $t\to \const$, 
we have $\theta(\const,x)=0$;
in particular,
$\gamma_{x}(\const)$ is a conjugate point of $\bm$ along $\gamma_{x}$.
This is a contradiction.
Hence,
we have $\tau_{1}(x)\leq \const$.
The relationship between $\tau$ and $\tau_{1}$ implies $\tau(x)\leq \const$.
Since $\dm$ is equal to $\sup_{x\in \bm}\, \tau(x)$, 
we have $\dm \leq \const$.
\end{proof}
In \cite{K3},
Lemma \ref{lem:Inscribed radius comparison} has been proved when $f=0$ and $N=n$.

\subsection{Inscribed radius rigidity}
Now,
we prove Theorem \ref{thm:Ball rigidity}.
\begin{proof}[Proof of Theorem \ref{thm:Ball rigidity}]
Let $\kappa\in \mathbb{R}$ and $\lambda \in \mathbb{R}$ satisfy the ball-condition.
For $N\in [n,\infty)$,
we suppose $\ric^{N}_{f,M}\geq (N-1)\kappa$ and $H_{f,\bm}\geq (N-1)\lambda$.
By Lemma \ref{lem:Inscribed radius comparison},
we have $\dm \leq \const$.

Take $p_{0}\in M$ satisfying $\rho_{\bm}(p_{0})=\const$.
We put 
\begin{equation*}
\Omega:=\{p\in \inte M \setminus \{p_{0}\} \,|\, \rho_{\bm}(p)+\rho_{p_{0}}(p)=\const\}.
\end{equation*}
Take a foot point $x_{p_{0}}$ on $\bm$ of $p_{0}$,
and the normal minimal geodesic $\gamma_{0}:[0,\const]\to M$ from $x_{p_{0}}$ to $p_{0}$.
Then for all $t\in (0,\const)$,
we have $\gamma_{0}(t)\in \Omega$.
Therefore,
$\Omega$ is a non-empty closed subset of $\inte M \setminus \{p_{0}\}$.

We prove that
$\Omega$ is an open subset of $\inte M \setminus \{p_{0}\}$.
Fix $p\in \Omega$,
and take a foot point $x_{p}$ on $\bm$ of $p$.
Note that
$x_{p}$ is also a foot point on $\bm$ of $p_{0}$.
We take the normal minimal geodesic $\gamma:[0,\const]\to M$ from $x_{p}$ to $p_{0}$.
Then $\gamma|_{(0,\const)}$ passes through $p$.
There exists an open neighborhood $U$ of $p$ such that $\rho_{p_{0}}$ and $\rho_{\bm}$ are smooth on $U$,
and for every $q\in U$
there exists a unique normal minimal geodesic in $M$ from $p_{0}$ to $q$ that lies in $\inte M$.
By Lemmas \ref{lem:finite pointed Laplacian comparison} and \ref{lem:Laplacian comparison},
for all $q\in U$
\begin{align*}
\frac{\Delta_{f} (\rho_{\bm}+\rho_{p_{0}})(q)}{N-1}&\geq -\left(\frac{\lambda \ck(\rho_{\bm}(q))-\kappa \sk(\rho_{\bm}(q))}{\ck(\rho_{\bm}(q))+\lambda \sk(\rho_{\bm}(q))}+\frac{\ck(\rho_{p_{0}}(q))}{\sk(\rho_{p_{0}}(q))}\right) \\
                                               &=    -\frac{s_{\kappa,\lambda}(\rho_{\bm}(q)+\rho_{p_{0}}(q))}{s_{\kappa,\lambda}(\rho_{\bm}(q)) \sk(\rho_{p_{0}}(q))} \geq 0.
\end{align*}
Lemma \ref{lem:maximal principle} implies $U\subset \Omega$.
We prove the openness of $\Omega$.

Since $\inte M\setminus \{p_{0}\}$ is connected, 
we have $\Omega=\inte M\setminus \{p_{0}\}$,
and hence $\rho_{\bm}+\rho_{p_{0}}=\const$ on $M$.
This implies $M=B_{\const}(p_{0})$ and $\bm=\partial B_{\const}(p_{0})$.
Furthermore,
we see that the cut locus for $p_{0}$ is empty,
and the equality in (\ref{eq:finite pointed Laplacian comparison}) holds on $\inte M\setminus \{p_{0}\}$.
For each $u\in U_{p_{0}}M$,
choose an orthonormal basis $\{e_{u,i}\}_{i=1}^{n}$ of $T_{p_{0}}M$ with $e_{n}=u$.
Let $\{Y_{u,i}\}^{n-1}_{i=1}$ be the Jacobi fields along $\gamma_{u}$
with initial conditions $Y_{u,i}(0)=0$ and $Y_{u,i}'(0)=e_{u,i}$,
where $\gamma_{u}:[0,\const]\to M$ is the normal geodesic with $\gamma_{u}(0)=p_{0}$ and $\gamma_{u}'(0)=u$.
Then for all $i$
we have $Y_{u,i}=s_{\kappa}\, E_{u,i}$ on $[0,\const]$,
where $\{E_{u,i}\}^{n-1}_{i=1}$ are the parallel vector fields along $\gamma_{u}$ with initial condition $E_{u,i}(0)=e_{u,i}$ (see Remark \ref{rem:equality pointed Laplacian comparison}).
Let $\tilde{p}_{0}$ denote the center point of $\ball$.
Choose a linear isometry $I:T_{p_{0}}M\to T_{\tilde{p}_{0}}\ball$.
Define a map $\Phi:M\to \ball$ by $\Phi(p):=\exp_{\tilde{p}_{0}}\circ I\circ \exp^{-1}_{p_{0}}(p)$,
where $\exp_{p_{0}}$ and $\exp_{\tilde{p}_{0}}$ are the exponential maps at $p_{0}$ and at $\tilde{p}_{0}$,
respectively.
For every $p\in \inte M$
the differential map $D(\Phi|_{\inte M})_{p}$ of $\Phi|_{\inte M}$ at $p$
sends an orthonormal basis of $T_{p}M$ to that of $T_{\Phi(p)} \ball$,
and for every $x\in \bm$
the map $D(\Phi|_{\bm})_{x}$ sends an orthonormal basis of $T_{x}\bm$ to that of $T_{\Phi(x)}\partial \ball$.
Hence,
$\Phi$ is a Riemannian isometry with boundary from $M$ to $\ball$,
and $(M,d_{M})$ is isometric to $(\ball, d_{\ball})$.

Now,
the equality in Lemma \ref{lem:Laplacian comparison} holds on $\inte M\setminus \{p_{0}\}$.
For each $x\in \bm$
we see $f\circ \gamma_{x}=f(x)-(N-n)\log s_{\kappa,\lambda}$ on $[0,\const]$ (see Remark \ref{rem:Equality in Laplacian comparison}).
If we suppose $N>n$,
then $f(\gamma_{x}(t))$ tends to infinity as $t\to \const$.
This is a contradiction
since $f(\gamma_{x}(\const))=f(p_{0})$.
Hence,
we obtain $N=n$.
We complete the proof of Theorem \ref{thm:Ball rigidity}.
\end{proof}

\section{Volume comparisons}\label{sec:Volume comparisons}
Let $M$ be an $n$-dimensional, 
connected complete Riemannian manifold with boundary with Riemannian metric $g$,
and let $f:M\to \mathbb{R}$ be a smooth function.
\subsection{Absolute volume comparisons}\label{sec:Absolute volume comparison}
Let $\bar{\theta}_{f}:[0,\infty) \times \bm \to \mathbb{R}$ be a function defined by
\begin{equation*}
  \bar{\theta}_{f}(t,x) := \begin{cases}
                            \theta_{f}(t,x) & \text{if $t< \tau(x)$}, \\
                            0           & \text{if $t\geq \tau(x)$}.
                       \end{cases}
\end{equation*}

By the coarea formula (see e.g., Theorem 3.2.3 in \cite{F}),
we show:
\begin{lem}\label{lem:Basic volume lemma}
Suppose that
$\bm$ is compact.
Then for all $r\in (0,\infty)$
\begin{equation}\label{eq:volume comparison1}
m_{f}( B_{r}(\bm))=\int_{\bm} \int^{r}_{0}\bar{\theta}_{f}(t,x)\,dt\,d\vol_{h},
\end{equation}
where $h$ is the induced Riemannian metric on $\bm$.
\end{lem}
\begin{proof}
Since $\bm$ is compact,
$B_{r}(\bm)$ is also compact;
in particular,
$m_{f}(B_{r}(\bm))<\infty$.
From Proposition \ref{prop:metric neighborhood},
we derive
\begin{equation*}
B_{r}(\bm)=\expp \left(\bigcup_{x\in \bm} \{tu_{x}\mid t\in [0,\min\{r,\tau(x)\}] \}\right).
\end{equation*}
The map $\expp$ is diffeomorphic on $\bigcup_{x\in \bm} \{tu_{x}\mid t\in (0,\min\{r,\tau(x)\}) \}$.
Furthermore,
the cut locus $\cut \bm$ for the boundary is a null set of $M$.
Hence,
the coarea formula and the Fubini theorem imply (\ref{eq:volume comparison1}).
\end{proof}
Bayle \cite{B} has stated the following absolute volume comparison inequality of Heintze-Karcher type without proof (see Theorem E.2.2 in \cite{B}, and also \cite{M}).
\begin{lem}[\cite{B}]\label{lem:absolute volume comparison}
Suppose that
$\bm$ is compact.
For $N\in [n,\infty)$,
we suppose $\ric^{N}_{f,M}\geq (N-1)\kappa$ and $H_{f,\bm} \geq (N-1)\lambda$.
Then for all $r\in (0,\infty)$
\begin{equation}\label{eq:absolute volume comparison1}
m_{f}(B_{r}(\bm))\leq s_{N,\kappa,\lambda}(r)\, m_{f,\bm}(\bm);
\end{equation}
in particular,
we have $(\ref{eq:volume growth sup})$.
\end{lem}
\begin{proof}
Fix $r\in (0,\infty)$.
By Lemma \ref{lem:Basic comparison},
for all $x\in \bm$ and $t\in (0,r)$,
we have $\bar{\theta}_{f}(t,x)\leq \bar{s}^{N-1}_{\kappa,\lambda}(t)\,\bar{\theta}_{f}(0,x)$.
Integrate the both sides of the inequality over $(0,r)$ with respect to $t$,
and then do that over $\bm$ with respect to $x$.
By Lemma \ref{lem:Basic volume lemma},
we have (\ref{eq:absolute volume comparison1}).
\end{proof}
Lemma \ref{lem:absolute volume comparison} has been proved in \cite{HK} when $f=0$ and $N=n$.

In the case of $N=\infty$,
Morgan \cite{Mo} has shown the following volume comparison inequality (see Theorem 2 in \cite{Mo}, and also \cite{M}).
\begin{lem}[\cite{Mo}]\label{lem:infinite absolute volume comparison}
Suppose that
$\bm$ is compact.
Suppose $\ric^{\infty}_{f,M}\geq 0$ and $H_{f,\bm} \geq 0$.
Then for all $r\in (0,\infty)$
\begin{equation}\label{eq:infinite absolute volume comparison}
m_{f}(B_{r}(\bm))\leq r\, m_{f,\bm}(\bm);
\end{equation}
in particular,
we have $(\ref{eq:infinite volume growth sup})$.
\end{lem}
\begin{proof}
Fix $r\in (0,\infty)$.
By Lemma \ref{lem:infinite Basic comparison},
for all $x\in \bm$ and $t\in (0,r)$,
we have $\bar{\theta}_{f}(t,x)\leq \bar{\theta}_{f}(0,x)$.
Integrate the both sides of the inequality over $(0,r)$ with respect to $t$,
and then do that over $\bm$ with respect to $x$.
Lemma \ref{lem:Basic volume lemma} implies the lemma.
\end{proof}
\begin{rem}\label{rem:infinite absolute volume eq}
In Lemma \ref{lem:absolute volume comparison} (resp. \ref{lem:infinite absolute volume comparison}),
assume that
for some $r>0$
the equality in (\ref{eq:absolute volume comparison1}) (resp. (\ref{eq:infinite absolute volume comparison})) holds.
For each $x\in \bm$,
choose an orthonormal basis $\{ e_{x,i} \}_{i=1}^{n-1}$ of $T_{x}\bm$.
Let $\{Y_{x,i}\}^{n-1}_{i=1}$ be the $\bm$-Jacobi fields along $\gamma_{x}$
with initial conditions $Y_{x,i}(0)=e_{x,i}$ and $Y_{x,i}'(0)=-A_{u_{x}}e_{x,i}$.
Then for all $i$
we see $Y_{x,i}=s_{\kappa,\lambda}\,E_{x,i}$ (resp. $Y_{x,i}=E_{x,i}$) on $[0,\min \{r,\bar{C}_{\kappa,\lambda}\}]$ (resp. $[0,r]$),
where $\{E_{x,i}\}^{n-1}_{i=1}$ are the parallel vector fields along $\gamma_{x}$ with initial condition $E_{x,i}(0)=e_{x,i}$.
Moreover,
$f\circ \gamma_{x}=f(x)-(N-n)\log s_{\kappa,\lambda}$ on $[0,\min \{r,\bar{C}_{\kappa,\lambda}\}]$ (cf. Remarks \ref{rem:Equality in Basic comparison} and \ref{rem:Equality in infinite Basic comparison}).
\end{rem}
\subsection{Relative volume comparison}\label{sec:Relative volume comparison}
We have the following relative volume comparison theorem of Bishop-Gromov type:
\begin{thm}\label{thm:volume comparison}
Let $M$ be an $n$-dimensional, 
connected complete Riemannian manifold with boundary,
and let $f:M\to \mathbb{R}$ be a smooth function.
Suppose that
$\bm$ is compact.
For $N\in [n,\infty)$,
we suppose $\ric^{N}_{f,M}\geq (N-1)\kappa$ and $H_{f,\bm}\geq (N-1)\lambda$.
Then for all $r,R\in (0,\infty)$ with $r\leq R$,
we have
\begin{equation}\label{eq:volume comparison}
\frac{m_{f} (B_{R}(\bm))}{m_{f}(B_{r}(\bm))}\leq \frac{s_{N,\kappa,\lambda}(R)}{s_{N,\kappa,\lambda}(r)}.
\end{equation}
\end{thm}
\begin{proof}
Lemma \ref{lem:Basic comparison} implies that
for all $s,t\in [0,\infty)$ with $s\leq t$,
\begin{equation}\label{eq:volume comparison2}
\bar{\theta}_{f}(t,x)\; \bar{s}_{\kappa,\lambda}^{N-1}(s) \leq \bar{\theta}_{f}(s,x)\; \bar{s}_{\kappa,\lambda}^{N-1}(t).
\end{equation}
By integrating the both sides of (\ref{eq:volume comparison2}) over $[0,r]$ with respect to $s$,
and then doing that over $[r,R]$ with respect to $t$,
we conclude
\begin{equation*}\label{eq:volume comparison3}
\frac{\int^{R}_{r}\bar{\theta}_{f}(t,x)\,dt}{\int^{r}_{0}\bar{\theta}_{f}(s,x)\,ds}\leq \frac{s_{N,\kappa,\lambda}(R)-s_{N,\kappa,\lambda}(r)}{s_{N,\kappa,\lambda}(r)}.
\end{equation*}
From Lemma \ref{lem:Basic volume lemma},
we derive
\begin{align*}
          \frac{m_{f}( B_{R}(\bm))}{m_{f}(B_{r}(\bm))}
&  =   1+\frac{\int_{\bm} \int^{R}_{r}\bar{\theta}_{f}(t,x)\,dt\,d\vol_{h}}{\int_{\bm} \int^{r}_{0}\bar{\theta}_{f}(s,x)\,ds\,d\vol_{h}}\\
&\leq 1+\frac{s_{N,\kappa,\lambda}(R)-s_{N,\kappa,\lambda}(r)}{s_{N,\kappa,\lambda}(r)}
    =      \frac{s_{N,\kappa,\lambda}(R)}{s_{N,\kappa,\lambda}(r)}.
\end{align*}
This proves the theorem.
\end{proof}
In \cite{Sa},
Theorem \ref{thm:volume comparison} has been proved when $f=0$ and $N=n$.

In the case of $N=\infty$,
we have:
\begin{thm}\label{thm:infinite volume comparison}
Let $M$ be a connected complete Riemannian manifold with boundary,
and let $f:M\to \mathbb{R}$ be a smooth function.
Suppose that
$\bm$ is compact.
Suppose $\ric^{\infty}_{f,M}\geq 0$ and $H_{f,\bm}\geq 0$.
Then for all $r,R\in (0,\infty)$ with $r\leq R$,
we have
\begin{equation}\label{eq:infinite volume comparison}
\frac{m_{f} (B_{R}(\bm))}{m_{f}(B_{r}(\bm))}\leq \frac{R}{r}.
\end{equation}
\end{thm}
\begin{proof}
By Lemma \ref{lem:infinite Basic comparison}, 
for all $s,t\in [0,\infty)$ with $s\leq t$,
we have $\bar{\theta}_{f}(t,x) \leq \bar{\theta}_{f}(s,x)$.
Integrating the both sides over $[0,r]$ with respect to $s$,
and then doing that over $[r,R]$ with respect to $t$,
we see 
\begin{equation*}
r\,\int^{R}_{r}\bar{\theta}_{f}(t,x)\,dt \leq (R-r)\,\int^{r}_{0}\bar{\theta}_{f}(s,x)\,ds.
\end{equation*}
By Lemma \ref{lem:Basic volume lemma},
we complete the proof.
\end{proof}
\begin{rem}
In \cite{Sa},
the author has proved a measure contraction inequality around the boundary when $f=0$ and $N=n$.
We can prove similar measure contraction inequalities in our setting.
The measure contraction inequalities enable us to give another proof of Theorem \ref{thm:volume comparison},
and of Theorem \ref{thm:infinite volume comparison}.
\end{rem}

\subsection{Volume growth rigidity}\label{sec:Volume growth rigidity}
We have the following lemma:
\begin{lem}\label{lem:half}
Suppose that
$\bm$ is compact.
For $N\in [n,\infty)$, 
we suppose $\ric^{N}_{f,M}\geq (N-1)\kappa$ and $H_{f,\bm} \geq (N-1)\lambda$.
Assume that
there exists $R\in (0,\bar{C}_{\kappa,\lambda}]\setminus\{\infty\}$ such that
for every $r\in (0,R]$
the equality in $(\ref{eq:volume comparison})$ holds.
Then we have $\tau\geq R$ on $\bm$.
\end{lem}
\begin{proof}
The proof is by contradiction.
Suppose that
a point $x_{0}\in \bm$ satisfies $\tau(x_{0})<R$.
Put $t_{0}:=\tau(x_{0})$,
and take $\epsilon>0$ satisfying $t_{0}+\epsilon<R$.
By the continuity of $\tau$,
there exists a closed geodesic ball $B$ in $\bm$ centered at $x_{0}$ such that
for all $x\in B$ we have $\tau(x)\leq t_{0}+\epsilon$.
Lemma \ref{lem:Basic comparison} implies that
$m_{f}(B_{R}(\bm))$ is not larger than
\begin{equation*}
\int_{\bm\setminus B}\,\int^{\min\{R,\tau(x)\}}_{0}s^{N-1}_{\kappa,\lambda}(t)\,dt\,d\,m_{f,\bm}
+\int_{B}\,\int^{t_{0}+\epsilon}_{0}\,s^{N-1}_{\kappa,\lambda}(t)\,dt\,d\,m_{f,\bm}.
\end{equation*}
This is smaller than $m_{f,\bm}(\bm)\,s_{N,\kappa,\lambda}(R)$.
On the other hand,
$s_{N,\kappa,\lambda}(R)$ is equal to $m_{f}(B_{R}(\bm))/m_{f,\bm}(\bm)$.
This is a contradiction.
\end{proof}
In the case of $N=\infty$,
we have:
\begin{lem}\label{lem:infinite half}
Suppose that
$\bm$ is compact.
Suppose $\ric^{\infty}_{f,M}\geq 0$ and $H_{f,\bm} \geq 0$.
Assume that
there exists $R\in (0,\infty)$ such that
for every $r\in (0,R]$
the equality in $(\ref{eq:infinite volume comparison})$ holds.
Then we have $\tau \geq R$ on $\bm$.
\end{lem}
\begin{proof}
Suppose that
for some $x_{0}\in \bm$
we have $\tau(x_{0})<R$.
Put $t_{0}:=\tau(x_{0})$,
and take $\epsilon>0$ with $t_{0}+\epsilon<R$.
The continuity of $\tau$ implies that
there exists a closed geodesic ball $B$ in $\bm$ centered at $x_{0}$ such that
$\tau$ is smaller than or equal to $t_{0}+\epsilon$ on $B$.
By Lemma \ref{lem:infinite Basic comparison},
\begin{equation*}
m_{f}(B_{R}(\bm))\leq R\,m_{f,\bm}(\partial M \setminus B)+(t_{0}+\epsilon)\,m_{f,\bm}(B)<R\,m_{f,\bm}(\bm).
\end{equation*}
On the other hand,
$m_{f}(B_{R}(\bm))/m_{f,\bm}(\bm)$ is equal to $R$.
This is a contradiction.
We conclude the lemma.
\end{proof}
Suppose that
$\bm$ is compact.
Suppose that
for $N\in [n,\infty)$ we have $\ric^{N}_{f,M}\geq (N-1)\kappa$ and $H_{f,\bm} \geq (N-1)\lambda$
(resp. $\ric^{\infty}_{f,M}\geq 0$ and $H_{f,\bm} \geq 0$),
and that there exists $R\in (0,\bar{C}_{\kappa,\lambda}]\setminus \{\infty\}$ (resp. $R \in (0,\infty)$) such that
for every $r\in (0,R]$
the equality in (\ref{eq:volume comparison}) (resp. (\ref{eq:infinite volume comparison})) holds.
In this case,
for every $r\in (0,R)$
the level set $\rho_{\bm}^{-1}(r)$ is an $(n-1)$-dimensional submanifold of $M$ (see Lemmas \ref{lem:half} and \ref{lem:infinite half}).
In particular,
$(B_{r}(\bm),g)$ is an $n$-dimensional (not necessarily, connected) complete Riemannian manifold with boundary.
We denote by $d_{B_{r}(\bm)}$ and by $d_{\kappa,\lambda,r}$
the Riemannian distances on $(B_{r}(\bm),g)$ and on $[0,r]\times_{\kappa,\lambda}\bm$,
respectively.
\begin{lem}\label{lem:half rigidity}
Suppose that
$\bm$ is compact.
For $N\in [n,\infty)$,
we suppose $\ric^{N}_{f,M}\geq (N-1)\kappa$ and $H_{f,\bm} \geq (N-1)\lambda$.
Assume that
there exists $R\in (0,\bar{C}_{\kappa,\lambda}]\setminus \{\infty\}$ such that
for every $r\in (0,R]$
the equality in $(\ref{eq:volume comparison})$ holds.
Then for every $r\in (0,R)$,
the metric space $(B_{r}(\bm),d_{B_{r}(\bm)})$ is isometric to $([0,r]\times_{\kappa,\lambda}\bm,d_{\kappa,\lambda,r})$,
and for every $x\in \bm$
we have $f\circ \gamma_{x}=f(x)-\log s_{\kappa,\lambda}$ on $[0,r]$.
\end{lem}
\begin{proof}
Since each connected component of $\bm$ one-to-one corresponds to the connected component of $B_{r}(\bm)$,
it suffices to consider the case where $B_{r}(\bm)$ is connected.
For each $x\in \bm$,
choose an orthonormal basis $\{ e_{x,i} \}_{i=1}^{n-1}$ of $T_{x}\bm$.
Let $\{Y_{x,i}\}_{i=1}^{n-1}$ be the $\bm$-Jacobi fields along $\gamma_{x}$
with initial conditions $Y_{x,i}(0)=e_{x,i}$ and $Y_{x,i}'(0)=-A_{u_{x}}e_{x,i}$.
For all $i$
we see $Y_{x,i}=s_{\kappa,\lambda}\, E_{x,i}$ on $[0,\min \{R,\bar{C}_{\kappa,\lambda}\}]$,
where $\{E_{x,i}\}_{i=1}^{n-1}$ are the parallel vector fields along $\gamma_{x}$ with initial condition $E_{x,i}(0)=e_{x,i}$.
Moreover,
$f\circ \gamma_{x}=f(x)-(N-n)\log s_{\kappa,\lambda}$ on $[0,\min \{R,\bar{C}_{\kappa,\lambda}\}]$ (see Remark \ref{rem:infinite absolute volume eq}).
Define a map $\Phi:[0,r]\times \bm\to B_{r}(\bm)$ by $\Phi(t,x):=\gamma_{x}(t)$.
For each $p\in (0,r)\times \bm$
the map $D(\Phi|_{(0,r)\times \bm})_{p}$ sends an orthonormal basis of $T_{p}([0,r]\times \bm)$ to that of $T_{\Phi(p)}B_{r}(\bm)$,
and for each $x\in \{0,r\}\times \bm$
the map $D(\Phi|_{\{0,r\}\times \bm})_{x}$ sends an orthonormal basis of $T_{x}(\{0,r\}\times \bm)$ to that of $T_{\Phi(x)}\partial (B_{r}(\bm))$.
Hence,
$\Phi$ is a Riemannian isometry with boundary from $[0,r]\times_{\kappa,\lambda}\bm$ to $B_{r}(\bm)$.
\end{proof}
In the case of $N=\infty$,
we have:
\begin{lem}\label{lem:infinite half rigidity}
Suppose that
$\bm$ is compact.
Suppose $\ric^{\infty}_{f,M}\geq 0$ and $H_{f,\bm} \geq 0$.
Assume that there exists $R\in (0,\infty)$ such that
for every $r\in (0,R]$
the equality in $(\ref{eq:infinite volume comparison})$ holds.
Then for every $r\in (0,R)$
the metric space $(B_{r}(\bm),d_{B_{r}(\bm)})$ is isometric to $([0,r]\times \bm,d_{[0,r]\times \bm})$.
\end{lem}
\begin{proof}
We may assume that
$B_{r}(\bm)$ is connected.
For each $x\in \bm$,
choose an orthonormal basis $\{ e_{x,i} \}_{i=1}^{n-1}$ of $T_{x}\bm$.
Let $\{Y_{x,i}\}_{i=1}^{n-1}$ be the $\bm$-Jacobi fields along $\gamma_{x}$
with initial conditions $Y_{x,i}(0)=e_{x,i}$ and $Y_{x,i}'(0)=-A_{u_{x}}e_{x,i}$.
For all $i$
we have $Y_{x,i}=E_{x,i}$ on $[0,r]$,
where $\{E_{x,i}\}_{i=1}^{n-1}$ are the parallel vector fields along $\gamma_{x}$ with initial condition $E_{x,i}(0)=e_{x,i}$ (see Remark \ref{rem:infinite absolute volume eq}).
Define a map $\Phi:[0,r]\times \bm\to B_{r}(\bm)$ by $\Phi(t,x):=\gamma_{x}(t)$.
We see that
$\Phi$ is a Riemannian isometry with boundary from $[0,r]\times \bm$ to $B_{r}(\bm)$.
\end{proof}
Now,
we prove Theorem \ref{thm:volume growth distance rigidity}.
\begin{proof}[Proof of Theorem \ref{thm:volume growth distance rigidity}]
Suppose that
$\bm$ is compact.
For $N\in [n,\infty)$,
suppose $\ric^{N}_{f,M}\geq (N-1)\kappa$ and $H_{f,\bm} \geq (N-1)\lambda$.
Suppose (\ref{eq:assumption of volume growth}).

By Lemma \ref{lem:absolute volume comparison} and Theorem \ref{thm:volume comparison},
for all $r,R\in (0,\infty)$ with $r\leq R$,
\begin{equation*}\label{eq:volume growth1}
\frac{m_{f}(B_{R}(\bm))}{s_{N,\kappa,\lambda}(R)}=\frac{m_{f}(B_{r}(\bm))}{s_{N,\kappa,\lambda}(r)}=m_{f,\bm}(\bm).
\end{equation*}
If $\kappa$ and $\lambda$ satisfy the ball-condition,
then for $R=\const$,
and for every $r\in (0,R]$ the equality in (\ref{eq:volume comparison}) holds;
in particular,
Lemmas \ref{lem:Inscribed radius comparison} and \ref{lem:half} imply that
$\tau$ is equal to $\const$ on $\bm$.
If $\kappa$ and $\lambda$ do not satisfy the ball-condition,
then for every $R\in (0,\infty)$,
and for every $r\in (0,R]$
the equality in (\ref{eq:volume comparison}) holds;
in particular,
Lemma \ref{lem:half} implies that
$\tau=\infty$ on $\bm$.
It follows that
$\tau$ coincides with $\bconst$ on $\bm$.
From Lemma \ref{lem:half rigidity},
for every $x\in \bm$
we derive $f\circ \gamma_{x}=f(x)-(N-n)\log s_{\kappa,\lambda}$ on $I_{\kappa,\lambda}$.

If $\kappa$ and $\lambda$ satisfy the ball-condition,
then $\dm=\const$.
By Lemma \ref{lem:bmcompact},
$M$ is compact.
There exists $p\in M$ with $\rho_{\bm}(p)=\const$.
Due to Theorem \ref{thm:Ball rigidity},
$(M,d_{M})$ is isometric to $(\ball,d_{\ball})$ and $N=n$.

If $\kappa$ and $\lambda$ do not satisfy the ball-condition,
then $\cut \bm=\emptyset$.
It follows that
$\bm$ is connected.
Take a sequence $\{r_{i}\}$ with $r_{i}\to \infty$.
By Lemma \ref{lem:half rigidity},
for each $r_{i}$,
there exists a Riemannian isometry $\Phi_{i}:[0,r_{i}]\times\bm\to B_{r_{i}}(\bm)$ with boundary
from $[0,r_{i}]\times_{\kappa,\lambda} \bm$ to $B_{r_{i}}(\bm)$ defined by $\Phi_{i}(t,x):=\gamma_{x}(t)$.
Since $\cut \bm=\emptyset$,
we obtain a Riemannian isometry $\Phi:[0,\infty)\times \bm\to M$ with boundary
from $[0,\infty)\times_{\kappa,\lambda}\bm$ to $M$
defined by $\Phi(t,x):=\gamma_{x}(t)$ such that $\Phi|_{[0,r_{i}]\times\bm}=\Phi_{i}$ for all $r_{i}$.
This proves Theorem \ref{thm:volume growth distance rigidity}.
\end{proof}
Next,
we prove Theorem \ref{thm:infinite volume growth distance rigidity}.
\begin{proof}[Proof of Theorem \ref{thm:infinite volume growth distance rigidity}]
Suppose that
$\bm$ is compact.
Suppose $\ric^{\infty}_{f,M}\geq 0$ and $H_{f,\bm} \geq 0$.
Furthermore,
we assume (\ref{eq:assumption of infinite volume growth}).

By Lemma \ref{lem:infinite absolute volume comparison} and Theorem \ref{thm:infinite volume comparison},
for all $R\in (0,\infty)$ and $r \in (0,R]$,
\begin{equation*}\label{eq:volume growth3}
\frac{m_{f}(B_{R}(\bm))}{R}=\frac{m_{f}(B_{r}(\bm))}{r}=m_{f,\bm}(\bm).
\end{equation*}
For every $R\in (0,\infty)$,
and for every $r \in (0,R]$
the equality in $(\ref{eq:infinite volume comparison})$ holds.
From Lemma \ref{lem:infinite half},
it follows that
$\tau=\infty$ on $\bm$.
We have $\cut \bm=\emptyset$,
and hence $\bm$ is connected.
Take a sequence $\{r_{i}\}$ with $r_{i}\to \infty$.
Lemma \ref{lem:infinite half rigidity} implies that
for each $r_{i}$
there exists a Riemannian isometry $\Phi_{i}:[0,r_{i}]\times \bm\to B_{r_{i}}(\bm)$ with boundary
from $[0,r_{i}]\times \bm$ to $B_{r_{i}}(\bm)$ defined by $\Phi_{i}(t,x):=\gamma_{x}(t)$.
Since $\cut \bm=\emptyset$,
we obtain a Riemannian isometry $\Phi:[0,\infty)\times\bm\to M$ with boundary
from $[0,\infty)\times \bm$ to $M$
defined by $\Phi(t,x):=\gamma_{x}(t)$ such that $\Phi|_{[0,r_{i}]\times \bm}=\Phi_{i}$ for all $r_{i}$.
This proves Theorem \ref{thm:infinite volume growth distance rigidity}.
\end{proof}
\section{Splitting theorems}\label{sec:Splitting theorems}
Let $M$ be an $n$-dimensional, 
connected complete Riemannian manifold with boundary,
and let $f:M\to \mathbb{R}$ be a smooth function.
\subsection{Main splitting theorems}
We prove Theorem \ref{thm:splitting}.
\begin{proof}[Proof of Theorem \ref{thm:splitting}]
Let $\kappa \leq 0$ and $\lambda:=\sqrt{\vert \kappa \vert}$.
For $N\in [n,\infty)$,
we suppose $\ric^{N}_{f,M}\geq (N-1)\kappa$ and $H_{f,\bm} \geq (N-1)\lambda$.
Suppose that
for some $x_{0}\in \bm$ we have $\tau(x_{0})=\infty$.

For the connected component $\bm_{0}$ of $\bm$ containing $x_{0}$,
we put
\begin{equation*}
\Omega:=\{y\in \bm_{0} \mid \tau(y)=\infty \}.
\end{equation*}
The assumption implies that
$\Omega$ is non-empty.
From the continuity of $\tau$,
it follows that $\Omega$ is closed in $\bm_{0}$.

We show the openness of $\Omega$ in $\bm_{0}$.
Fix $y_{0}\in \Omega$.
Take $l \in (0,\infty)$,
and put $p_{0}:=\gamma_{y_{0}}(l)$.
There exists an open neighborhood $U$ of $p_{0}$ in $\inte M$ contained in $D_{\bm}$.
Taking $U$ smaller,
we may assume that 
for each point $q\in U$
the unique foot point on $\bm$ of $q$ belongs to $\bm_{0}$.
By Lemma \ref{lem:asymptote},
there exists $\epsilon \in (0,\infty)$ such that
for all $q\in B_{\epsilon}(p_{0})$,
all asymptotes for $\gamma_{y_{0}}$ from $q$ lie in $\inte M$.
We may assume $U\subset B_{\epsilon}(p_{0})$.
Fix $q_{0}\in U$,
and take an asymptote $\gamma_{q_{0}}:[0,\infty)\to M$ for $\gamma_{y_{0}}$ from $q_{0}$.
For $t\in (0,\infty)$,
define a function $b_{\gamma_{y_{0}},t}:M\to \mathbb{R}$ by
\begin{equation*}
b_{\gamma_{y_{0}},t}(p):=b_{\gamma_{y_{0}}}(q_{0})+t-d_{M}(p,\gamma_{q_{0}}(t)).
\end{equation*}
We see that
$b_{\gamma_{y_{0}},t}-\rho_{\bm}$ is a support function of $b_{\gamma_{y_{0}}}-\rho_{\bm}$ at $q_{0}$.
Since $\gamma_{q_{0}}$ lie in $\inte M$,
for every $t\in(0,\infty)$
the function $b_{\gamma_{y_{0}},t}$ is smooth on a neighborhood of $q_{0}$.
From Lemma \ref{lem:finite pointed Laplacian comparison},
we deduce $\Delta_{f} b_{\gamma_{y_{0}},t}(q_{0})\leq (N-1)(s'_{\kappa}(t)/s_{\kappa}(t))$.
Note that $s'_{\kappa}(t)/s_{\kappa}(t)\to \lambda$ as $t\to \infty$.
Furthermore,
$\rho_{\bm}$ is smooth on $U$,
and by Lemma \ref{lem:Laplacian comparison} we have $\Delta_{f} \rho_{\bm}\geq (N-1)\lambda$ on $U$.
Hence,
$b_{\gamma_{y_{0}}}-\rho_{\bm}$ is $f$-subharmonic on $U$.
Since $b_{\gamma_{y_{0}}}-\rho_{\bm}$ takes the maximal value $0$ at $p_{0}$,
Lemma \ref{lem:maximal principle} implies $b_{\gamma_{y_{0}}}=\rho_{\bm}$ on $U$.
From Lemma \ref{lem:busemann function},
it follows that
$\Omega$ is open in $\bm_{0}$.

Since $\bm_{0}$ is a connected component of $\bm$,
we have $\Omega=\bm_{0}$.
By Lemma \ref{lem:splitting1},
$\bm$ is connected
and $\cut \bm=\emptyset$.
The equality in Lemma \ref{lem:Laplacian comparison} holds on $\inte M$.
For each $x\in \bm$,
choose an orthonormal basis $\{e_{x,i}\}_{i=1}^{n-1}$ of $T_{x}\bm$.
Let $\{Y_{x,i}\}_{i=1}^{n-1}$ be the $\bm$-Jacobi fields along $\gamma_{x}$
with initial conditions $Y_{x,i}(0)=e_{x,i}$ and $Y'_{x,i}(0)=-A_{u_{x}}e_{x,i}$.
For all $i$
we see $Y_{x,i}=s_{\kappa,\lambda}E_{x,i}$ on $[0,\infty)$,
where $\{E_{x,i}\}_{i=1}^{n-1}$ are the parallel vector fields along $\gamma_{x}$ with initial condition $E_{x,i}(0)=e_{x,i}$.
Moreover,
for all $t\in [0,\infty)$
we have $(f\circ \gamma_{x})(t)=f(x)+(N-n)\lambda t$ (see Remark \ref{rem:Equality in Laplacian comparison}).
Define a map $\Phi:[0,\infty)\times \bm\to M$ by $\Phi(t,x):=\gamma_{x}(t)$.
For every $p\in (0,\infty)\times \bm$
the map $D(\Phi|_{(0,\infty)\times \bm})_{p}$ sends an orthonormal basis of $T_{p}((0,\infty)\times \bm)$ to that of $T_{\Phi(p)}M$,
and for every $x\in \{0\}\times \bm$
the map $D(\Phi|_{\{0\}\times \bm})_{x}$ sends an orthonormal basis of $T_{x}(\{0\}\times \bm)$ to that of $T_{\Phi(x)}\bm$.
Hence,
$\Phi$ is a Riemannian isometry with boundary from $[0,\infty)\times_{\kappa,\lambda} \bm$ to $M$.
This proves Theorem \ref{thm:splitting}.
\end{proof}
Next,
we prove Theorem \ref{thm:infinite splitting}.
\begin{proof}[Proof of Theorem  \ref{thm:infinite splitting}]
Assume $\sup f(M)<\infty$.
Suppose $\ric^{\infty}_{f,M}\geq 0$ and $H_{f,\bm} \geq 0$.
Let $x_{0}\in \bm$ satisfy $\tau(x_{0})=\infty$.

For the connected component $\bm_{0}$ of $\bm$ containing $x_{0}$,
we put
\begin{equation*}
\Omega:=\{y\in \bm_{0} \mid \tau(y)=\infty \}.
\end{equation*}
The assumption and the continuity of $\tau$ imply that 
$\Omega$ is a non-empty closed subset of $\bm_{0}$.

We prove the openness of $\Omega$ in $\bm_{0}$.
Fix $y_{0}\in \Omega$.
Take $l \in (0,\infty)$,
and put $p_{0}:=\gamma_{y_{0}}(l)$.
There exists an open neighborhood $U$ of $p_{0}$ in $\inte M$ contained in $D_{\bm}$.
We may assume that 
for each point $q\in U$
the unique foot point on $\bm$ of $q$ belongs to $\bm_{0}$.
By Lemma \ref{lem:asymptote},
there exists $\epsilon \in (0,\infty)$ such that
for all $q\in B_{\epsilon}(p_{0})$,
all asymptotes for $\gamma_{y_{0}}$ from $q$ lie in $\inte M$.
We may assume $U\subset B_{\epsilon}(p_{0})$.
By Lemma \ref{lem:infinite pointed Laplacian comparison},
$b_{\gamma_{y_{0}}}$ is $f$-subharmonic on $U$.
Furthermore,
$\rho_{\bm}$ is smooth on $U$,
and Lemma \ref{lem:infinite Laplacian comparison} implies $\Delta_{f} \rho_{\bm} \geq 0$ on $U$.
Therefore,
$b_{\gamma_{y_{0}}}-\rho_{\bm}$ is $f$-subharmonic on $U$.
Since $b_{\gamma_{y_{0}}}-\rho_{\bm}$ takes the maximal value $0$ at $p_{0}$,
Lemma \ref{lem:maximal principle} implies $b_{\gamma_{y_{0}}}=\rho_{\bm}$ on $U$.
By Lemma \ref{lem:busemann function},
$\Omega$ is open in $\bm_{0}$.

Since $\bm_{0}$ is a connected component of $\bm$,
we have $\Omega=\bm_{0}$.
By Lemma \ref{lem:splitting1},
$\bm$ is connected
and $\cut \bm=\emptyset$.
The equality in Lemma \ref{lem:infinite Laplacian comparison} holds on $\inte M$.
For each $x\in \bm$,
choose an orthonormal basis $\{e_{x,i}\}_{i=1}^{n-1}$ of $T_{x}\bm$.
Let $\{Y_{x,i}\}_{i=1}^{n-1}$ be the $\bm$-Jacobi fields along $\gamma_{x}$
with initial conditions $Y_{x,i}(0)=e_{x,i}$ and $Y'_{x,i}(0)=-A_{u_{x}}e_{x,i}$.
For all $i$
we see $Y_{x,i}=E_{x,i}$ on $[0,\infty)$,
where $\{E_{x,i}\}_{i=1}^{n-1}$ are the parallel vector fields along $\gamma_{x}$ with initial condition $E_{x,i}(0)=e_{x,i}$ (see Remark \ref{rem:Equality in Laplacian comparison}).
Hence,
the map $\Phi:[0,\infty)\times \bm\to M$ defined by $\Phi(t,x):=\gamma_{x}(t)$ is a Riemannian isometry with boundary from $[0,\infty)\times \bm$ to $M$.
This completes the proof of Theorem \ref{thm:infinite splitting}.
\end{proof}
Lemma \ref{lem:bmcompact} and the continuity of $\tau$ imply that
if $\bm$ is compact and $M$ is non-compact,
then for some $x_{0}\in \bm$ we have $\tau(x_{0})=\infty$.
By Theorems \ref{thm:splitting} and \ref{thm:infinite splitting},
we have the following rigidity results that have been proved in \cite{K3} (see also \cite{CK}) when $f=0$ and $N=n$.
\begin{cor}\label{cor:Kasue splitting}
Let $M$ be an $n$-dimensional,
connected complete Riemannian manifold with boundary,
and let $f:M\to \mathbb{R}$ be a smooth function.
Let $\kappa \leq 0$ and $\lambda:=\sqrt{\vert \kappa \vert}$.
For $N\in [n,\infty)$,
we suppose $\ric^{N}_{f,M}\geq (N-1)\kappa$ and $H_{f,\bm} \geq (N-1)\lambda$.
If $M$ is non-compact and $\bm$ is compact,
then $(M,d_{M})$ is isometric to $([0,\infty)\times_{\kappa,\lambda}\bm,d_{\kappa,\lambda})$,
and for all $x\in \bm$ and $t\in [0,\infty)$
we have $(f \circ \gamma_{x})(t)=f(x)+(N-n)\lambda t$.
\end{cor}
\begin{cor}\label{cor:infinite Kasue splitting}
Let $M$ be a connected complete Riemannian manifold with boundary,
and let $f:M\to \mathbb{R}$ be a smooth function such that $\sup f(M)<\infty$.
Suppose $\ric^{\infty}_{f,M}\geq 0$ and $H_{f,\bm} \geq 0$.
If $M$ is non-compact and $\bm$ is compact,
then the metric space $(M,d_{M})$ is isometric to $([0,\infty)\times\bm,d_{[0,\infty)\times\bm})$.
\end{cor}
\subsection{Weighted Ricci curvature on the boundary}
Let $h$ be the induced Riemannian metric on $\bm$.
For $x \in \bm$,
and for a unit vector $u$ in $T_{x}\bm$,
we denote by $K_{g}(u_{x},u)$ the sectional curvature at $x$ in $(M,g)$ determined by $u_{x}$ and $u$.

It seems that
the following is well-known,
especially in a submanifold setting (see e.g., Proposition 9.36 in \cite{Be}, and Lemma 5.4 in \cite{Sa}).
\begin{lem}\label{lem:boundary Ricci curvature1}
Take $x\in \bm$,
and a unit vector $u$ in $T_{x}\bm$.
Choose an orthonormal basis $\{ e_{x,i} \}_{i=1}^{n-1}$ of $T_{x}\bm$ with $e_{x,1}=u$.
Then we have
\begin{equation*}
\ric_{h}(u)=\ric_{g}(u)-K_{g}(u_{x},u)+\tr A_{S(u,u)}-\sum_{i=1}^{n-1} \Vert S(u,e_{x,i})\Vert^{2}.
\end{equation*}
\end{lem}
For all $x\in \bm$ and $u\in T_{x}\bm$,
we see
\begin{align}\label{eq:boundary gradient}
h((\nabla(f|_{\bm}))_{x},u)&=g((\nabla f)_{x},u),\\ 
\Hess (f|_{\bm})(u,u)&=\Hess f(u,u)+g\left((\nabla f)_{x},u_{x}\right)\,g\left(S(u,u),u_{x}\right).  \label{eq:boundary hessian}
\end{align}

We show the following:
\begin{lem}\label{lem:boundary weighted Ricci curvature}
Take $x\in \bm$,
and a unit vector $u$ in $T_{x}\bm$.
Choose an orthonormal basis $\{ e_{x,i} \}_{i=1}^{n-1}$ of $T_{x}\bm$ with $e_{x,1}=u$.
Then for all $N\in [n,\infty)$,
we have
\begin{align}\label{eq:boundary weighted Ricci curvature}
\ric^{N-1}_{f|_{\bm}}(u) & = \ric^{N}_{f}(u)+g((\nabla f)_{x},u_{x})\,g(S(u,u),u_{x})\\
                                    & - K_{g}(u_{x},u)+\tr A_{S(u,u)}-\sum_{i=1}^{n-1} \Vert S(u,e_{x,i})\Vert^{2}.\notag
\end{align}
\end{lem}
\begin{proof}
Assume $N \in (n,\infty)$.
By (\ref{eq:boundary gradient}) and (\ref{eq:boundary hessian}),
we have
\begin{multline*}
\ric^{N-1}_{f|_{\bm}}(u)=\ric_{h}(u)+\Hess (f|_{\bm})(u,u)-\frac{h((\nabla(f|_{\bm}))_{x},u)^{2}}{(N-1)-(n-1)}\\
                                    =\ric_{h}(u)+\Hess f(u,u)+g((\nabla f)_{x},u_{x})\,g(S(u,u),u_{x})-\frac{g((\nabla f)_{x},u)^{2}}{N-n}.
\end{multline*}
By Lemma \ref{lem:boundary Ricci curvature1},
we see (\ref{eq:boundary weighted Ricci curvature}).

Assume $N=n$.
If $f$ is constant,
then we see $\ric^{N-1}_{f|_{\bm}}(u)=\ric_{h}(u)$ and $\ric^{N}_{f}(u)=\ric_{g}(u)$,
and hence Lemma \ref{lem:boundary Ricci curvature1} implies (\ref{eq:boundary weighted Ricci curvature}).
If $f$ is not constant,
then both the left hand side of (\ref{eq:boundary weighted Ricci curvature}) and the right hand side are equal to $-\infty$.
Therefore,
we complete the proof.
\end{proof}
In the case of $N=\infty$,
we have:
\begin{lem}\label{lem:infinite boundary weighted Ricci curvature}
Take $x\in \bm$,
and a unit vector $u$ in $T_{x}\bm$.
Choose an orthonormal basis $\{ e_{x,i} \}_{i=1}^{n-1}$ of $T_{x}\bm$ with $e_{x,1}=u$.
Then we have
\begin{align}\label{eq:infinite boundary weighted Ricci curvature}
\ric^{\infty}_{f|_{\bm}}(u) & = \ric^{\infty}_{f}(u)+g((\nabla f)_{x},u_{x})\,g(S(u,u),u_{x})\\
                                    & - K_{g}(u_{x},u)+\tr A_{S(u,u)}-\sum_{i=1}^{n-1} \Vert S(u,e_{x,i})\Vert^{2}.\notag
\end{align}
\end{lem}
\begin{proof}
From (\ref{eq:boundary hessian}),
it follows that
\begin{align*}
\ric^{\infty}_{f|_{\bm}}(u)&=\ric_{h}(u)+\Hess (f|_{\bm})(u,u)\\
                                      &=\ric_{h}(u)+\Hess f(u,u)+g((\nabla f)_{x},u_{x})\,g(S(u,u),u_{x}).
\end{align*}
Using Lemma \ref{lem:boundary Ricci curvature1},
we have (\ref{eq:infinite boundary weighted Ricci curvature}).
\end{proof}

\subsection{Multi-splitting}\label{sec:Multi-splitting}
By Lemma \ref{lem:boundary weighted Ricci curvature},
we see the following:
\begin{lem}\label{lem:boundary nonnegative Ricci curvature}
For $N\in [n,\infty)$,
we suppose $\ric^{N}_{f,M}\geq 0$.
If $(M,d_{M})$ is isometric to $([0,\infty)\times \bm,d_{[0,\infty)\times \bm})$,
then $\ric^{N-1}_{f|_{\bm},\bm}\geq 0$.
\end{lem}
\begin{proof}
There exists a Riemannian isometry with boundary from $M$ to $[0,\infty)\times\bm$.
Take $x\in \bm$,
and choose an orthonormal basis $\{e_{x,i}\}_{i=1}^{n-1}$ of $T_{x}\bm$.
Let $\{Y_{x,i}\}_{i=1}^{n-1}$ be the $\bm$-Jacobi fields along $\gamma_{x}$
with initial conditions $Y_{x,i}(0)=e_{x,i}$ and $Y'_{x,i}(0)=-A_{u_{x}}e_{x,i}$.
For all $i$
we have $Y_{x,i}=E_{x,i}$,
where $\{E_{x,i}\}_{i=1}^{n-1}$ are the parallel vector fields along $\gamma_{x}$ with initial condition $E_{x,i}(0)=e_{x,i}$.
We see $A_{u_{x}}e_{x,i}=0_{x}$ and $Y''_{x,1}(0)=0_{x}$;
in particular,
$\tr A_{u_{x}}=0$ and $K_{g}(u_{x},e_{x,1})=0$.
For all $i,j$
we have $S(e_{x,i},e_{x,j})=0_{x}$.
By (\ref{eq:boundary weighted Ricci curvature}) and $\ric^{N}_{f,M}\geq 0$,
we have $\ric^{N-1}_{f|_{\bm},\bm}\geq 0$.
\end{proof}
Let $M_{0}$ be a connected complete Riemannian manifold (without boundary).
A normal geodesic $\gamma:\mathbb{R} \to M_{0}$ is said to be a \textit{line}
if for all $s,t\in \mathbb{R}$ 
we have $d_{M_{0}}(\gamma(s),\gamma(t))=|s-t|$.

Fang, Li and Zhang \cite{FLZ} have proved the following splitting theorem of Cheeger-Gromoll type (see Theorem 1.3 in \cite{FLZ}):
\begin{thm}[\cite{FLZ}]\label{thm:splitting theorem of Cheeger-Gromoll type}
Let $M_{0}$ be an $n$-dimensional,
connected complete Riemannian manifold,
and let $f:M_{0}\to \mathbb{R}$ be a smooth function.
For $N\in [n,\infty)$,
we suppose $\ric^{N}_{f,M_{0}}\geq 0$.
If $M_{0}$ contains a line,
then there exists an $(n-1)$-dimensional Riemannian manifold $N_{0}$ such that
$M_{0}$ is isometric to the standard product $\mathbb{R}\times N_{0}$.
\end{thm}
We have the following corollary of Theorem \ref{thm:splitting}:
\begin{cor}\label{cor:boundary splitting}
Let $M$ be an $n$-dimensional,
connected complete Riemannian manifold with boundary,
and let $f:M\to \mathbb{R}$ be a smooth function.
For $N\in [n,\infty)$,
we suppose $\ric^{N}_{f,M}\geq 0$ and $H_{f,\bm} \geq 0$.
If for some $x_{0}\in \bm$
we have $\tau(x_{0})=\infty$,
then there exist $k\in \{0,\dots,n-1\}$ and
an $(n-1-k)$-dimensional,
connected complete Riemannian manifold $N_{0}$ containing no line such that
$(\bm,d_{\bm})$ is isometric to $(\mathbb{R}^{k}\times N_{0},d_{\mathbb{R}^{k}\times N_{0}})$.
In particular,
$(M,d_{M})$ is isometric to $([0,\infty)\times \mathbb{R}^{k}\times N_{0},d_{[0,\infty)\times \mathbb{R}^{k}\times N_{0}})$.
\end{cor}
\begin{proof}
Due to Theorem \ref{thm:splitting},
the metric space $(M,d_{M})$ is isometric to $([0,\infty)\times \bm,d_{[0,\infty)\times \bm})$.
Lemma \ref{lem:boundary nonnegative Ricci curvature} implies $\ric^{N-1}_{f|_{\bm},\bm}\geq 0$.
Applying Theorem \ref{thm:splitting theorem of Cheeger-Gromoll type} to $\bm$ inductively,
we complete the proof.
\end{proof}
In the case of $N=\infty$,
we see:
\begin{lem}\label{lem:infinite boundary nonnegative Ricci curvature}
If $\ric^{\infty}_{f,M}\geq 0$,
and if the metric space $(M,d_{M})$ is isometric to $([0,\infty)\times \bm,d_{[0,\infty)\times \bm})$,
then $\ric^{\infty}_{f|_{\bm},\bm}\geq 0$.
\end{lem}
\begin{proof}
There exists a Riemannian isometry with boundary from $M$ to $[0,\infty)\times\bm$.
Take $x\in \bm$,
and choose an orthonormal basis $\{e_{x,i}\}_{i=1}^{n-1}$ of $T_{x}\bm$.
Let $\{Y_{x,i}\}_{i=1}^{n-1}$ be the $\bm$-Jacobi fields along $\gamma_{x}$
with initial conditions $Y_{x,i}(0)=e_{x,i}$ and $Y'_{x,i}(0)=-A_{u_{x}}e_{x,i}$.
For all $i$
we have $Y_{x,i}=E_{x,i}$,
where $\{E_{x,i}\}_{i=1}^{n-1}$ are the parallel vector fields along $\gamma_{x}$ with initial condition $E_{x,i}(0)=e_{x,i}$.
This implies $A_{u_{x}}e_{x,i}=0_{x}$ and $Y''_{x,1}(0)=0_{x}$.
Hence,
$\tr A_{u_{x}}=0$ and $K_{g}(u_{x},e_{x,1})=0$.
For all $i,j$
we see $S(e_{x,i},e_{x,j})=0_{x}$.
From (\ref{eq:infinite boundary weighted Ricci curvature}),
and from $\ric^{\infty}_{f,M}\geq 0$,
we deduce $\ric^{\infty}_{f|_{\bm},\bm}\geq 0$.
\end{proof}
Fang, Li and Zhang \cite{FLZ} have proved the following splitting theorem of Cheeger-Gromoll type (see Theorem 1.1 in \cite{FLZ}):
\begin{thm}[\cite{FLZ}]\label{thm:infinite splitting theorem of Cheeger-Gromoll type}
Let $M_{0}$ be an $n$-dimensional,
connected complete Riemannian manifold,
and let $f:M_{0}\to \mathbb{R}$ be a smooth function such that  $\sup f(M_{0})<\infty$.
Suppose $\ric^{\infty}_{f,M_{0}}\geq 0$.
If $M_{0}$ contains a line,
then there exists an $(n-1)$-dimensional Riemannian manifold $N_{0}$ such that
$M_{0}$ is isometric to the standard product $\mathbb{R}\times N_{0}$.
\end{thm}
\begin{rem}\label{rem:splitting theorem of Cheeger-Gromoll type}
Lichnerowicz \cite{Lic} has proved Theorem \ref{thm:infinite splitting theorem of Cheeger-Gromoll type}
under the assumption that
$f$ is bounded.
\end{rem}
In the case of $N=\infty$,
we have the following:
\begin{cor}\label{cor:infinite boundary splitting}
Let $M$ be an $n$-dimensional,
connected complete Riemannian manifold with boundary,
and let $f:M\to \mathbb{R}$ be a smooth function such that $\sup f(M)<\infty$.
Suppose $\ric^{\infty}_{f,M}\geq 0$ and $H_{f,\bm} \geq 0$.
If for some $x_{0}\in \bm$
we have $\tau(x_{0})=\infty$,
then there exist $k\in \{0,\dots,n-1\}$ and
an $(n-1-k)$-dimensional,
connected complete Riemannian manifold $N_{0}$ containing no line such that
$(\bm,d_{\bm})$ is isometric to $(\mathbb{R}^{k}\times N_{0},d_{\mathbb{R}^{k}\times N_{0}})$.
In particular,
$(M,d_{M})$ is isometric to $([0,\infty)\times \mathbb{R}^{k}\times N_{0},d_{[0,\infty)\times \mathbb{R}^{k}\times N_{0}})$.
\end{cor}
\begin{proof}
By Theorem \ref{thm:infinite splitting},
$(M,d_{M})$ is isometric to $([0,\infty)\times \bm,d_{[0,\infty)\times \bm})$.
From Lemma \ref{lem:infinite boundary nonnegative Ricci curvature},
we derive $\ric^{\infty}_{f|_{\bm},\bm}\geq 0$.
Notice that
$\sup_{x\in \bm}f(x)$ is finite.
By using Theorem \ref{thm:infinite splitting theorem of Cheeger-Gromoll type},
we obtain the corollary.
\end{proof}
\subsection{Variants of splitting theorems}\label{sec:Variants of splitting theorems}
We have already known several rigidity results
studied in \cite{K3} (and \cite{CK}, \cite{I})
for manifolds with boundary whose boundaries are disconnected.
We study generalizations of the results in \cite{K3} (and \cite{CK}, \cite{I}). 

The following has been proved in \cite{K3} (see Lemma 1.6 in \cite{K3}):
\begin{lem}[\cite{K3}]\label{lem:disconnected lemma}
Suppose that
$\bm$ is disconnected.
Let $\{\bm_{i}\}_{i=1,2,\dots}$ denote the connected components of $\bm$.
Assume that $\bm_{1}$ is compact.
Put $D:=\inf_{i=2,3,\dots}\, d_{M}(\bm_{1},\bm_{i})$.
Then there exists a connected component $\bm_{2}$ of $\bm$ such that $d_{M}(\bm_{1},\bm_{2})=D$.
Furthermore,
for every $i=1,2$
there exists $x_{i}\in \bm_{i}$ such that $d_{M}(x_{1},x_{2})=D$.
The normal minimal geodesic $\gamma:[0,D]\to M$ from $x_{1}$ to $x_{2}$ is orthogonal to $\bm$ both at $x_{1}$ and at $x_{2}$,
and the restriction $\gamma|_{(0,D)}$ lies in $\inte M$.
\end{lem}

First,
we prove the following:
\begin{thm}\label{thm:disconnected splitting1}
Let $M$ be an $n$-dimensional,
connected complete Riemannian manifold with boundary,
and let $f:M\to \mathbb{R}$ be a smooth function.
Suppose that $\bm$ is disconnected.
Let $\{\bm_{i}\}_{i=1,2,\dots}$ denote the connected components of $\bm$.
Assume that $\bm_{1}$ is compact.
Put $D:=\inf_{i=2,3,\dots}\, d_{M}(\bm_{1},\bm_{i})$.
For $N\in [n,\infty]$,
we suppose $\ric^{N}_{f,M}\geq 0$ and $H_{f,\bm} \geq 0$.
Then $(M,d_{M})$ is isometric to $([0,D]\times \bm_{1},d_{[0,D]\times \bm_{1}})$.
Moreover,
if $N\in [n,\infty)$,
then for every $x\in \bm_{1}$
the function $f\circ \gamma_{x}$ is constant on $[0,D]$.
\end{thm}
\begin{proof}
By Lemma \ref{lem:disconnected lemma},
there exists a connected component $\bm_{2}$ of $\bm$ such that
$d_{M}(\bm_{1},\bm_{2})=D$.
For each $i=1,2$,
let $\rho_{\bm_{i}}:M\to \mathbb{R}$ be the distance function from $\bm_{i}$ defined as $\rho_{\bm_{i}}(p):=d_{M}(p,\bm_{i})$.
Put
\begin{equation*}
\Omega:=\{p\in \inte M \mid \rho_{\bm_{1}}(p)+\rho_{\bm_{2}}(p)=D\}.
\end{equation*}
Lemma \ref{lem:disconnected lemma} implies that
$\Omega$ is a non-empty closed subset of $\inte M$.

We show that
$\Omega$ is open in $\inte M$.
Take $p\in \Omega$.
For each $i=1,2$,
there exists a foot point $x_{p,i}\in \bm_{i}$ on $\bm_{i}$ of $p$ such that
$d_{M}(p,x_{p,i})=\rho_{\bm_{i}}(p)$.
From the triangle inequality,
we derive $d_{M}(x_{p,1},x_{p,2})=D$.
The normal minimal geodesic $\gamma:[0,D]\to M$ from $x_{p,1}$ to $x_{p,2}$ is orthogonal to $\bm$ at $x_{p,1}$ and at $x_{p,2}$.
Furthermore,
$\gamma|_{(0,D)}$ lies in $\inte M$ and passes through $p$.
There exists an open neighborhood $U$ of $p$ such that 
$U$ is contained in $\inte M$
and $\rho_{\bm_{i}}$ is smooth on $U$.
By using Lemmas \ref{lem:Basic comparison} and \ref{lem:infinite Basic comparison},
we see $\Delta_{f}\, \rho_{\bm_{i}} \geq 0$ on $U$;
in particular,
$-(\rho_{\bm_{1}}+\rho_{\bm_{2}})$ is $f$-subharmonic on $U$.
By Lemma \ref{lem:maximal principle},
$\Omega$ is open in $\inte M$.

Since $\inte M$ is connected,
we have $\inte M=\Omega$.
For each $x\in \bm_{1}$,
choose an orthonormal basis $\{e_{x,i}\}_{i=1}^{n-1}$ of $T_{x}\bm$.
Let $\{Y_{x,i}\}_{i=1}^{n-1}$ be the $\bm$-Jacobi fields along $\gamma_{x}$
with initial conditions $Y_{x,i}(0)=e_{x,i}$ and $Y'_{x,i}(0)=-A_{u_{x}}e_{x,i}$.
For all $i$
we see $Y_{x,i}=E_{x,i}$ on $[0,D]$,
where $\{E_{x,i}\}_{i=1}^{n-1}$ are the parallel vector fields along $\gamma_{x}$ with initial condition $E_{x,i}(0)=e_{x,i}$.
Moreover,
if $N\in [n,\infty)$,
then $f\circ \gamma_{x}$ is constant on $[0,D]$ (see Remarks \ref{rem:Equality in Basic comparison} and \ref{rem:Equality in infinite Basic comparison}).
We see that
a map $\Phi:[0,D]\times \bm_{1}\to M$ defined by $\Phi(t,x):=\gamma_{x}(t)$ is a Riemannian isometry with boundary from $[0,D]\times \bm_{1}$ to $M$.
\end{proof}
Next,
we show the following:
\begin{thm}\label{thm:disconnected splitting2}
Let $M$ be an $n$-dimensional,
connected complete Riemannian manifold with boundary,
and let $f:M\to \mathbb{R}$ be a smooth function.
Suppose that
$\bm$ is disconnected.
Let $\{\bm_{i}\}_{i=1,2,\dots}$ denote the connected components of $\bm$.
Assume that $\bm_{1}$ is compact.
Put $D:=\inf_{i=2,3,\dots}\, d_{M}(\bm_{1},\bm_{i})$.
Let $\kappa>0$.
For $N\in [n,\infty)$,
we suppose $\ric^{N}_{f,M}\geq (N-1)\kappa$ and $H_{f,\bm} \geq (N-1)\lambda$.
Then $\lambda<0$ and $D\leq 2D_{\kappa,\lambda}$,
where $D_{\kappa,\lambda}:=\inf\, \{\,t>0\,|\, s'_{\kappa,\lambda}(t)=0\,\}$.
If $D=2D_{\kappa,\lambda}$,
then $(M,d_{M})$ is isometric to $([0,D]\times_{\kappa,\lambda} \bm_{1},d_{[0,D]\times_{\kappa,\lambda} \bm_{1}})$,
and for every $x\in \bm_{1}$
we have $f\circ \gamma_{x}=f(x)-(N-n)\log s_{\kappa,\lambda}$ on $[0,D]$.
\end{thm}
\begin{proof}
If $\lambda \geq 0$,
then Theorem \ref{thm:disconnected splitting1} implies that
$(M,d_{M})$ is isometric to $([0,D]\times \bm_{1},d_{[0,D]\times \bm_{1}})$,
and for every $x\in \bm_{1}$
the function $f\circ \gamma_{x}$ is constant on $[0,D]$.
This contradicts the positivity of $\kappa$,
and hence we have $\lambda<0$.

We prove that
if $D\geq 2D_{\kappa,\lambda}$,
then the metric space $(M,d_{M})$ is isometric to $([0,2D_{\kappa,\lambda}]\times_{\kappa,\lambda} \bm_{1},d_{[0,D_{\kappa,\lambda}]\times_{\kappa,\lambda} \bm_{1}})$,
and for every $x\in \bm_{1}$
we have $f\circ \gamma_{x}=f(x)-(N-n)\log s_{\kappa,\lambda}$ on $[0,2D_{\kappa,\lambda}]$.
Assume $D\geq 2D_{\kappa,\lambda}$.
By Lemma \ref{lem:disconnected lemma},
there exists a connected component $\bm_{2}$ of $\bm$ such that
$d_{M}(\bm_{1},\bm_{2})=D$.
For each $i=1,2$,
let $\rho_{\bm_{i}}:M\to \mathbb{R}$ be the distance function from $\bm_{i}$ defined as $\rho_{\bm_{i}}(p):=d_{M}(p,\bm_{i})$.
Put
\begin{equation*}
\Omega:=\{p\in \inte M \mid \rho_{\bm_{1}}(p)+\rho_{\bm_{2}}(p)=D\}.
\end{equation*}
The set $\Omega$ is a non-empty closed subset of $\inte M$.

We show that
$\Omega$ is open in $\inte M$.
Take $p\in \Omega$.
For each $i=1,2$,
we take a foot point $x_{p,i}\in \bm_{i}$ on $\bm_{i}$ of $p$ such that
$d_{M}(p,x_{p,i})=\rho_{\bm_{i}}(p)$.
From the triangle inequality,
we derive $d_{M}(x_{p,1},x_{p,2})=D$.
The normal minimal geodesic $\gamma:[0,D]\to M$ from $x_{p,1}$ to $x_{p,2}$ is orthogonal to $\bm$ at $x_{p,1}$ and at $x_{p,2}$.
Furthermore,
$\gamma|_{(0,D)}$ lies in $\inte M$ and passes through $p$.
There exists an open neighborhood $U$ of $p$ such that $\rho_{\bm_{i}}$ is smooth on $U$.
By using Lemma \ref{lem:Basic comparison},
for all $q\in U$,
we see
\begin{multline}\label{eq:disconnected}
       -\frac{\Delta_{f}\, \left(\rho_{\bm_{1}}+\rho_{\bm_{2}}\right)(q)}{N-1}
\leq \frac{s'_{\kappa,\lambda}(\rho_{\bm_{1}}(q))}{s_{\kappa,\lambda}(\rho_{\bm_{1}}(q))}+ \frac{s'_{\kappa,\lambda}(\rho_{\bm_{2}}(q))}{s_{\kappa,\lambda}(\rho_{\bm_{2}}(q))}\\
  =   \frac{s'_{\kappa,\lambda}(\rho_{\bm_{1}}(q)+\rho_{\bm_{2}}(q))-\lambda s_{\kappa,\lambda}(\rho_{\bm_{1}}(q)+\rho_{\bm_{2}}(q))}{s_{\kappa,\lambda}(\rho_{\bm_{1}}(q))  s_{\kappa,\lambda}(\rho_{\bm_{2}}(q))}.
\end{multline}
Since $\kappa>0$,
the function $s'_{\kappa,\lambda}/s_{\kappa,\lambda}$ is monotone decreasing on $(0,\const)$,
and satisfies $s'_{\kappa,\lambda}(2D_{\kappa,\lambda})/s_{\kappa,\lambda}(2D_{\kappa,\lambda})=\lambda$.
By the triangle inequality and the assumption $D\geq 2D_{\kappa,\lambda}$,
we have $\rho_{\bm_{1}}+\rho_{\bm_{2}}\geq 2D_{\kappa,\lambda}$ on $U$.
Therefore,
by (\ref{eq:disconnected}),
$-(\rho_{\bm_{1}}+\rho_{\bm_{2}})$ is $f$-subharmonic on $U$.
By Lemma \ref{lem:maximal principle},
$\Omega$ is open in $\inte M$.

The connectedness of $\inte M$ implies $\inte M=\Omega$.
For each $x\in \bm_{1}$,
choose an orthonormal basis $\{e_{x,i}\}_{i=1}^{n-1}$ of $T_{x}\bm$.
Let $\{Y_{x,i}\}_{i=1}^{n-1}$ be the $\bm$-Jacobi fields along $\gamma_{x}$
with initial conditions $Y_{x,i}(0)=e_{x,i}$ and $Y'_{x,i}(0)=-A_{u_{x}}e_{x,i}$.
For all $i$
we have $Y_{x,i}=s_{\kappa,\lambda}E_{x,i}$ on $[0,D]$,
where $\{Y_{x,i}\}_{i=1}^{n-1}$ are the parallel vector fields along $\gamma_{x}$ with initial condition $E_{x,i}(0)=e_{x,i}$.
Moreover,
$f\circ \gamma_{x}=f(x)-(N-n)\log s_{\kappa,\lambda}$ on $[0,D]$ (see Remark \ref{rem:Equality in Basic comparison}).
We see $D=2D_{\kappa,\lambda}$.
A map $\Phi:[0,2D_{\kappa,\lambda}]\times \bm_{1}\to M$ defined by $\Phi(t,x):=\gamma_{x}(t)$ is a Riemannian isometry with boundary from $[0,2D_{\kappa,\lambda}]\times_{\kappa,\lambda} \bm_{1}$ to $M$.
\end{proof}

\section{Eigenvalue rigidity}\label{sec:Eigenvalue rigidity}
Let $M$ be an $n$-dimensional,
connected complete Riemannian manifold with boundary with Riemannian metric $g$,
and let $f:M\to \mathbb{R}$ be a smooth function.
\subsection{Lower bounds}
We prove the inequalities (\ref{eq:eigenvalue rigidity}) in Theorem \ref{thm:eigenvalue rigidity}
and (\ref{eq:infinite eigenvalue rigidity}) in Theorem \ref{thm:infinite eigenvalue rigidity}.

Allegretto and Huang \cite{AH} have shown the following inequality of Picone type in a Euclidean setting (see Theorem 1.1 in \cite{AH}):
\begin{lem}\label{lem:Picone identity}
Let $\phi$ and $\psi$ be functions on $M$ that are smooth on a domain $U$ in $M$,
and satisfy $\phi>0$ and $\psi \geq 0$ on $U$.
Then for all $p\in (1,\infty)$ we have the following inequality on $U$:
\begin{equation}\label{eq:Picone identity}
\Vert \nabla \psi \Vert^{p}\geq \Vert \nabla \phi \Vert^{p-2} g\left(\nabla \left(  \psi^{p}\,  \phi^{1-p} \right),\nabla \phi \right).
\end{equation}
\end{lem}
\begin{proof}
For a fixed $p\in (1,\infty)$,
we put $q:=p(p-1)^{-1}$.
By the Young inequality,
we have
\begin{equation}\label{eq:Young}
\Vert \nabla \psi \Vert \left( \frac{\psi \Vert \nabla \phi \Vert}{\phi}   \right)^{p-1}
\leq
\frac{\Vert \nabla \psi \Vert^{p}}{p}+\frac{1}{q}\left( \frac{\psi \Vert \nabla \phi \Vert}{\phi} \right)^{p}
\end{equation}
on $U$.
By (\ref{eq:Young}),
and by the Cauchy-Schwarz inequality,
we have
\begin{align}\label{eq:Cauchy-Schwarz}
&\Vert \nabla \psi \Vert^{p} \geq p \left(  \psi \phi^{-1} \right)^{p-1} \Vert \nabla \psi \Vert    \Vert \nabla \phi \Vert^{p-1} -(p-1)\left(  \psi\phi^{-1} \right)^{p}\Vert \nabla \phi \Vert^{p} \\ \notag
                                        \geq\  & p \left(  \psi \phi^{-1} \right)^{p-1}  g(\nabla \phi,\nabla \psi)  \Vert \nabla \phi \Vert^{p-2} -(p-1)\left(  \psi\phi^{-1} \right)^{p}\Vert \nabla \phi \Vert^{p}\\ \notag
                                           =  \  & \Vert \nabla \phi \Vert^{p-2} g\left(\nabla \left(  \psi^{p}\,  \phi^{1-p} \right),\nabla \phi \right). \notag
\end{align}
This completes the proof.
\end{proof}
\begin{rem}\label{rem:the equality case in Picone identity}
In Lemma \ref{lem:Picone identity},
we assume that
the equality in (\ref{eq:Picone identity}) holds on $U$.
In this case,
the equalities in (\ref{eq:Cauchy-Schwarz}) also hold on $U$.
From the equality in the Young inequality,
and from that in the Cauchy-Schwarz inequality,
we deduce that
for some constant $c\neq 0$
we have $\phi \Vert \nabla \psi \Vert =\psi \Vert \nabla \phi \Vert$ and $\nabla \psi=c\nabla \phi$ on $U$;
in particular,
$\psi=c\, \phi$ on $U$.
\end{rem}
Now,
we prove the inequality $(\ref{eq:eigenvalue rigidity})$ in Theorem \ref{thm:eigenvalue rigidity}.
\begin{prop}\label{prop:inequality in eigenvalue rigidity}
Suppose that
$M$ is compact.
Let $p\in (1,\infty)$.
For $N\in [n,\infty)$,
we suppose $\ric^{N}_{f,M}\geq (N-1)\kappa$ and $H_{f,\bm} \geq (N-1)\lambda$.
For $D\in (0,\bar{C}_{\kappa,\lambda}]\setminus \{\infty\}$,
assume $\dm \leq D$.
Then we have $(\ref{eq:eigenvalue rigidity})$.
\end{prop}
\begin{proof}
Let $\phi_{p,N,\kappa,\lambda,D}:[0,D]\to \mathbb{R}$ be a function satisfying (\ref{eq:model space eigenvalue problem}) for $\mu=\mu_{p,N,\kappa,\lambda,D}$.
We may assume $\phi_{p,N,\kappa,\lambda,D}|_{(0,D]}>0$.
The equation (\ref{eq:model space eigenvalue problem}) is written in the form
\begin{align*}
\left(\vert \phi'(t)\vert^{p-2} \phi'(t) s^{N-1}_{\kappa,\lambda}(t) \right)'+\mu\, \vert \phi(t)\vert^{p-2}\phi(t) s^{N-1}_{\kappa,\lambda}(t)=0,\\
                                                              \phi(0)=0, \quad \phi'(D)=0.                           \notag
\end{align*}
Therefore,
it follows that $\phi'_{p,N,\kappa,\lambda,D}|_{[0,D)}>0$.
Put $\Phi:=\phi_{p,N,\kappa,\lambda,D}\circ \rho_{\bm}$.
Take a non-negative,
non-zero smooth function $\psi$ on $M$ whose support is compact and contained in $\inte M$.
By Lemma \ref{lem:Picone identity},
we have
\begin{equation}\label{eq:eigenvalue rigidity4}
\Vert \nabla \psi \Vert^{p}\geq \Vert \nabla \Phi \Vert^{p-2} g\left(\nabla \left(  \psi^{p}\,  \Phi^{1-p} \right),\nabla \Phi \right)
\end{equation}
on $\inte M \setminus \cut \bm$.
By using (\ref{eq:eigenvalue rigidity4}) and Proposition \ref{prop:global p-Laplacian comparison},
we have
\begin{align*}
&\quad \, \int_{M}\, \Vert \nabla \psi \Vert^{p}\, d\,m_{f}\geq \int_{M}\,\Vert \nabla \Phi \Vert^{p-2} g\left(\nabla \left(  \psi^{p}\,  \Phi^{1-p} \right),\nabla \Phi \right) \,d\,m_{f}\\
&                                                                      \geq \int_{M}\,  \left( \psi^{p}\,  \Phi^{1-p} \right)  \left(\left( -\left( \left(\phi' \right)^{p-1} \right)' -(N-1)\frac{s'_{\kappa,\lambda}}{s_{\kappa,\lambda}}\left(\phi' \right)^{p-1}\right)\circ \rho_{\bm}\right)   \,d\,m_{f}\\
&                                                                      =\mu_{p,N,\kappa,\lambda,D}\,\int_{M}\, \psi^{p} \, d\,m_{f}.
\end{align*}
We obtain $R_{f,p}(\psi)\geq \mu_{p,N,\kappa,\lambda,D}$.
This implies (\ref{eq:eigenvalue rigidity}).
\end{proof}
Next,
we prove the inequality $(\ref{eq:infinite eigenvalue rigidity})$ in Theorem \ref{thm:infinite eigenvalue rigidity}.
\begin{prop}\label{prop:inequality in infinite eigenvalue rigidity}
Suppose that
$M$ is compact.
Let $p\in (1,\infty)$.
Suppose $\ric^{\infty}_{f,M}\geq 0$ and $H_{f,\bm} \geq 0$.
For $D\in (0,\infty)$,
assume $\dm \leq D$.
Then we have $(\ref{eq:infinite eigenvalue rigidity})$.
\end{prop}
\begin{proof}
Let $\phi_{p,\infty,D}:[0,D]\to \mathbb{R}$ be a function satisfying (\ref{eq:infinite model space eigenvalue problem}) for $\mu=\mu_{p,\infty,D}$.
We may assume $\phi_{p,\infty,D}|_{(0,D]}>0$.
In this case,
we have
$\phi'_{p,\infty,D}|_{[0,D)}>0$.
Put $\Phi:=\phi_{p,\infty,D}\circ \rho_{\bm}$.
Take a non-negative,
non-zero smooth function $\psi$ on $M$ whose support is compact and contained in $\inte M$.
By Lemma \ref{lem:Picone identity},
we have
\begin{equation}\label{eq:infinite eigenvalue rigidity4}
\Vert \nabla \psi \Vert^{p}\geq \Vert \nabla \Phi \Vert^{p-2} g\left(\nabla \left(  \psi^{p}\,  \Phi^{1-p} \right),\nabla \Phi \right)
\end{equation}
on $\inte M \setminus \cut \bm$.
By using (\ref{eq:infinite eigenvalue rigidity4}) and Proposition \ref{prop:global infinite p-Laplacian comparison},
we have
\begin{multline*}
\int_{M}\, \Vert \nabla \psi \Vert^{p}\, d\,m_{f}\geq \int_{M}\,\Vert \nabla \Phi \Vert^{p-2} g\left(\nabla \left(  \psi^{p}\,  \Phi^{1-p} \right),\nabla \Phi \right) \,d\,m_{f}\\
                                                                      \geq \int_{M}\,  \left( \psi^{p}\,  \Phi^{1-p} \right)  \left(-\left( \left(\phi' \right)^{p-1} \right)' \circ \rho_{\bm}\right)   \,d\,m_{f}=\mu_{p,\infty,D}\,\int_{M}\, \psi^{p} \, d\,m_{f}.
\end{multline*}
We obtain $R_{f,p}(\psi)\geq \mu_{p,\infty,D}$.
This implies (\ref{eq:infinite eigenvalue rigidity}).
\end{proof}
\begin{rem}\label{rem:the equality case in eigenvalue rigidity}
In Proposition \ref{prop:inequality in eigenvalue rigidity} (resp. \ref{prop:inequality in infinite eigenvalue rigidity}),
we assume that
there exists a non-negative,
non-zero smooth function $\psi:M\to \mathbb{R}$ whose support is compact and contained in $\inte M$ such that $R_{f,p}(\psi)=\mu_{p,N,\kappa,\lambda,D}$ (resp. $R_{f,p}(\psi)=\mu_{p,\infty,D}$).
In this case,
the equality in (\ref{eq:eigenvalue rigidity4}) (resp. (\ref{eq:infinite eigenvalue rigidity4})) holds on $\inte M \setminus \cut \bm$.
Therefore,
for some constant $c\neq 0$
we have $\psi=c\, \Phi$ on $M$ (see Remark \ref{rem:the equality case in Picone identity}).
Furthermore,
the equality case in (\ref{eq:global p-Laplacian comparison}) (resp. \ref{eq:global infinite p-Laplacian comparison}) happens (see Remark \ref{rem:Equality case in global p-Laplacian comparison}).
\end{rem}
\subsection{Equality cases}
We prove Theorems \ref{thm:eigenvalue rigidity} and  \ref{thm:infinite eigenvalue rigidity}.

In the proofs,
we use the following fact:
\begin{prop}\label{prop:eigenfunction}
Suppose that
$M$ is compact.
Let $p\in (1,\infty)$.
Then there exists a non-negative,
non-zero function $\Psi$ in $W^{1,p}_{0}(M,m_{f})$ such that $R_{f,p}(\Psi)=\mu_{f,1,p}(M)$.
Moreover,
for some $\alpha \in (0,1)$
the function $\Psi$ is $C^{1,\alpha}$-H\"older continuous on $M$.
\end{prop}
Proposition \ref{prop:eigenfunction} is well-known in the standard case where $f=0$.
In the standard case,
the existence follows from the standard compactness argument,
and the regularity follows from the results by Tolksdorf in \cite{T}. 
The method of the proof also works in our weighted setting.

For $D\in (0,\infty)$,
we put $S_{D}(\bm):=\{\, q\in M \mid \rho_{\bm}(q)=D\,\}$.

Kasue has shown the following in the proof of Theorem 2.1 in \cite{K4}:
\begin{prop}[\cite{K4}]\label{prop:conclude rigidity}
Let $\kappa \in \mathbb{R}$ and $\lambda \in \mathbb{R}$.
Suppose that
$M$ is compact.
Assume that
for some $D\in (0,\bar{C}_{\kappa,\lambda})$
we have $\cut \bm=S_{D}(\bm)$.
For each $x\in \bm$,
choose an orthonormal basis $\{e_{x,i}\}_{i=1}^{n-1}$ of $T_{x}\bm$.
Let $\{Y_{x,i}\}^{n-1}_{i=1}$ be the $\bm$-Jacobi fields along $\gamma_{x}$
with initial conditions $Y_{x,i}(0)=e_{x,i}$ and $Y_{x,i}'(0)=-A_{u_{x}}e_{x,i}$.
Assume further that
for all $x\in \bm$ and $i$
we have $Y_{x,i}=s_{\kappa,\lambda}\, E_{x,i}$ on $[0,D]$,
where $\{E_{x,i}\}^{n-1}_{i=1}$ are the parallel vector fields along $\gamma_{x}$ with initial condition $E_{x,i}(0)=e_{x,i}$.
Then $\kappa$ and $\lambda$ satisfy the model-condition,
$(M,d_{M})$ is a $(\kappa,\lambda)$-equational model space,
and $D=D_{\kappa,\lambda}(M)$.
\end{prop}
Now,
we prove Theorem \ref{thm:eigenvalue rigidity}.
\begin{proof}[Proof of Theorem \ref{thm:eigenvalue rigidity}]
Suppose that
$M$ is compact.
Let $p\in (1,\infty)$.
For $N\in [n,\infty)$,
we suppose $\ric^{N}_{f,M}\geq (N-1)\kappa$ and $H_{f,\bm} \geq (N-1)\lambda$.
For $D\in (0,\bar{C}_{\kappa,\lambda}]\setminus \{\infty\}$,
assume $\dm \leq D$.
By Proposition \ref{prop:inequality in eigenvalue rigidity},
we have (\ref{eq:eigenvalue rigidity}).

Assume that
the equality in (\ref{eq:eigenvalue rigidity}) holds.
By Proposition \ref{prop:eigenfunction},
there exists a non-negative, 
non-zero function $\Psi$ in $W^{1,p}_{0}(M,m_{f})$ such that
$R_{f,p}(\Psi)=\mu_{p,N,\kappa,\lambda,D}$ and $\Psi$ is $C^{1,\alpha}$-H\"older continuous on $M$.
Put $\Phi:=\phi_{p,N,\kappa,\lambda,D}\circ \rho_{\bm}$.
Then $\Phi$ coincides with a constant multiplication of $\Psi$ on $M$ (see Remark \ref{rem:the equality case in eigenvalue rigidity});
in particular,
$\Phi$ is also $C^{1,\alpha}$-H\"older continuous.

For each $x\in \bm$,
choose an orthonormal basis $\{e_{x,i}\}_{i=1}^{n-1}$ of $T_{x}\bm$.
Let $\{Y_{x,i}\}^{n-1}_{i=1}$ be the $\bm$-Jacobi fields along $\gamma_{x}$
with initial conditions $Y_{x,i}(0)=e_{x,i}$ and $Y_{x,i}'(0)=-A_{u_{x}}e_{x,i}$.
For all $i$
we see $Y_{x,i}=s_{\kappa,\lambda}\, E_{x,i}$ on $[0,\tau(x)]$,
where $\{E_{x,i}\}^{n-1}_{i=1}$ are the parallel vector fields along $\gamma_{x}$ with initial condition $E_{x,i}(0)=e_{x,i}$.
Moreover,
$f\circ \gamma_{x}=f(x)-(N-n) \log s_{\kappa,\lambda}$ on $[0,\tau(x)]$ (see Remarks \ref{rem:Equality case in global p-Laplacian comparison} and \ref{rem:the equality case in eigenvalue rigidity}).

Let $D=\bar{C}_{\kappa,\lambda}$.
Since $D$ is finite,
$\kappa$ and $\lambda$ satisfy the ball-condition and $D=\const$.
There exists $p_{0}\in M$ such that $\rho_{\bm}(p_{0})=D(M,\bm)$.
Note that
$p_{0}$ belongs to $\cut \bm$.
Now,
we prove $\rho_{\bm}(p_{0})=\const$.
We assume $\rho_{\bm}(p_{0})<\const$.
Let $x_{0}$ be a foot point on $\bm$ of $p_{0}$.
From the property of Jacobi fields,
$p_{0}$ is not the first conjugate point of $\bm$ along $\gamma_{x_{0}}$.
Hence,
$\rho_{\bm}$ is not differentiable at $p_{0}$.
Since $\Phi$ is $C^{1,\alpha}$-H\"older continuous,
we see $\phi'_{p,N,\kappa,\lambda,D}(\rho_{\bm}(p_{0}))=0$.
From $\phi'_{p,N,\kappa,\lambda,D}|_{[0,D)}>0$,
we deduce $\rho_{\bm}(p_{0})=D$.
This contradicts $D=\const$.
Therefore,
$\rho_{\bm}(p_{0})=\const$.
By Theorem \ref{thm:Ball rigidity},
$(M,d_{M})$ is isometric to $(\ball,d_{\ball})$ and $N=n$.

Let $D\in (0,\bar{C}_{\kappa,\lambda})$.
We prove $\cut \bm=S_{D}(\bm)$.
Since $\dm \leq D$,
we see $S_{D}(\bm)\subset \cut \bm$.
We show the opposite.
Take $p_{0}\in \cut \bm$.
By the property of Jacobi fields,
$\rho_{\bm}$ is not differentiable at $p_{0}$.
The regularity of $\Phi$ implies $\phi'_{p,N,\kappa,\lambda,D}(\rho_{\bm}(p_{0}))=0$;
in particular,
$\rho_{\bm}(p_{0})=D$.
We have $\cut \bm=S_{D}(\bm)$.
By Proposition \ref{prop:conclude rigidity},
$\kappa$ and $\lambda$ satisfy the model-condition,
$(M,d_{M})$ is a $(\kappa,\lambda)$-equational model space,
and $D=D_{\kappa,\lambda}(M)$.
From $\tau=D_{\kappa,\lambda}(M)$ on $\bm$,
it follows that $f\circ \gamma_{x}=f(x)-(N-n) \log s_{\kappa,\lambda}$ on $[0,D_{\kappa,\lambda}(M)]$ for all $x\in \bm$.
We complete the proof of Theorem \ref{thm:eigenvalue rigidity}.
\end{proof}
\begin{rem}\label{rem:method of the proof}
In \cite{K4},
the proof of Theorem \ref{thm:eigenvalue rigidity} in the standard case where $f=0,N=n$ and $p=2$
relies on the approximation theorem obtained by Greene and Wu in \cite{GW}.
It seems that
the approximation theorem in \cite{GW} does not work in our non-linear case of $p\neq 2$.
\end{rem}
Next,
we prove Theorem \ref{thm:infinite eigenvalue rigidity}.
\begin{proof}[Proof of Theorem \ref{thm:infinite eigenvalue rigidity}]
Suppose that
$M$ is compact.
Let $p\in (1,\infty)$.
Suppose $\ric^{\infty}_{f,M}\geq 0$ and $H_{f,\bm} \geq 0$.
For $D\in (0,\infty)$,
we assume $\dm \leq D$.
By Proposition \ref{prop:inequality in infinite eigenvalue rigidity},
we have (\ref{eq:infinite eigenvalue rigidity}).

Assume that
the equality in (\ref{eq:infinite eigenvalue rigidity}) holds.
By Proposition \ref{prop:eigenfunction},
there exists a non-negative,
non-zero function $\Psi$ in $W^{1,p}_{0}(M,m_{f})$ such that
$R_{f,p}(\Psi)=\mu_{p,\infty,D}$ and $\Psi$ is $C^{1,\alpha}$-H\"older continuous on $M$.
Put $\Phi:=\phi_{p,\infty,D}\circ \rho_{\bm}$.
Then $\Phi$ coincides with a constant multiplication of $\Psi$ on $M$ (see Remark \ref{rem:the equality case in eigenvalue rigidity});
in particular,
$\Phi$ is also $C^{1,\alpha}$-H\"older continuous.

For each $x\in \bm$,
choose an orthonormal basis $\{e_{x,i}\}_{i=1}^{n-1}$ of $T_{x}\bm$.
Let $\{Y_{x,i}\}^{n-1}_{i=1}$ be the $\bm$-Jacobi fields along $\gamma_{x}$
with initial conditions $Y_{x,i}(0)=e_{x,i}$ and $Y_{x,i}'(0)=-A_{u_{x}}e_{x,i}$.
For all $i$
we have $Y_{x,i}=E_{x,i}$ on $[0,\tau(x)]$,
where $\{E_{x,i}\}^{n-1}_{i=1}$ are the parallel vector fields along $\gamma_{x}$ with initial condition $E_{x,i}(0)=e_{x,i}$ (see Remarks \ref{rem:Equality case in global p-Laplacian comparison} and \ref{rem:the equality case in eigenvalue rigidity}).

We prove $\cut \bm=S_{D}(\bm)$.
Since $\dm \leq D$,
it holds that $S_{D}(\bm)\subset \cut \bm$.
We show the opposite.
Take $p_{0}\in \cut \bm$.
By the property of Jacobi fields,
$\rho_{\bm}$ is not differentiable at $p_{0}$.
By the regularity of $\Phi$,
we see $\phi'_{p,\infty,D}(\rho_{\bm}(p_{0}))=0$;
in particular,
$\rho_{\bm}(p_{0})=D$.
It follows that $\cut \bm=S_{D}(\bm)$;
in particular,
$\dm=D$.
By Proposition \ref{prop:conclude rigidity},
we complete the proof of Theorem \ref{thm:infinite eigenvalue rigidity}.
\end{proof}
\subsection{Explicit lower bounds}\label{sec:Concrete large lower bounds}
For $N\in [2,\infty)$ and $D\in (0,\infty)$,
we see $\mu_{2,N,0,0,D}=\mu_{2,\infty,D}=\pi^{2}(2D)^{-2}$.

By Theorems \ref{thm:eigenvalue rigidity} and \ref{thm:infinite eigenvalue rigidity},
we have the following:
\begin{cor}
Let $M$ be an $n$-dimensional,
connected complete Riemannian manifold with boundary,
and let $f:M\to \mathbb{R}$ be a smooth function.
Suppose that
$M$ is compact.
For $N\in [n,\infty]$,
we suppose $\ric^{N}_{f,M}\geq 0$ and $H_{f,\bm} \geq 0$.
For $D\in (0,\infty)$,
we assume $\dm \leq D$.
Then we have
\begin{equation}\label{eq:Li-Yau estimate}
\mu_{f,1,2}(M)\geq \frac{{\pi}^{2}}{4D^{2}}.
\end{equation}
If the equality in $(\ref{eq:Li-Yau estimate})$ holds,
then $\dm=D$,
and $(M,d_{M})$ is a $(0,0)$-equational model space.
Moreover,
if $N\in [n,\infty)$,
then for every $x\in \bm$
the function $f\circ \gamma_{x}$ is constant on $[0,D]$.
\end{cor}
Li and Yau \cite{LY} have obtained (\ref{eq:Li-Yau estimate}) when $f=0$ and $N=n$.

Kasue \cite{K4} has proved the following (see Lemma 1.3 in \cite{K4}):
\begin{lem}[\cite{K4}]\label{lem:Kasue computable estimate}
For all $N\in [2,\infty),\,\kappa,\lambda\in \mathbb{R}$ and $D\in (0,\bar{C}_{\kappa,\lambda}]\setminus \{\infty\}$,
we  have
\begin{equation*}
\mu_{2,N,\kappa,\lambda,D}> \left(4 \max_{t\in [0,D]}\, \int^{D}_{t}\, s^{N-1}_{\kappa,\lambda}(s) \, ds\, \int^{t}_{0}\, s^{1-N}_{\kappa,\lambda}(s)\,ds \right)^{-1}.
\end{equation*}
\end{lem}
In the case of $p=2$,
by Theorem \ref{thm:eigenvalue rigidity} and Lemma \ref{lem:Kasue computable estimate}
we have:
\begin{cor}
Let $M$ be an $n$-dimensional,
connected complete Riemannian manifold with boundary,
and let $f:M\to \mathbb{R}$ be a smooth function.
Suppose that
$M$ is compact.
For $N\in [n,\infty)$,
we suppose $\ric^{N}_{f,M}\geq (N-1)\kappa$ and $H_{f,\bm} \geq (N-1)\lambda$.
For $D\in (0,\bar{C}_{\kappa,\lambda}]\setminus \{\infty\}$,
we assume $\dm \leq D$.
Then we have
\begin{equation*}
\mu_{f,1,2}(M)>\left(4 \max_{t\in [0,D]}\, \int^{D}_{t}\, s^{N-1}_{\kappa,\lambda}(s) \, ds\, \int^{t}_{0}\, s^{1-N}_{\kappa,\lambda}(s)\,ds \right)^{-1}.
\end{equation*}
\end{cor}

\section{First eigenvalue estimates}\label{sec:First eigenvalue estimates}
Let $M$ be an $n$-dimensional,
connected complete Riemannian manifold with boundary with Riemannian metric $g$,
and let $f:M\to \mathbb{R}$ be a smooth function.
\subsection{Area estimates}\label{sec:Area estimates}
Let $\Omega$ be a relatively compact domain in $M$ such that
$\partial \Omega$ is a smooth hypersurface in $M$ satisfying $\partial \Omega \cap \bm=\emptyset$.
For the canonical Riemannian volume measure $\vol_{\partial \Omega}$ on $\partial \Omega$,
let $m_{f,\partial \Omega}:=e^{-f|_{\partial \Omega}}\, \vol_{\partial \Omega}$.
Put
\begin{equation}\label{eq:diameter of Omega}
\delta_{1}(\Omega):=\inf_{p\in \Omega}\, \rho_{\bm}(p),\quad \delta_{2}(\Omega):=\sup_{p\in \Omega} \,\rho_{\bm}(p).
\end{equation}

Kasue \cite{K5} has proved the following when $f=0$ and $N=n$.
\begin{prop}\label{prop:Kasue volume estimate}
For $N\in [n,\infty)$,
we suppose $\ric^{N}_{f,M}\geq (N-1)\kappa$ and $H_{f,\bm} \geq (N-1)\lambda$.
Let $\Omega$ be a relatively compact domain in $M$ such that
$\partial \Omega$ is a smooth hypersurface in $M$ satisfying $\partial \Omega \cap \bm=\emptyset$.
Then
\begin{equation}\label{eq:Kasue volume estimate}
m_{f}( \Omega) \leq m_{f,\partial \Omega}\, (\partial \Omega) \,\sup_{t\in (\delta_{1}(\Omega),\delta_{2}(\Omega))}\, \frac{\int^{\delta_{2}(\Omega)}_{t}\,  s^{N-1}_{\kappa,\lambda}(s)\, ds}{s^{N-1}_{\kappa,\lambda}(t)},
\end{equation}
where $\delta_{1}(\Omega)$ and $\delta_{2}(\Omega)$ are the values defined as $(\ref{eq:diameter of Omega})$.
\end{prop}
\begin{proof}
Define a function $\phi:[\delta_{1}(\Omega),\delta_{2}(\Omega)]\to \mathbb{R}$ by
\begin{equation*}
\phi(t):=\int^{t}_{\delta_{1}(\Omega)}\, \frac{\int^{\delta_{2}(\Omega)}_{s}\,s^{N-1}_{\kappa,\lambda}(u) \,du}{s^{N-1}_{\kappa,\lambda}(s)} \,ds,
\end{equation*}
and put $\Phi:=\phi \circ \rho_{\bm}$.
By Lemma \ref{lem:p-Laplacian comparison},
on $\inte M\setminus \cut \bm$
\begin{equation}\label{eq:Kasue volume estimate Laplacian comparison}
\Delta_{f,2}\,\Phi \geq 1.
\end{equation}

By Lemma \ref{lem:avoiding the cut locus2},
there exists a sequence $\{\Omega_{k}\}_{k\in \mathbb{N}}$ of compact subsets of $\bar{\Omega}$
satisfying that
for every $k$,
the set $\partial \Omega_{k}$ is a smooth hypersurface in $M$ except for a null set in $(\partial \Omega,m_{f,\partial \Omega})$,
and satisfying the following:
(1) for all $k_{1},k_{2}\in \mathbb{N}$ with $k_{1}<k_{2}$,
we have $\Omega_{k_{1}}\subset \Omega_{k_{2}}$;
(2) $\bar{\Omega} \setminus \cut \bm=\bigcup_{k\in \mathbb{N}}\,\Omega_{k}$:
(3) for every $k\in \mathbb{N}$,
     and for almost every point $p \in \partial \Omega_{k}\cap \partial \Omega$ in $(\partial \Omega,m_{f,\partial \Omega})$,
there exists the unit outer normal vector for $\Omega_{k}$ at $p$
that coincides with the unit outer normal vector on $\partial \Omega$ for $\Omega$ at $p$;
(4) for every $k\in \mathbb{N}$,
on $\partial \Omega_{k}\setminus \partial \Omega$,
there exists the unit outer normal vector field $\nu_{k}$ for $\Omega_{k}$ such that $g(\nu_{k},\nabla \rho_{\bm})\geq 0$.

For the canonical Riemannian volume measure $\vol_{k}$ on $\partial \Omega_{k}\setminus \partial \Omega$,
put $m_{f,k}:=e^{-f|_{\partial \Omega_{k}\setminus \partial \Omega}}\,\vol_{k}$.
Let $\nu_{\partial \Omega}$ be the unit outer normal vector on $\partial \Omega$ for $\Omega$.
By integrating the both sides of (\ref{eq:Kasue volume estimate Laplacian comparison}) on $\Omega_{k}$,
and by the Green formula,
we have
\begin{multline*}
m_{f}\left(\Omega_{k}\right)
\leq \int_{\Omega_{k}} \, \Delta_{f,2}\,\Phi \,d\,m_{f}\\
= -\int_{\partial \Omega_{k}\setminus \partial \Omega} g(\nu_{k},\nabla \Phi)  \,d\,m_{f,k}-\int_{\partial \Omega_{k} \cap \partial \Omega} g(\nu_{\partial \Omega},\nabla \Phi)  \,d\,m_{f,\partial \Omega}.
\end{multline*}
Since $g(\nu_{k},\nabla \Phi)\geq 0$ on $\partial \Omega_{k}\setminus \partial \Omega$,
we have
\begin{equation*}
m_{f}\left(\Omega_{k}\right) \leq -\int_{\partial \Omega_{k}\cap \partial \Omega} g(\nu_{\partial \Omega},\nabla \Phi)  \,d\,m_{f,\partial \Omega}.
\end{equation*}
Therefore,
from the Cauchy-Schwarz inequality,
we derive
\begin{align*}
        m_{f}\left(\Omega_{k}\right)
&\leq \int_{\partial \Omega_{k}\cap \partial \Omega} \left(\phi'\circ \rho_{\bm} \right) \vert g(\nu_{\partial \Omega},\nabla \rho_{\bm}) \vert  \,d\,m_{f,\partial \Omega}\\
&\leq m_{f,\partial \Omega}\, (\partial \Omega) \,\sup_{t\in (\delta_{1}(\Omega),\delta_{2}(\Omega))}\,\phi'(t).
\end{align*}
By letting $k \to \infty$,
we have (\ref{eq:Kasue volume estimate}).
\end{proof}
\begin{rem}
In \cite{K5},
the key points of the proof of Proposition \ref{prop:Kasue volume estimate}
in the standard case where $f=0$ and $N=n$
are to use the comparison theorem concerning a generalized Laplacian of $\rho_{\bm}$ proved in \cite{K2},
and to apply the approximation theorem in \cite{GW} to $\rho_{\bm}$.
We see that
similar theorems also hold in our weighted case.
From this point of view,
Proposition \ref{prop:Kasue volume estimate} can be proved in the same way as that in \cite{K5}.
\end{rem}
In the case of $N=\infty$,
we have the following:
\begin{prop}\label{prop:infinite Kasue volume estimate}
Suppose $\ric^{\infty}_{f,M}\geq 0$ and $H_{f,\bm} \geq 0$.
Let $\Omega$ be a relatively compact domain in $M$ such that
$\partial \Omega$ is a smooth hypersurface in $M$ satisfying $\partial \Omega \cap \bm=\emptyset$.
Then
\begin{equation}\label{eq:infinite Kasue volume estimate}
m_{f}(\Omega) \leq m_{f,\partial \Omega}\, (\partial \Omega) \,\left(\delta_{2}(\Omega)-\delta_{1}(\Omega)   \right),
\end{equation}
where $\delta_{1}(\Omega)$ and $\delta_{2}(\Omega)$ are the values defined as $(\ref{eq:diameter of Omega})$.
\end{prop}
\begin{proof}
Define a function $\phi:[\delta_{1}(\Omega),\delta_{2}(\Omega)]\to \mathbb{R}$ by
\begin{equation*}
\phi(t):=-\frac{t^{2}}{2}+\delta_{2}(\Omega)t-\delta_{1}(\Omega)\delta_{2}(\Omega)+\frac{\delta_{1}(\Omega)^{2}}{2},
\end{equation*}
and put $\Phi:=\phi \circ \rho_{\bm}$.
By Lemma \ref{lem:infinite p-Laplacian comparison},
on $\inte M\setminus \cut \bm$
\begin{equation}\label{eq:infinite Kasue volume estimate Laplacian comparison}
\Delta_{f,2}\,\Phi \geq 1.
\end{equation}

By Lemma \ref{lem:avoiding the cut locus2},
there exists a sequence $\{\Omega_{k}\}_{k\in \mathbb{N}}$ of compact subsets of $\bar{\Omega}$ satisfying that
for every $k$,
the set $\partial \Omega_{k}$ is a smooth hypersurface in $M$ except for a null set in $(\partial \Omega,m_{f,\partial \Omega})$,
satisfying the following:
(1) for all $k_{1},k_{2}\in \mathbb{N}$ with $k_{1}<k_{2}$,
we have $\Omega_{k_{1}}\subset \Omega_{k_{2}}$;
(2) $\bar{\Omega} \setminus \cut \bm=\bigcup_{k\in \mathbb{N}}\,\Omega_{k}$;
(3) for every $k\in \mathbb{N}$,
and for almost every point $p \in \partial \Omega_{k}\cap \partial \Omega$ in $(\partial \Omega,m_{f,\partial \Omega})$,
there exists the unit outer normal vector for $\Omega_{k}$ at $p$
that coincides with the unit outer normal vector on $\partial \Omega$ for $\Omega$ at $p$;
(4) for every $k\in \mathbb{N}$,
on $\partial \Omega_{k}\setminus \partial \Omega$,
there exists the unit outer normal vector field $\nu_{k}$ for $\Omega_{k}$ such that $g(\nu_{k},\nabla \rho_{\bm})\geq 0$.

For the canonical Riemannian volume measure $\vol_{k}$ on $\partial \Omega_{k}\setminus \partial \Omega$,
put $m_{f,k}:=e^{-f|_{\partial \Omega_{k}\setminus \partial \Omega}}\,\vol_{k}$.
Let $\nu_{\partial \Omega}$ be the unit outer normal vector on $\partial \Omega$ for $\Omega$.
By integrating the both sides of (\ref{eq:infinite Kasue volume estimate Laplacian comparison}) on $\Omega_{k}$,
and by the Green formula,
we have
\begin{multline*}
m_{f}\left(\Omega_{k}\right)
\leq \int_{\Omega_{k}} \, \Delta_{f,2}\,\Phi \,d\,m_{f}\\
= -\int_{\partial \Omega_{k}\setminus \partial \Omega} g(\nu_{k},\nabla \Phi)  \,d\,m_{f,k}-\int_{\partial \Omega_{k}\cap \partial \Omega} g(\nu_{\partial \Omega},\nabla \Phi)  \,d\,m_{f,\partial \Omega}.
\end{multline*}
Since $g(\nu_{k},\nabla \Phi)\geq 0$ on $\partial \Omega_{k}\setminus \partial \Omega$,
we have
\begin{equation*}
m_{f}\left(\Omega_{k}\right) \leq -\int_{\partial \Omega_{k}\cap \partial \Omega} g(\nu_{\partial \Omega},\nabla \Phi)  \,d\,m_{f,\partial \Omega}.
\end{equation*}
By the Cauchy-Schwarz inequality,
\begin{align*}
        m_{f}\left(\Omega_{k}\right)
&\leq \int_{\partial \Omega_{k}\cap \partial \Omega} \left(\delta_{2}(\Omega)-\rho_{\bm}\right) \vert g(\nu_{\partial \Omega},\nabla \rho_{\bm}) \vert  \,d\,m_{f,\partial \Omega}\\
&\leq m_{f,\partial \Omega}\, (\partial \Omega) \,\left(\delta_{2}(\Omega)-\delta_{1}(\Omega)   \right).
\end{align*}
Letting $k \to \infty$,
we obtain (\ref{eq:infinite Kasue volume estimate}).
\end{proof}
\subsection{Eigenvalue estimates}
Let $\alpha \in (0,\infty)$.
The \textit{$f$-Dirichlet $\alpha$-isoperimetric constant} $ID_{\alpha}(M,m_{f})$ of $M$ is defined as
\begin{equation*}
ID_{\alpha}(M,m_{f}):=\inf_{\Omega}\,  \frac{m_{f,\partial \Omega}(\partial \Omega)}{\left( m_{f}(\Omega) \right)^{1/\alpha}},
\end{equation*}
where the infimum is taken over all relatively compact domains $\Omega$ in $M$ such that
$\partial \Omega$ are smooth hypersurfaces in $M$ satisfying $\partial \Omega \cap \partial M=\emptyset$.
The \textit{$f$-Dirichlet $\alpha$-Sobolev constant} $SD_{\alpha}(M,m_{f})$ of $M$ is defined as
\begin{equation*}
SD_{\alpha}(M,m_{f}):=\inf_{\phi \in W^{1,1}_{0}(M,m_{f})\setminus \{0\}}\, \frac{\int_{M}\,  \Vert \nabla \phi \Vert \,d\,m_{f}}{\left(\int_{M}\,  \vert \phi \vert^{\alpha} \,d\,m_{f} \right)^{1/\alpha}},
\end{equation*}
where the infimum is taken over all non-zero functions $\phi$ in $W^{1,1}_{0}(M,m_{f})$.

The following relationship between 
the isoperimetric constant and the Sobolev constant has been 
formally established by Federer and Fleming in \cite{FF} (see e.g., \cite{Ch}, \cite{Li}),
and later used by Cheeger in \cite{Che1} for the estimate of the first Dirichlet eigenvalue of the Laplacian.
\begin{prop}[\cite{FF}]\label{thm:isoperimetric and Sobolev}
For all $\alpha\in (0,\infty)$
we have 
\begin{equation*}
ID_{\alpha}(M,m_{f})=SD_{\alpha}(M,m_{f}).
\end{equation*}
\end{prop}
A proof of Proposition \ref{thm:isoperimetric and Sobolev} has been given in \cite{Li} in the case of $f=0$ (see Theorem 9.5 in \cite{Li}).
The method of the proof also works in our weighted setting.

For $N\in [2,\infty)$,
$\kappa,\lambda\in \mathbb{R}$,
and $D\in (0,\bar{C}_{\kappa,\lambda}]$,
let $C(N,\kappa,\lambda,D)$ be a positive constant defined by
\begin{equation}\label{eq:constant}
C(N,\kappa,\lambda,D):=\sup_{t\in [0,D)}\, \frac{\int^{D}_{t}\,  s^{N-1}_{\kappa,\lambda}(s)\, ds}{s^{N-1}_{\kappa,\lambda}(t)}.
\end{equation}
Notice that 
$C(N,\kappa,\lambda,\infty)$ is finite if and only if $\kappa<0$ and $\lambda=\sqrt{\vert \kappa \vert}$;
in this case,
we have
$C(N,\kappa,\lambda,D)=\left((N-1)\lambda \right)^{-1}\,\left(1-e^{-(N-1)\lambda\, D} \right)$;
in particular,
$(2\,C(N,\kappa,\lambda,\infty))^{-2}=\left((N-1)\lambda/2 \right)^{2}$.

By using Proposition \ref{prop:Kasue volume estimate},
we obtain the following:
\begin{thm}\label{thm:p-Laplacian1}
Let $M$ be an $n$-dimensional,
connected complete Riemannian manifold with boundary,
and let $f:M\to \mathbb{R}$ be a smooth function.
Suppose that
$\bm$ is compact.
Let $p\in (1,\infty)$.
For $N\in [n,\infty)$,
we suppose $\ric^{N}_{f,M}\geq (N-1)\kappa$ and $H_{f,\bm} \geq (N-1)\lambda$.
For $D\in(0,\bar{C}_{\kappa,\lambda}]$,
we assume $\dm \leq D$.
Then we have
\begin{equation}\label{eq:p-Laplacian1}
\mu_{f,1,p}(M)\geq (\,p\,C(N,\kappa,\lambda,D)\,)^{-p},
\end{equation}
where $C(N,\kappa,\lambda,D)$ is the constant defined as $(\ref{eq:constant})$.
\end{thm}
\begin{proof}
Let $\Omega$ be a relatively compact domain in $M$ such that
$\partial \Omega$ is a smooth hypersurface in $M$ satisfying $\partial \Omega \cap \bm=\emptyset$.
By Proposition \ref{prop:Kasue volume estimate},
we have
\begin{equation*}\label{eq:Kasue}
m_{f}(\Omega) \leq m_{f,\partial \Omega}( \partial \Omega)\,C(N,\kappa,\lambda,D).
\end{equation*}
By Proposition \ref{thm:isoperimetric and Sobolev},
we have $ID_{1}(M,m_{f})=SD_{1}(M,m_{f})$.
We obtain $SD_{1}(M,m_{f})\geq C(N,\kappa,\lambda,D)^{-1}$.
Therefore,
for all $\phi \in W^{1,1}_{0}(M,m_{f})$
\begin{equation}\label{eq:Poincare1}
\int_{M}\, \vert \phi \vert \,d\,m_{f} \leq C(N,\kappa,\lambda,D) \int_{M}\,\Vert \nabla \phi \Vert \,d\,m_{f}.
\end{equation}

Let $\psi$ be a non-zero function in $W^{1,p}_{0}(M,m_{f})$.
Put $q:=p\,(1-p)^{-1}$.
In (\ref{eq:Poincare1}),
by replacing $\phi$ with $\vert \psi \vert^{p}$, 
and by the H\"older inequality,
we see
\begin{multline*}
\int_{M}\, \vert \psi \vert^{p} \,d\,m_{f} \leq p\,C(N,\kappa,\lambda,D)\, \int_{M}\, \vert \psi \vert^{p-1} \,  \Vert \nabla \psi\Vert \,d\,m_{f}\\
                                \leq p\, C(N,\kappa,\lambda,D)\, \left(\int_{M}\, \vert \psi \vert^{p}\,d\,m_{f}\right)^{1/q} \left(\int_{M}\, \Vert \nabla \psi\Vert^{p}\,d\, m_{f}\right)^{1/p}.\\
\end{multline*}
Considering the Rayleigh quotient $R_{f,p}(\psi)$,
we obtain (\ref{eq:p-Laplacian1}).
\end{proof}
In the case of $N=\infty$,
we have the following:
\begin{thm}\label{thm:infinite p-Laplacian1}
Let $M$ be a connected complete Riemannian manifold with boundary,
and let $f:M\to \mathbb{R}$ be a smooth function.
Suppose that
$\bm$ is compact.
Let $p\in (1,\infty)$.
Suppose $\ric^{\infty}_{f,M}\geq 0$ and $H_{f,\bm} \geq 0$.
For $D\in(0,\infty]$,
we assume $D(M,\bm)\leq D$.
Then we have
\begin{equation}\label{eq:infinite p-Laplacian1}
\mu_{f,1,p}(M)\geq (\,pD\,)^{-p}.
\end{equation}
\end{thm}
\begin{proof}
Let $\Omega$ be a relatively compact domain in $M$ such that
$\partial \Omega$ is a smooth hypersurface in $M$ satisfying $\partial \Omega \cap \bm=\emptyset$.
Proposition \ref{prop:infinite Kasue volume estimate} implies
$m_{f}(\Omega) \leq m_{f,\partial \Omega}( \partial \Omega)\,D$.
From Proposition \ref{thm:isoperimetric and Sobolev},
we derive $SD_{1}(M,m_{f})\geq D^{-1}$.
Therefore,
for all $\phi \in W^{1,1}_{0}(M,m_{f})$
\begin{equation}\label{eq:infinite Poincare1}
\int_{M}\, \vert \phi \vert \,d\,m_{f} \leq D \int_{M}\,\Vert \nabla \phi \Vert \,d\,m_{f}.
\end{equation}

Take a non-zero function $\psi$ in $W^{1,p}_{0}(M,m_{f})$.
Put $q:=p\,(1-p)^{-1}$.
In (\ref{eq:infinite Poincare1}),
by replacing $\phi$ with $\vert \psi \vert^{p}$, 
and by the H\"older inequality,
we see
\begin{equation*}
\int_{M}\, \vert \psi \vert^{p} \,d\,m_{f} \leq p\, D\, \left(\int_{M}\, \vert \psi \vert^{p}\,d\,m_{f}\right)^{1/q} \left(\int_{M}\, \Vert \nabla \psi\Vert^{p}\,d\, m_{f}\right)^{1/p}.
\end{equation*}
Considering the Rayleigh quotient $R_{f,p}(\psi)$,
we obtain (\ref{eq:infinite p-Laplacian1}).
\end{proof}
Now,
we prove Theorem \ref{thm:spectrum rigidity}.
\begin{proof}[Proof of Theorem \ref{thm:spectrum rigidity}]
Suppose that
$\bm$ is compact.
Let $p\in (1,\infty)$.
Let $\kappa<0$ and $\lambda:=\sqrt{\vert \kappa \vert}$.
For $N\in [n,\infty)$,
we suppose $\ric^{N}_{f,M}\geq (N-1)\kappa$ and $H_{f,\bm} \geq (N-1)\lambda$.
We have 
\begin{equation*}
C(N,\kappa,\lambda,D)=\left((N-1)\lambda \right)^{-1}\,\left(1-e^{-(N-1)\lambda\, D} \right).
\end{equation*}
The right hand side is monotone increasing as $D\to \infty$.
From Theorem \ref{thm:p-Laplacian1},
we derive (\ref{eq:noncompact estimate}).

Assume that
the equality in (\ref{eq:noncompact estimate}) holds.
By Theorem \ref{thm:p-Laplacian1},
we have $D=\infty$.
Since $\bm$ is compact,
$M$ is non-compact.
By Corollary \ref{cor:Kasue splitting},
$(M,d_{M})$ is isometric to $([0,\infty)\times_{\kappa,\lambda}\bm, d_{\kappa,\lambda})$,
and for all $x\in \bm$ and $t\in [0,\infty)$
we have $(f \circ \gamma_{x})(t)=f(x)+(N-n)\lambda t$.
This completes the proof of Theorem \ref{thm:spectrum rigidity}.
\end{proof}

\end{document}